\documentclass[11pt, reqno]{amsart}
 \usepackage[english]{babel}
 \usepackage[utf8]{inputenc} 
\usepackage[T1]{fontenc} 
   \usepackage{textcomp}  
\usepackage[margin=1.3 in]{geometry}
 \usepackage{mathtools,mathrsfs}
 
\mathtoolsset{showonlyrefs}
\usepackage{graphicx,enumerate,bbm}
\usepackage{hyperref,color}
\usepackage[dvipsnames]{xcolor}
\hypersetup{
	colorlinks=true,
	pdfpagemode=UseNone,
    citecolor=OliveGreen,
    linkcolor=RoyalBlue,
    urlcolor=black,
	pdfstartview=FitW
}
\usepackage{tikz}  
\usepackage{comment}
\usepackage{appendix}

\usepackage[tt=false]{libertinus}

\usepackage[libertine,smallerops,varbb]{newtxmath}
\usepackage{microtype} 

\newcommand{\ind}[1]{\mathbf{1}_{\left\{ #1 \right\}}}

\DeclareMathOperator*{\argmin}{arg\,min}
\theoremstyle{plain}
\newtheorem{theorem}{Theorem}[section]
\newtheorem{proposition}[theorem]{Proposition}
\newtheorem*{proposition*}{Proposition}
\newtheorem*{theorem*}{Theorem}
\newtheorem{lemma}[theorem]{Lemma}
\newtheorem*{lemma*}{Lemma}
\newtheorem{corollary}[theorem]{Corollary}
\theoremstyle{definition}

\newtheorem{remark}[theorem]{Remark}

\newcommand*{\dif}{\ensuremath{\mathop{}\!\mathrm{d}}}

\def\sdif{\operatorname{\triangle}}

\def\R{\mathbb{R}}

\def\E{\mathbf{E}}
\def\P{\mathbf{P}} 
\def\Q{\mathbf{Q}}

\def\s{\mathfrak{s}}
\def\L{\mathscr{L}}
  
\numberwithin{equation}{section}

\renewcommand{\bar}[1]{\mkern 1.5mu\overline{\mkern-1.5mu#1\mkern-1.5mu}\mkern 1.5mu}

\title[BBM conditioned on large level sets]{Branching brownian motion conditioned on large level sets}

\author{Xinxin Chen \and Heng Ma} 
 
\address[Xinxin Chen]
{Beijing Normal University, School of Mathematical Sciences, China   
}
\email{xinxin.chen(at)bnu.edu.cn}
 
\address[Heng Ma]
{Peking University, School of Mathematical Sciences, China
}
\email{hengmamath(at)gmail.com}
\urladdr{\url{hengmamath.github.io}}

\begin{document}

 \begin{abstract}
We study the precise large deviation probabilities for the sizes of intermediate level sets in  branching Brownian motion (BBM). Our  conclusions improve a result of A\"{i}dekon, Hu and Shi in [J. Math. Sci. \textbf{238}(2019)]. Additionally, we analyze the typical behaviors of BBM conditioned on large level sets. Our approach relies on the connections between intermediate level sets, additive martingale limits of BBM, and the global minimum of linearly transformed BBMs.
\end{abstract}

\maketitle

\section{Introduction}
We consider a standard binary branching Brownian motion (BBM) on the real line, which is a particle system constructed inductively as follows. Initially, a single particle performs standard Brownian motion. After an exponentially distributed random time with mean one that is independent of its motion,
the particle undergoes a branching event: it dies and produces two offspring at its current position.  Each offspring then independently continues the same process, performing Brownian motion and eventually branching in the same manner. We denote by $\mathcal{N}_{t}$ the sets of all particles alive at time $t$, by $X_{t}(u)$ the position of $u \in \mathcal{N}_{t}$, and by $\P$ the corresponding probability measure.

To characterize how this cloud of particles is distributed in space, one of the most fundamental aspects is to determine the sizes of the level sets, defined as follows.  For each $y \geq 0$,   the $y$-level set of the BBM at time $t$ consists of all particles alive at time $t$ that are positioned above the level $y$. Let $\L_t (y)$ denote  the size of the $y$-level set:
\begin{equation}
  \L_t (y):=   \# \{   u \in \mathcal{N}_{t}:   X_{t}(u) \geq y \}  .
\end{equation}  
For $x\geq0$, the typical asymptotic behavior of the sizes of $xt$-level sets  is well understood: 
\begin{quote}
 Almost surely, for sufficiently large time $t$, the $\sqrt{2} t$-level set at large time $t$ is empty, while  $\L_{t} (x t) $ is comparable to its expectation for each $x \in [0,\sqrt{2})$.\footnote{The many-to-one formula and Gaussian tail inequality yield  that  $ \mathbf{E}\left[\L_{t} (x t) \right]= \frac{1+o(1)}{x \sqrt{2 \pi t}} \exp\{ (1- \frac{x^2}{2})t\}$.} 
\end{quote}
Precisely, it was proved  by Hu and Shi \cite{HS09}  for branching random walks (BRWs) and by   Robert \cite{Robert13} for BBMs that the maximal position  $M_{t} := \max_{u \in \mathcal{N}_{t}}X_{t}(u)$ satisfies
\begin{equation}\label{eq-HS-maximum}
  \limsup_{t \to \infty} \frac{\sqrt{2}t-M_{t} }{\ln t} = \frac{3}{2\sqrt{2}}  \text{ a.s. \, and }\, \liminf_{t \to \infty} \frac{\sqrt{2}t- M_{t}}{\ln t} = \frac{1}{2\sqrt{2}} \text{ a.s. }
\end{equation} 
This implies the emptiness of $\sqrt{2}t$-level set.  Additionally, Biggins \cite{Biggins79} for BRWs and 
 Glenz,  Kistler, Schmidt for BBMs \cite{GKS18} showed that for $x \in [0,\sqrt{2})$, 
\begin{equation}\label{eq-typical-GKS}
  \frac{\L_{t} (x t) }{\mathbf{E}\left[\L_{t} (x t) \right]} \xrightarrow[t \to \infty]{\mathrm{a.s.}} W_{\infty}(x) , 
\end{equation}
where $W_{\infty}(x)$ is the limit of the additive martingale $ W_{t}(x):=\sum_{u \in \mathcal{N}_{t}} e^{x X_{t}(u)- (1+\frac{x^2}{2}) t} $.
The distribution of particles in the BBM around the level $x t$ (with $x \in [0,\sqrt{2})$) 
was also analyzed in  \cite{Biggins79}. Specifically, it was shown that 
 $ \sum_{u \in \mathcal{N}_{t}} \frac{f(X_{t}(u)-x t)    }{\mathbf{E}\left[\L_{t} (x t) \right]} \overset{t \to \infty}{\rightarrow }  W_{\infty}(x) \int f(h) e^{-x h} \dif h$ for any function $f$ such that $h \mapsto f(h)e^{-xh}$ is direct Riemann integrable.
 
Once the typical behavior is understood, one may be interested in exploring the atypical behavior of the level sets. A\"{i}dekon, Hu and Shi \cite{AHS19} investigated the upper   deviation probability of the level sets, for both BBMs and the Gaussian free field (GFF). They demonstrated that  for $x > 0$ and $ (1-\frac{x^2}{2} )_{+}<a<1$, the following holds:
   \begin{equation}
\lim_{t \to \infty}  \frac{1}{t} \ln \mathbf{P}\left(\L_{t} (x t)  \geq e^{a t}\right) = -I(x,a), 
\end{equation} 
where the rate function $I(\cdot,\cdot)$ is defined by
\begin{equation}
  I(x,a) := \frac{x^2}{2(1-a)} - 1 . 
\end{equation}   
In other words, $\mathbf{P}\left(\L_{t} (x t)  \geq e^{a t}\right) = e^{-   I(x,a) t + o(t) }$.  This $o(t)$ term hides important information about the evolution of the BBM conditioned on the level set  being large. For example, \cite[\S 3.1]{AHS19} showed that the strategy $\max_{u \in \mathcal{N}_{pt}} X_{pt}(u) \geq bpt$ (where $p$ and $b$ are defined in \eqref{eq-def-theta-p-b}) can induce this atypical event. However, determining whether this strategy is optimal (i.e., whether this is a typical event for the conditioned process), hinges on the precise order of the $o(t)$ term. Here comes our first question:
 \begin{quote}
   \textbf{Question 1:} What is the exact decaying order of $\mathbf{P}\left(\L_{t} (x t)  \geq e^{a t}\right)$?  
 \end{quote}  
 As highlighted above, our primary interest lies in understanding the typical behavior of the BBM when the level set is unusually large.
 For instance, what is the genealogy structure of these particles above level $x t$ conditioned on $\L_{t} (x t)  \geq e^{at}$? What is the distribution of the maximal position   $\max_{u \in \mathcal{N}_{t}} X_{t}(u) $ conditioned on $\L_{t} (x t)  \geq e^{at}$? Overall, our second question read as follows:
  \begin{quote}
    \textbf{Question 2:} What is the  typical behavior of the BBM under  $\mathbf{P}\left( \cdot \mid \L_{t} (x t)  \geq e^{a t}\right)$?
  \end{quote} 

\subsection{Main results}
 Our first main result Theorem \ref{thm-LDP} solves Question 1. 
 We first introduce three important parameters that appear everywhere in this paper:
\begin{equation}\label{eq-def-theta-p-b}
 \theta:= \frac{2(1-a)}{x}, \ p :=\frac{(1-a)\left[x^2-2(1-a)\right]}{x^2-2(1-a)^2} ,
    \text { and } \  b:=\frac{2}{\theta}=\frac{x}{1-a} .
  \end{equation}
 We emphasis that $\theta \in (0,\sqrt{2})$ and $p \in (0,1-a)$ for any $x > 0$ and $ (1-\frac{x^2}{2} )_{+}<a<1$. Moreover, for any  $ |\lambda| < \sqrt{2}$, define 
\begin{equation}\label{eq-C-W-infinity}
   C_{W_{\infty}(\lambda)} := \lim_{y \to \infty} y^{\frac{2}{\lambda^2}} \mathbf{P} \left(  W_{\infty}(\lambda) > y   \right)  . 
\end{equation}
The fact that $C_{W_{\infty}(\lambda)}\in (0,\infty) $ for $ |\lambda| < \sqrt{2}$ is obtained in \cite{Liu00}, with a different proof also provided in \cite{CDM24}.

 
 \begin{theorem}[Precise Large Deviation Estimates]\label{thm-LDP}
  Fix $x > 0$, $(1-\frac{x^2}{2})_{+} <a<1$, and $y> 0$. We have 
  \begin{equation}
    \mathbf{P}\left(\L_{t} (x t)  \geq \frac{y}{\sqrt{t}} e^{a   t}\right) = [1+o(1)]  \hyperref[eq-cst-a-x]{C_{\star}} (x,a) y^{-\frac{2}{\theta^2}}  e^{-I(x,a)t }  \text{ as } t \to \infty.
  \end{equation}
  Above, the constant $  \hyperref[eq-cst-a-x]{C_{\star}}(\cdot,\cdot)$ is defined by  
    \begin{equation}\label{eq-cst-a-x}
    \hyperref[eq-cst-a-x]{C_{\star}} (x,a) :=   \sqrt{\frac{1-p}{1-p+ 2p/\theta^2}}   \frac{C_{W_{\infty}(\theta)} }{ (\theta \sqrt{2 \pi(1-p)} )^{\frac{2}{\theta^2}} }.  
  \end{equation}
\end{theorem} 
 
Our first theorem  slightly differs from Question 1,  which asks about the decay rate of $\P(\L_{t}(xt) \geq e^{at})$, simply because we find it more convenient to present the result (and its proof) for  $\P( \L_{t}(xt) \geq e^{at}/\sqrt{t} )$. The reason is that the law of large numbers \eqref{eq-typical-GKS} suggests it is more natural to consider the large deviation probability of $\L_{t}(xt)/ \E[ \L_{t}(xt) ]$, and there do have a $ t^{-1/2}$ term in $ \E[ \L_{t}(xt) ]$. Nonetheless, an immediate adaptation of our method provides the answer for the original form of Question 1:
For any fixed $y> 0$, 
   \begin{equation}\label{eq-level-set-LDP-y}
    \mathbf{P}\left(\L_{t} (x t)  \geq  y e^{a t}\right) \sim  \hyperref[eq-cst-a-x]{C_{\star}} (x,a)  y^{-\frac{2}{\theta^2}}  t^{-\frac{1}{\theta^2}} e^{-I(x,a)t } \  \text{ as } t \to \infty. 
   \end{equation}

   Here is an immediate corollary of our first theorem. 
   Recall that a random variable $X$ is said to have the Pareto distribution with parameter $\lambda>0$, if $\P( X> y)= y^{-\lambda}$ for $y\geq 1$ and $\lambda>0$. Then Theorem \ref{thm-LDP} yields that 
   \begin{equation}
    \left(  \sqrt{t} e^{-at}  \L_{t} (x t)  \mid  \L_{t} (x t)  \geq    e^{a t}/\sqrt{t}  \right)   \text{ converges in law to }   \mathrm{Pareto} \left( {2}/{\theta^2} \right) . 
   \end{equation} 
This provides a partial answer to Question 2 concerning the behavior of the level set when conditioned on it being large.
In theorems \ref{thm-CondOverlap} and \ref{thm-CondMaxs},  we examine the conditional distribution of  two particularly interesting quantities in the BBM:
the overlap and the maximal position, respectively.  
Before presenting our results, we introduce some   essential terminology. 
 For any two particles $u,v \in \cup_{r \geq 0}\mathcal{N}_{r}$,   the  overlap between $u$ and $v$ is defined as the covariance of their positions, conditionally on the genealogy structure of the BBM: 
\begin{align}
  \mathscr{R}(u,v) &:=  \E \left[ X(u) X(v) \mid   \mathcal{N}_{r}, r \geq 0  \right] \\
  & = \text{ the death time of the most recent common ancestor of } u \text{ and } v.  
\end{align} 
For a given $t>0$, we introduce  a crucial curve in studying the level set large deviation probability, as noted by \cite{AHS19}. Let 
\begin{equation} \label{AHScurve}  
  \mathsf{F}_{t}(r) := t \, \mathsf{f} \left( \frac{r}{t}  \right) , \, r \in [0,t], \text{ where } \mathsf{f}(\lambda)   := x- \ind{a+\lambda \leq 1} \sqrt{2(1-\lambda)(1-\lambda-a)}  .
\end{equation}  

\begin{theorem}[Entropy Repulsion: Overlap]\label{thm-CondOverlap} Given $x >0$ and $(1-\frac{x^2}{2} )_{+}<a<1$. 
  Define $\theta,p,b$ as in \eqref{eq-def-theta-p-b}. 
 Conditioned on the BBM up to time $t$, select two particles $u^1_{t},u^2_{t}$ independently and uniformly from the $xt$-level set. Define 
  \begin{equation}\label{eq-def-ss}
    \s :=\argmin_{t>0} \left\{ \min_{v\in \mathcal{N}_{t}} \left[ \left( \frac{\theta^2}{2}+1 \right)t- \theta X_{t}(v) \right] \right\} . 
  \end{equation}  
Then the  collection of random vectors, conditionally on a large level set size,   
  \begin{equation}\label{eq-overleap-tight}
    \left( \mathscr{R}(u^{1}_{t},u^{2}_{t})- \s   ,  X_{\mathscr{R}(u^{1}_{t},u^{2}_{t})}(u_{t}^{1})- \mathsf{F}_{t} (\mathscr{R}(u^{1}_{t},u^{2}_{t}))   \mid  \L_{t} (x t)  \geq \frac{e^{a t}}{\sqrt{t}} \right)_{t>0}    \text{ is tight }.
  \end{equation}  
As a consequence,  we have the following conditional central limit theorem: 
  \begin{equation}\label{eq-overleap-conv}
   \left(  \frac{  \mathscr{R}(u^{1}_{t},u^{2}_{t}) -  p t}{   \sqrt{pt}} ,   \frac{ X_{ \mathscr{R}(u^{1}_{t},u^{2}_{t}) }(u^{1}_{t})- bp t}{\sqrt{p t}} \mid  \L_{t} (x t)  \geq \frac{e^{a t}  }{\sqrt{t}}   \right)  \Rightarrow  \mathcal{G} \times \left(   \frac{ \theta  }{1-\theta^2/2}  ,   \frac{1+\theta^2/2}{1-\theta^2/2}  \right)  ,
  \end{equation} 
 where  $\mathcal{G}$ follows a centered Gaussian distribution with variance $\frac{1-p}{1-p+2p/\theta^2}$.  
\end{theorem} 

Recall that  $M_{t}:=\max_{u \in \mathcal{N}_{t}} X_{t}(u)$ denotes the maximal position of the particles in the BBM at time $t$. 

\begin{theorem}[Entropy Repulsion: Maximum]\label{thm-CondMaxs}
Given $x >0$ and $(1-\frac{x^2}{2} )_{+}<a<1$.  
  Define $\theta,p,b$ as in \eqref{eq-def-theta-p-b}.  Let $v : =   b p+ \sqrt{2}(1-p) $. Then  as $t \to \infty$,  
  \begin{equation}\label{eq-cond-maxi-clt}
   \left( \frac{ M_{t} - v t}{\sqrt{t}}      \mid  \L_{t} (x t)  \geq \frac{e^{a t} }{\sqrt{t}}   \right)   \Rightarrow   \frac{\sqrt{2}-\theta}{\sqrt{2}+ \theta}  \, \mathcal{G} ,
  \end{equation}
 where  $\mathcal{G}$ is the same as in \eqref{eq-overleap-conv}.
\end{theorem}

\begin{remark} In contrast to Theorems \ref{thm-CondOverlap} and \ref{thm-CondMaxs}, without   conditioning, the overlap   converges in distribution, and the maximal position exhibits a ballistic spreading speed $\sqrt{2}$ which is strictly smaller than $v$, see \eqref{eq-HS-maximum}.  More precisely, we will prove in Lemma \ref{lem-overlap} that  
  \begin{equation} 
    \lim_{t \to \infty} \P( \mathscr{R}(u^{1}_{t},u^{2}_{t}) \geq r ) =  \E \left[ \sum_{v \in \mathcal{N}_{r} } \left( e^{x X_{r}(v) - (\frac{x^2}{2}+1) r} \frac{ W_{\infty}^{(v)}(x) }{W_{\infty}(x)} \right)^2  \right], \, \forall  \, r >0  .
  \end{equation}
 And the limiting distribution has no mass at infinity: $  \lim\limits_{r \to \infty}\lim\limits_{t \to \infty} \P( \mathscr{R}(u^{1}_{t},u^{2}_{t}) \geq r ) = 0 $. Additionally, it is well-known that $M_{t}$ satisfies (see for example \cite{Bramson83,Aidekon13,BDZ16})
 \begin{equation}
  M_{t} - \sqrt{2}t + \frac{3}{2 \sqrt{2}} \ln t \text{ converges in law.}
 \end{equation}  
 We regard Theorems \ref{thm-CondOverlap} and \ref{thm-CondMaxs} as phenomena of entropy repulsion: the imposed constraint that $\L_{t} (x t) \geq e^{a t}/\sqrt{t}$ affects the natural evolution of the BBM,  compelling a particle in the BBM to reach the space-time curve $(r , F_{t}(r))$ at some random time $r$
around $pt$, with Gaussian fluctuations of order $\sqrt{t}$. 
\end{remark}

\begin{remark}
Although in this paper, we focus on the simplest binary branching Brownian motion setting  as in \cite{AHS19}, our approach could be  applied to the context of the branching random walk under the  same framework as in \cite{CDM24}. 
\end{remark}

\subsection{Related work and further questions}
In this paper, we focus exclusively  on rare events where the level sets become unusually large. The lower deviation probabilities, where the level set is unusually small, were investigated in \cite{Oz20} for BBMs and \cite{Zhang23} for BRWs, with the corresponding rate functions obtained.
It is also interesting to study BRWs on $\R^{d}$ or on general graphs.
We refer to \cite{AB14} for BRWs on  $\mathbb{R}^{d}$  and to \cite{LMW24} for BRWs on free groups, where the exponential growth rate of the level sets was obtained. Additionally,  
branching random walks are closely connected to random recursive tree (and random search trees), see  \cite{Devroye87,Pittel94,ABF13}.  In particular, by using the result on the minimum of BRWs, \cite{ABF13} showed that the height of a random recursive tree on $n$ nodes is $e \ln n - \frac{3}{2} \ln(\ln n) + O_{\P}(1)$.
A natural question arising from this is whether the results on level sets of BRWs can be extended to describe the growth rate of the number of nodes at a distance greater than $\lambda e \ln n$, with $\lambda \in (0,1)$.

In \cite{GKS18}, Glenz, Kistler, and Schmidt conjectured that a law of large numbers, as stated in \eqref{eq-typical-GKS}, holds for all models within the BBM-universality class (see \cite{Zeitouni12, Kistler14, Arguin16, BK22}).
Biskup and Louidor confirmed this for intermediate level sets of the two-dimensional discrete Gaussian free field (2D DGFF) in \cite{BL19}.  The large deviation probabilities for these level sets in the 2D DGFF were studied by A\"{i}d\'{e}kon, Hu, and Shi in \cite{AHS19}. While  the rate function was derived, a precise estimate similar to Theorem \ref{thm-LDP}, remains unclear.
Furthermore, results similar to \eqref{eq-typical-GKS} have been shown  for  local times of two-dimensional random walks in \cite{AB22}, for the local times of planar Brownian motion in \cite{Jego20}, and  for a random model of the Riemann-zeta function in \cite{AHK22}.
However, the large deviation probabilities for these models are still not fully understood, and we believe that they are worth further exploration.

\subsection{Proof ideas}  
First let us recall the arguments in \cite{AHS19}.  
If we start the BBM at time-space point $(r,\mathsf{F}_{t}(r))$, under what conditions will there be approximately $e^{at +o(t)}$ particles positioned above level $x t$ at time $t$ with positive probability? According to the law of large numbers \eqref{eq-typical-GKS}, we  need 
\begin{equation}\label{eq-AHScurve-2}
(t-r) - \frac{  [xt - \mathsf{F}_{t}(r) ]^{2} }{2(t-r)}  = at, \text{ if }  r \leq (1-a)t,  \text{ and } \mathsf{F}_{t}(r)= x t , \text{ if }r \geq (1-a) t.
\end{equation}
This is exactly the definition  \eqref{AHScurve}. Thus, in order that $\{ \L_{t}(xt) \geq e^{at}\}$ occurs,  one strategy is  to have a particle located above level $\mathsf{F}_{t}(r)$ at time $r$.
Since $\P( M_{r} \geq y) \approx e^{r- y^2/(2r)+o(r)}$ for $y \geq \sqrt{2} r$, we take  the best choice of $r$ such that $\sup _{ r \leq t }\left\{ r-\frac{ \mathsf{F}_{t}(r) ^2}{2 r}\right\}$ is attained. It turns out that the maximizer is $r^{*}= pt$, and  
\begin{equation}
   \sup _{ r \leq (1-a)t }\left\{ r-\frac{ \mathsf{F}_{t}(r) ^2}{2 r}\right\} =-I(a, x) . 
\end{equation}  
This strategy gives the  lower bound $\P(\L_{t}(xt)\geq e^{at}) \geq e^{- I(x,a)t+o(t)}$ in \cite{AHS19}. The corresponding upper bound, however, is more subtle. In \cite{AHS19}, the authors discretize space and time to 
classify possible trajectories,
and estimate the number of particles following each  one. Since the total number of  trajectories grows sub-exponentially—specifically $\exp\{O(t^{1-\delta}\ln t)\}$—it appears challenging to improve their argument to obtain a tighter bound.
However it is shown in \cite{AHS19} that the probability of having more than $e^{at}$ particles
 following a trajectory that is far from $r \mapsto \mathsf{F}_{t}(r)$ is negligible compared to $e^{-I(a,t)t+o(t)}$. This lead us to guess that the optimal strategy for the large deviation event involves forcing the BBM approach the curve $r \mapsto \mathsf{F}_{t}(r)$ near the time $pt$.

 \begin{figure}[t]
  \centering
  \includegraphics[height=6cm]{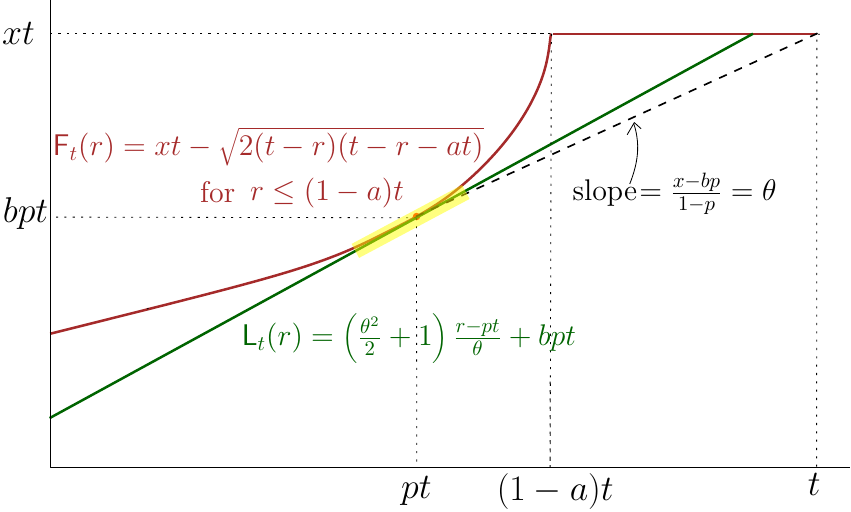} 
  \caption{The functions $\mathsf{F}_{t}$ and $\mathsf{L}_{t}$}\label{fig-FL}
 \end{figure}

 Studying the probability of BBM hitting a curve is challenging, while hitting a line is simpler.  Our approach is to approximate  $\mathsf{F}_{t}$ using its tangent line at the point  $(pt, \mathsf{F}_{t}(pt))$. Since $ \mathsf{F}_{t}(pt) = b pt$, let $\mathsf{L}_{t}$ denote this tangent line at $(pt, bpt)$, see Figure \ref{fig-FL}. Then  
 \begin{equation}
   \mathsf{L}_{t}(r):=  \left( \frac{\theta^2}{2}+1  \right)  \frac{r-p t}{\theta}  +  b p t \quad   \text{ for } r \in [0,t]
 \end{equation}
 Observe that the BBM  hits the line $\mathsf{L}_{t}$ if and only if there exists $r>0$ and $u \in \mathcal{N}_{r}$ such that $X_{r}(u)=  \left( \frac{\theta^2}{2}+1  \right)  \frac{r-p t}{\theta}  +  b p t $. We define 
\begin{equation}
  V_{r}(u) :=\left( \frac{\theta^2}{2} + 1  \right) r - \theta X_{r}(u) \  \text{ for all } u \in \mathcal{N}_{r}  , r \geq 0. 
\end{equation}  
Since $b =2/\theta$,  we can rewrite  $X_{r}(u)=  \left( \frac{\theta^2}{2}+1  \right)  \frac{r-p t}{\theta}  +  b p t $ as 
  $ V_{r}(u)=- (1-\frac{\theta^2}{2}) pt$. Set 
\begin{equation}
  \mathbf{I}:= \min \{  V_{r}(u) : u\in \mathcal{N}_r , r \geq 0 \}.
\end{equation}
Thus, the event that the BBM hits the line  $\mathsf{L}_{t}$  is equivalent to the event that the global minimum $\mathbf{I}$ of the linearly transformed BBM is smaller than $ -(1-\theta^2/2)pt$:
\begin{equation}
 \{ \text{BBM hits }  \mathsf{L}_{t} \} =   \left\{ \mathbf{I} \leq - (1- \frac{\theta^2}{2}) pt   \right\}. \label{eq-opts-1}
\end{equation} 

In \S \ref{sec-OptS} Proposition \ref{prop-up-to-constant},  we demonstrate that   the probability of the event  $\L_{t}(xt) \geq e^{at}/\sqrt{t}$ while $|\mathbf{I} +  (1- \frac{\theta^2}{2}) pt| > z>0$ is bounded above by $C e^{-I(a,x)t - z/2}$. And in \S \ref{sec-sharp-estimate} we show that for fixed $z$,
the probability of the event  $\L_{t}(xt) \geq e^{at}/\sqrt{t}$ while $\mathbf{I} +  (1- \frac{\theta^2}{2}) pt \in [z,z+1]$ is asymptotically equivalent to $C_{z} e^{-I(a,x)t}$. This implies that \eqref{eq-opts-1} provides
the optimal strategy for the large deviation event $\L_{t}(xt) \geq e^{at}/\sqrt{t}$.

The main challenge in this paper lies in proving Proposition \ref{prop-up-to-constant}, which is tackled using the following two key elements:
(i) In Lemma \ref{lem-rough-bnd} we provides 
a tighter upper bound for the large deviation probability than \cite{AHS19} by use of the  tail inequality for the martingale limit $W_{\infty}(\theta)$.
(ii)  A useful but seemingly overlooked  inequality (Lemma \ref{lem-one-big-jump}) from \cite{HN04} turns out can be 
combined with the moment estimate \eqref{martmin}  obtained recently in \cite{CDM24} and Lemma \ref{lem-rough-bnd} to conclude the desired result. 
In \S \ref{sec-sharp-estimate}, we apply the widely-used spine decomposition technique along with the results from \cite{CDM24} on BRW conditioned on an extremely negative global minimum. This approach enables us to prove Theorems \ref{thm-LDP}, \ref{thm-CondOverlap}, and \ref{thm-CondMaxs}.

 \subsection*{Notation convention}
 Through out the paper, we  assume  $x > 0$ and  $(1-\frac{x^2}{2})_{+} <a<1$.
 The parameters $\theta, p, b$ are functions of $x,a$ as defined in \eqref{eq-def-theta-p-b}. 
Let  $\psi= \psi_{\theta}$ represent the log-Laplace transform associated with $(V_{t}(u):u\in \mathcal{N}_{t})$: $
  \psi(\lambda) := \ln \mathbf{E} \left[  \sum_{u \in \mathcal{N}_{1}}  e^{-\lambda V_{1}(u)} \right] = \left( \frac{\theta^2}{2} \lambda -1 \right) \left( \lambda-1 \right) $ for $\lambda \in \mathbb{R}$.
Let $\kappa$ be the largest zero of $\psi$. Then 
\begin{equation}
  \label{def-ka-dpsika}
  \kappa = 2/\theta^2 , \  \psi'(\kappa) = 1- {\theta^2}/{2} \text{ and } \psi''(\kappa) = \theta^2 . 
\end{equation}
We can then rewrite the event in \eqref{eq-opts-1} as $\{ \mathbf{I} \leq -\psi'(\kappa)s \}$ where $s:= p t$. 

We use $ \Rightarrow   $ to denote convergence in distribution.   
We use $C$ and $c$ to denote positive constants that may vary between lines. When a constant depends on a parameter $\lambda$, we write $C_\lambda$ or $c_\lambda$. We write $f \lesssim g$ to indicate that there exists a constant $C > 0$ such that $f \leq Cg$, and $f \lesssim_\lambda g$ when the constant $C$ depends on $\lambda$. Additionally, we use the standard notation $\Theta(f)$ to denote a non-negative quantity such that $c_1 f \leq \Theta(f) \leq c_2 f$ for some constants $c_1, c_2 > 0$.

 \section{Preliminaries}
 \label{ref-preliminary}

   \subsection{Spine Decomposition}\label{sec-spine}

   The spinal decomposition is a powerful technique introduced in \cite{LPP95}
   for studying Galton-Waston processes and generalized to BRWs and BBMs in  \cite{Lyo97,Kyprianou04} etc.  Here we follow \cite[Chapter 4]{Ber14}. 
   We will distinguish one particular line of descent from the root, so at each time there will be a marked particle that we will call the spine.  By enlarging the probability space to take this extra information into account, we can significantly simplify the description of the process constructed via the tilted probability using the additive martingale $(W_{t}(\beta))_{t \geq 0}$ defined by 
   \begin{equation}
    W_{t}(\beta) := \sum_{u \in \mathcal{N}_{t}} e^{ \beta X_{t}(u)- (\frac{\beta^2}{2}+1)t} 
   \end{equation}

   First, sample a standard BBM $(X_{t}(u):u \in \mathcal{N}_{t})$ according to  $\mathbf{P}$. Conditioned on this BBM, we construct a distinguished line of descent from the root inductively as follows. Let $u_0$ denote the root particle in the BBM. For $j \geq 0$, select a child $u_{j+1}$ of $u_{j}$ uniformly at random.  
   For each particle $v$, denote its  death time by $d_{v}$  and its life time by $\eta_{v}$, respectively. 
   Define $w_{t}= u_{j}$ for $t \in [d_{u_{j}}-\eta_{u_{j}}, d_{u_{j}} )$. (Note that $d_{u_{j+1}}=\eta_{u_{j+1}}+d_{u_{j}}$, so $w_{t}$ is well-defined for all $t \geq 0$.) 
   Let $\mathbb{P}^{*}$ denote the law of the pair  $(\{X_{t}(u) :u \in \mathcal{N}_{t}\},w_{t} )_{t \geq 0}$. 
   Write  
  $\Xi_{t}:= X_{t}(w_{t})$ for the position of the spine at time $t$.
  Let $(s_{j})_{j \geq 1}$ be the times of branching events along the spine, and 
  $n=( n_t )_{t \geq 0}$ be the corresponding counting process. (In other words $n_{t}=|w_{t}|$, where $|u|$ represents the generation of particle $u$). Then we have 
  \begin{equation}\label{eq-density-Pstar}
    \dif \P^{*} (X,w) =   \dif \mathbb{W}(\Xi)    \dif \mathbb{L} (n)\prod_{j \geq 1} \frac{1}{2}   \dif \P_{\Xi(s_{j}) } ( X^{(j)} )  .
  \end{equation}
Above, $\mathbb{W}$ represents the Wiener measure, $\mathbb{L}$ represents the law of Poisson process with rate $1$, $\P_{x}$ is the law of standard BBM starting from $x$, and $X^{(j)}$ denotes the process   $\{X_{s_{j}+t}(u)-\Xi(s_{j}): u \in \mathcal{N}^{(v_{j})}_{t} : t \geq 0\}$, where $v_{j}$ is the brother of $w_{s_{j}}$. Here  $\mathcal{N}^{(v)}_{t}$ denotes the set of all descendants of particle $v$ at time $b_v+t$. 

Let $(\mathcal{F}_{t})_{t \geq 0}$ and  $(\mathcal{F}_{t}^{*})_{t \geq 0}$ denote  the natural filtration of the BBM without and with a spine, respectively. Then it is clear that $(e^{\beta \Xi_{t}- {\beta^2 t}/{2}} e^{-t}2^{n_{t}}, \mathcal{F}^{*}_{t})_{t \geq 0}$ is a $\P^{*}$ martingale.   
Fix a bounded stopping time  $\tau$ of $(\mathcal{F}^{*}_{t})_{t \geq 0}$, we  define  the probability measure $(\mathbf{Q}^{\beta,*}_{\tau} \otimes \P) $ by
   \begin{align}
   & \dif (\mathbf{Q}^{\beta,*}_{\tau} \otimes \P) (X,w)  :=  e^{\beta \Xi_{\tau}-  (\frac{\beta^2}{2}+1)\tau } 2^{|w_{\tau}| }\dif \P^{*} (X,w) \\
    &=e^{\beta \Xi_{\tau}-\frac{\beta^2}{2}\tau} \dif \mathbb{W}(\Xi)  e^{-\tau} 2^{n_{\tau}} \dif \mathbb{L} (n)  \prod_{j \geq 1} \frac{1}{2} \dif \P_{\Xi(s_{j}) } ( X^{(j)} ) .  \label{eq-density-Qtau}
  \end{align}   
The construction in \eqref{eq-density-Qtau} shows that the process under  $\mathbf{Q}^{\beta,*}_{\tau} \otimes \P$ corresponds to the law of a non-homogeneous branching motion with distinguished and randomized spine having the following properties:
\begin{itemize}
  \item Spine particles behave differently from the ordinary particles before the stopping time $\tau$: They  perform   a Brownian motion with drift $\beta$ and branch according to their own (independent) exponential clocks with rate $2$.  But after the stopping time $\tau$, spine particles follow the same dynamics as the ordinary particles. 
  \item 
  At each branching event of a spine particle, one of its two offspring is chosen uniformly at random to continue as the new spine particle, while the other becomes an ordinary particle. 
\item Ordinary particles  initiate $\P$-BBMs at their space-time point of creation.
\end{itemize}
 
\begin{lemma}\label{changeofp} 
 The spinal decomposition can be stated as follows.  
  \begin{enumerate}[(i)] 
   \item Under  $\mathbf{Q}^{\beta,*}_{\tau} \otimes \P$, $ \left( \Xi_{t} - \beta ( t \wedge \tau ): t\geq 0 \right) $ is a standard Brownian motion. 
   \item Assume that  $\tau $ is a bounded stopping time w.r.t. $(\mathcal{F}_{t})$. Then 
   \begin{equation}\label{eq-cop-11}
    \dif (\Q^{\beta,{*}}_{\tau} \otimes \P ) |_{\mathcal{F}_{\infty}} =W_{\tau}(\beta) \dif \P . 
   \end{equation}
   Moreover, for any particle $u$,
\begin{equation}\label{eq-cop-22}
    (\Q^{\beta,{*}}_{\tau} \otimes \P )(w_{\tau}= u\mid\mathcal{F}_{\infty})=\frac{ 1 }{W_{\tau}(\beta)} \, e^{\beta X_{\tau}(u)- (\beta^2/2+1) \tau } \ind{ u \in \mathcal{N}_{\tau}} .
\end{equation}

 \end{enumerate}
 \end{lemma}

 \subsection{Global minimum and martingale limits} 

In this section, we summarize the main results from \cite{CDM24}, which play a key role in the analysis of the large deviation behavior of the level set. 

 \begin{lemma}[{\cite[Theorem 1.3, Lemmas 3.1 and 3.4]{CDM24}}]\label{martingale limit and minimum} 
  Fix $\beta \in (0,\sqrt{2})$ and set $\kappa(\beta):= {2}/{\beta^2}$. Define $ \mathbf{I}^{\beta}_{t}  := \inf_{0 \leq r \leq t} \min_{u \in \mathcal{N}_{r}}  \{   (\frac{\beta^2}{2}+1)r - \beta X_{r}(u) \}$, and $ \mathbf{I}^{\beta} \equiv    \mathbf{I}^{\beta}_{\infty} := \inf_{t\geq 0}\mathbf{I}^{\beta}_{t} $.  
  \begin{enumerate}[(i)]
    \item For any $z \in \mathbb{R}$, we have 
    \begin{equation}\label{mintail}
      \mathbf{P}( \mathbf{I}^{\beta} \leq -z )  \leq  e^{ - \kappa(\beta) z}  . 
    \end{equation}
    \item There exists a continuous function  $C_{\eqref{martmin}}(\beta)$ on $\beta \in (0,\sqrt{2})$ such that 
    for $\delta= \frac{1}{2}$, $t \in [1,\infty]$ and $z \geq 0$, 
     \begin{equation}\label{martmin}
    \mathbf{E}[ W_{t}(\beta)^{\kappa(\beta) + \delta} ; \mathbf{I}^{\beta}_{t} \geq -z ] \leq C_{\eqref{martmin}}(\beta) e^{\delta z} .
     \end{equation} 
    \end{enumerate}
As a result of (i) and (ii), for any  $z>0$, we have 
\begin{align}
\mathbf{P}( W_{\infty}(\beta) \geq z  ) & \leq         \mathbf{P}( \mathbf{I}^{\beta} \leq -\ln z ) +  \mathbf{P}( W_{\infty}(\beta) \geq z , \mathbf{I}^{\beta} \geq - \ln z) \\
& \leq [1+ C_{\eqref{martmin}}(\beta)] z^{-\kappa(\beta)} .  \label{martail} 
     \end{align} 
  \end{lemma}

  In the following we state the   exact asymptotic order of some probabilities concerned in the above lemma. 
  Similar to \eqref{eq-def-ss}, define 
    $  \s^{\beta }:=\argmin_{t>0}  \{ \min_{v\in \mathcal{N}_{t}}  [  ( \frac{\beta^2}{2}+1  )t- \beta X_{t}(v)  ] \} $.

  \begin{lemma}[{\cite[Theorem 1.1]{CDM24}}]\label{BRWCondMin}
    Fix $\beta \in (0,\sqrt{2})$. Then there exists a constant $C_{\mathbf{I}^{\beta}}\in(0,\infty)$ such that as $z\to \infty$, 
      \begin{equation}\label{tailM}
      \P (  \mathbf{I}^{\beta} \leq -z ) \sim C_{\mathbf{I}^{\beta}} \,  e^{ - \kappa(\beta) z}  . 
      \end{equation}
Moreover, conditionally on $\{\mathbf{I}^{\beta} \leq-z\}$,  we have the following convergence in distribution:  
\begin{equation*} 
\left(  \exp(\mathbf{I}^{\beta}) W_{\infty}(\beta), \, \mathbf{I}^{\beta} +z ,\, \sqrt{\frac{\psi'(\kappa(\beta))}{z}} \left[  \mathfrak{s}^{\beta}- \frac{z}{\psi'(\kappa(\beta))} \right]  \mid \mathbf{I}^{\beta} \leq-z  \right)  \Longrightarrow ( \mathcal{Z}, -\mathcal{E}, \mathcal{G})  
\end{equation*}  as $ z \to \infty$, 
where $\mathcal{Z},\mathcal{E},G$ are
independent random variables. Here, $\mathcal{Z}$ is positive, $\mathcal{E}$ follows an  exponential distribution with mean $\kappa(\beta)^{-1}$,   and $\mathcal{G}$ has a centered Gaussian distribution with variance $\psi''(\kappa(\beta))/\psi'(\kappa(\beta))^2$.  Moreover, the constant defined in \eqref{eq-C-W-infinity} satisfies $C_{W_{\infty}(\beta)} = C_{\mathbf{I}^{\beta}} \E[ \mathcal{Z}^{\kappa(\beta)}] $.
\end{lemma}

  \subsection{High points of BBM}
The following lemma is a slightly enhanced version of \eqref{eq-typical-GKS}.
As the result follows directly from the exact same argument used in  \cite[Theorem 1.1 and Proposition 1.3]{GKS18}, the proof is omitted. 
  
  \begin{lemma}\label{lem-enhanced-convergence}
   Fix $\beta \in [0,\sqrt{2})$. If $(\delta_{t})_{t >0}$ satisfies  $\delta_{t} \to 0$ as $t \to \infty$, then we have 
   \begin{equation} 
   \sup_{|\hat{\beta}-\beta| \leq \delta_{t}}  \left| \frac{\L_{t} ( \hat{\beta}  t) }{\mathbf{E} [\L_{t} (  \hat{\beta} t) ]} -  W_{\infty}(\beta) \right|  \overset{t \to \infty}{\longrightarrow} 0  \     \text{ a.s.} 
  \end{equation}
  Moreover, 
  there exists a constant  $c_{\beta}>0$ such that for large $t$ (depending only on $\beta$ and $\delta_{t}$),
    \begin{equation}\label{eq-exp-condexp-diff}
      \sup_{|\hat{\beta}-\beta| \leq \delta_{t}}   \mathbf{P}\left(  \left| \frac{\L_t (\hat{\beta} t)}{\mathbf{E}[\L_t (\hat{\beta} t)]}   -  \frac{\mathbf{E} ( \L_t (\hat{\beta} t) | \mathcal{F}_{r} ) }{\mathbf{E}[\L_t (\hat{\beta} t)]}   \right| \geq 1  \right) \leq 4 e^{- c_{\beta} r} . 
    \end{equation} 
  \end{lemma}
   
Conditionally on the BBM,  
we choose two individuals $u^{1}_{t}$ and $u^{2}_{t}$ independently and uniformly at random from the level set $ \{ v \in \mathcal{N}_{t}: X_{t}(v) \geq \beta t  \}$.  We denote the (annealed) probability that the most recent common ancestor $u^{1}_{t}$ and $u^{2}_{t}$ are born after time $r$ by  
   \begin{equation}
     \mathtt{OL}_{t}(r, \beta):= \P   ( \mathscr{R}(u^{1}_{t} , u^{2}_{t}) \geq r ) 
   \end{equation}
   In the following Lemma we will show that $     \mathtt{OL}_{t}(r, \beta)$ converges to a limit $     \mathtt{OL}(r, \beta) $ as $t \to \infty$, and that the limit satisfies $ \mathtt{OL}(r, \beta) \to 0$ as $r \to \infty$. Jagannath studied  the overlap distribution of BRWs with Gaussian increments in \cite{Jagannath16}. (There is a slight difference between what we are concerned and the standard questions about overlap distribution).  
    Recently, Chataignier and Pain \cite{CP24} discovered an intriguing phase transition in the convergence rate of  $\mathtt{OL}_{t}(at, \beta) $ at $  \beta = \sqrt{2}/\sqrt{3}$ where particles ($u^1_{t} and u^{2}_{t}$) are chosen according to the Gibbs measure.

 \begin{lemma}[Overlap distribution]\label{lem-overlap}
   Fix $\beta \in [0,\sqrt{2})$. Assume that    $(\delta_{t})_{t \geq 0}$ satisfies $ \delta_{t}\to 0$ as $t \to \infty$.    Define  for each $r \geq 0$, 
   \begin{equation}
     \mathtt{OL} (r, \beta):=\E \left[ \sum_{v \in \mathcal{N}_{r} } \left( e^{\beta X_{r}(v) - (\frac{\beta^2}{2}+1) r} \frac{ W_{\infty}^{(v)}(\beta) }{W_{\infty}(\beta)} \right)^2  \right].
   \end{equation}
   Then   for any $r>0$, we have 
   \begin{equation}\label{eq-ol-conv-1}
    \lim_{t \to \infty}  \sup_{|\delta| \leq \delta_{t} } \left| \mathtt{OL}_{t}(r, \beta+\delta )  -   \mathtt{OL}(r, \beta)   \right|   \overset{t \to \infty}{\longrightarrow} 0. 
   \end{equation}
   Additionally, we have $ \lim_{r \to \infty}   \mathtt{OL}(r, \beta)  = 0$. 
   \end{lemma}
  
\subsection{Useful inequalities}

The following lemma is borrowed from  \cite{HN04}, by which  the large deviation probabilities for heavy-tailed random walks was studied.  In our context, this lemma can be effectively combined with the bound \eqref{martmin}. 

\begin{lemma}[{\cite[Lemma 3.2]{HN04}}]\label{lem-one-big-jump}
Let $S_{n}=\sum_{i=1}^{n} X_{i}$, where $(X_{i})_{i=1}^{n}$ are independent nonnegative r.v.'s with finite expectation.  Let $\mu_{n}:= \E(S_{n})$. Then for any $t>0$, and $ \lambda>0$, we have 
\begin{equation}
  \P\left(S_{n}> t\right) \leq  \sum_{i=1}^{n} \P\left(X_{i}> \frac{t}{\lambda} \right)+\left(\frac{e \mu_{n}}{t}\right)^{\lambda}.
\end{equation} 
\end{lemma}
      
The following are classical Chernoff bounds. 
The first inequality can also be deduced directly from Lemma \ref{lem-one-big-jump}. For the second one, see for example, \cite[Theorem 2.4.7]{Roc24}.

\begin{lemma}[Chernoff bounds]\label{Chernoff-bounds}
  Let $X_1, \cdots, X_n$ be independent $\{0,1\}$-valued random variables  with $\P\left(X_i=1\right)=p_i$. Let $X=\sum_{i=1}^n X_i$ and $\mu=\mathbf{E}[X]$. Then for $t>\mu$,  
 \begin{equation}
   \P(X > t) \leq  \left( \frac{e\mu}{t}  \right)^{t}. 
 \end{equation} 
 and for $\delta \in (0,1)$
\begin{equation}
   \P(X \leq(1-\delta) \mu) \leq \mathrm{e}^{-\mu \delta^2 / 2}.
\end{equation}
\end{lemma}

\section{The Optimal Strategy}
\label{sec-OptS}
This section is devoted to proving that the optimal strategy to make the size of the level set $\L_{t}(xt) $ exceed $ e^{at}/\sqrt{t}$ is to drive the global minimum $\mathbf{I}$ of the linearly transformed BBM $(V_{t}(u): u \in \mathcal{N}_{t})$ towards $- \psi'(\kappa) s $. We will prove a slightly stronger result that holds uniformly for $(x,a)$. For small $\epsilon >0$, let us introduce
\begin{equation}\label{eq-admissable-xa}
 \mathcal{A}^{\eqref{eq-admissable-xa}}_{\epsilon}:= \left\{ (x,a): x  \in [2\sqrt{\epsilon}, \epsilon^{-1}], (1-{x^2}/{2})_{+}+\epsilon   \leq a \leq 1-\epsilon \right\}. 
\end{equation} 
 
\begin{proposition}\label{prop-up-to-constant} 
 Given $\epsilon \in (0,1/4)$.  There exists  constant $C_{\ref{eq-up2cons}}(\epsilon)$ depending only on $\epsilon$ such that for any $(x,a) \in   \mathcal{A}^{\eqref{eq-admissable-xa}}_{\epsilon}$, $t \geq 1$, and  $0 \leq z \leq \ln^{3/2}(t)$,  
  \begin{equation}\label{eq-up2cons}
    \mathbf{P}\left(\L_{t} (x t)  \geq \frac{1}{\sqrt{t}} e^{a   t} , \left| \mathbf{I} +   \psi'(\kappa)  s \right| >z \right) \leq C_{\ref{eq-up2cons}}(\epsilon) \, e^{-I(x,a)t }  \, e^{-\frac{1}{2}z}. 
  \end{equation}
  As a result, we have $ \mathbf{P}(\L_{t} (x t)  \geq \frac{1}{\sqrt{t}} e^{a   t} ) \leq  C_{\ref{eq-up2cons}}(\epsilon) \, e^{-I(x,a)t } $ for any $  (x,a) \in \mathcal{A}^{\eqref{eq-admissable-xa}}_{\epsilon}$.
\end{proposition}

\subsection{The road to Proposition \ref{prop-up-to-constant} }\label{sec-Up-to-cst}
Our first observation is that, 
the large deviation probability of $\L_{t}(x) \geq e^{a t}/\sqrt{t}$ can be upper bounded by using the tail probability of the additive martingale $W_{\infty}(\theta)$ with $\theta= {2(1-a)}/{x}$.  We first introduce a set which is slightly larger than \eqref{eq-admissable-xa}, where  the case  $a= 1-{x^2}/{2} >0$ is allowed. For $\epsilon>0$, set  
\begin{equation}
 \widetilde{\mathcal{A}}^{\eqref{eq-admissable-xa2}}_{\epsilon}
 := \{ (x,a): x  \in [2\sqrt{\epsilon}, \epsilon^{-1}], \max\{ 1-{x^2}/{2} , \epsilon  \}  \leq a \leq 1-\epsilon\} .\label{eq-admissable-xa2}
\end{equation}

\begin{lemma}\label{lem-rough-bnd}
  Given $\epsilon \in (0,1/4)$. Then there exists a constant $C_{\ref{eq-rough-bnd}}(\epsilon)$ (see  \eqref{eq-rough-bound-cst} for definition) depending only on $\epsilon$ such that  for any $(x,a) \in\widetilde{\mathcal{A}}^{\eqref{eq-admissable-xa2}}_{\epsilon} $ and $t \geq 1$   
   \begin{equation}\label{eq-rough-bnd}
     \mathbf{P}\left(\L_t (x t) \geq \frac{1}{\sqrt{t}} e^{ a t }\right)\leq   C_{\ref{eq-rough-bnd}}(\epsilon) \exp \left\{ -I(x,a) t +\frac{  x^{2}}{4(1-a)^2} \ln t   \right\}. 
   \end{equation}
 \end{lemma}

 For technical reasons, the upper bound in Lemma \ref{lem-rough-bnd} includes an additional $\ln t$ term, which should not appear when compared to our final result in Theorem \ref{thm-LDP}. In the following lemma,  we eliminate this extra  $\ln t$ term for the moderate deviation case where ${ \L_t (xt)}/{\mathbf{E}\L_t ( xt) } \geq e^{\ell(t)}$ with $1\ll \ell(t)\ll t$.

 \begin{lemma}\label{Moderate-deviation}
Fix $\beta \in (0,\sqrt{2})$ and  define   $\mathbf{I}^{\beta}$ as in Lemma \ref{martingale limit and minimum}. 
Assume that $\ell(t)$ and $\delta_{t}$ satisfy $\ell(t) \to \infty $, $\ell(t)/t \to 0 $ as $t \to \infty$ and $\delta_{t}^2 \leq \ell(t)/t $.
 Then the following assertions hold.
 \begin{enumerate}[(i)]
  \item For sufficiently large $t$ depending only on $\beta$, $\ell(t)$ and $\delta_{t}$, the following inequality holds  for any $|\delta| \leq \delta_{t}$ and 
   $ \frac{\delta^2 t}{2} \leq y \leq \ell(t) $:
  \begin{equation}
   \mathbf{P}\left( \frac{ \L_t ([\beta+\delta]t )}{\mathbf{E}\L_t ([\beta+\delta]t ) } \geq e^{y}\right) \leq  2[1+ C_{\eqref{martmin}}(\beta)]   e^{ - \frac{2}{\beta^{2}} \left( y - \frac{\delta^2 t }{2}  \right)} 
   ;
  \end{equation} 
  \item  For sufficiently large $t$ depending only on $\beta$, $\ell(t)$ and $\delta_{t}$, the following inequality holds for any $|\delta| \leq \delta_{t}$,
  $ \frac{\delta^2 t}{2} \leq y \leq \ell(t) $ and $0 \leq z \leq y- \frac{\delta^2 t}{2}$:
 \begin{equation} 
   \mathbf{P}\left( \frac{ \L_t ([\beta+\delta]t )}{\mathbf{E}\L_t ([\beta+\delta]t ) } \geq e^{y} , \mathbf{I}^{\beta} \geq - y + \frac{\delta^2 t}{2} + z \right) \leq 2 C_{\eqref{martmin}}(\beta)  e^{  - \frac{2}{\beta^{2}} \left( y - \frac{\delta^2 t }{2}  \right) - \frac{1}{2}z
  }   .
\end{equation}
 \end{enumerate} 
\end{lemma}

\begin{remark}
  The upper bound provided  in assertion (i) of Lemma \ref{Moderate-deviation} for the probability $    \mathbf{P}\left(\frac{  \L_t (\beta t )}{\mathbf{E}\L_t (\beta t ) } \geq e^{\ell(t) }\right) $
 is likely not optimal  when $\ell(t) \gg \sqrt{t}$. In this case   as suggested by Theorem \ref{thm-LDP},  the bound should instead be of the form  $ \exp (   - {\ell(t)}/{ [\frac{\beta^2}{2}- \frac{\ell(t)}{t}] }   ) $. However, the bound provided in Lemma \ref{Moderate-deviation} is sufficient for our proof.
\end{remark} 

Now we are ready to prove the main result in this section.

\begin{proof}[Proof of Proposition \ref{prop-up-to-constant} admitting Lemmas \ref{lem-rough-bnd} and \ref{Moderate-deviation}] It suffices to show  that there is a constant $C>0$ depending only on $\epsilon$ such that for sufficiently large $t$ and for any  $(x,a) \in   \mathcal{A}^{\eqref{eq-admissable-xa}}_{\epsilon}$, and  $z \in [0, \ln^{3/2}(t)]$, the following inequality holds:
  \begin{equation}\label{eq-p-t-z-1}
 \mathrm{P}(t,z):= \mathbf{P}\left(\L_{t} (x t)  \geq \frac{1}{\sqrt{t}} e^{a   t} ,  \mathbf{I}     \geq - \psi'(\kappa)   s   + z \right) \leq  C  e^{-I(x,a)t }  e^{-\frac{1}{2}z}. 
  \end{equation} 
 In fact, by using Lemma \ref{martingale limit and minimum} and the facts that $\psi'(\kappa)=1-\frac{\theta^2}{2}$,  $\theta b = 2$, $I(x,a)= (\frac{b^2}{2}-1)p$, we have 
 \begin{equation}\label{eq-p-t-z-2}
\mathbf{P} \left( \mathbf{I} \leq -\psi'(\kappa)s \pm z \right)  \leq  e^{ \frac{2}{\theta^2}(\frac{\theta^2}{2}-1)s } e^{\pm \frac{2}{\theta^2} z } = e^{- I(x,a)t } e^{ \pm \frac{2}{\theta^2} z } .
 \end{equation}
Recall that $\theta< \sqrt{2}$. Then the proposition can be derived by combining \eqref{eq-p-t-z-1} and \eqref{eq-p-t-z-2} with the following inequalities: 
\begin{align}
  &\mathbf{P}\left(\L_{t} (x t)  \geq \frac{1}{\sqrt{t}} e^{a   t}\right) \leq \mathrm{P}(t,1)+ \mathbf{P} \left( \mathbf{I} \leq -\psi'(\kappa)s + 1   \right) \  \text{ and } \\
 & \mathbf{P}\left(\L_{t} (x t)  \geq \frac{1}{\sqrt{t}} e^{a   t} , \left| \mathbf{I} +   \psi'(\kappa) s \right| >z \right) \leq   \mathrm{P}(t,z)+ \mathbf{P} \left( \mathbf{I} \leq -\psi'(\kappa)s -z    \right). 
\end{align} 
 
 Now, in order to control $\mathrm{P}(t,z)$, we shall apply Lemma \ref{lem-one-big-jump} 
   to the conditional probability
  \begin{equation}
    \mathbf{P}\left(\L_{t} (x t)   
  \geq \frac{1}{\sqrt{t}} e^{a   t}, \mathbf{I}     \geq - \psi'(\kappa) s   + z \mid \mathcal{F}_{s} \right).
  \end{equation}
We decompose the BBM at time $s$ as follows.  For each  $ u \in \mathcal{N}_{s}$, let $\mathcal{N}^{(u)}_{r}$ denote the set of particles alive at time $s+r$ that  are descendants of $u$. Define  $X^{(u)} := ( X^{(u)}_{r}(v): v \in \mathcal{N}^{(u)}_{r})$, where $X^{(u)}_{r}(v):= X_{r+s}(v)- X_{s}(u) $.  
Thanks to the  branching property, conditioned on $\mathcal{F}_{s}$, $\{X^{(u)} : u \in \mathcal{N}_{s} \}$ are independent BBMs. For each $u \in \mathcal{N}_{s}$, define $\Delta_s(u) := bs-X_{s}(u)$, $M^{(u)}_{r}:= \max _{v \in \mathcal{N}^{(u)}_{r}} X^{(u)}_{r}(v)$, and 
\begin{equation}
\L^{(u)}_{r}(y):= \# \{ v \in \mathcal{N}^{(u)}_{r}: X^{(u)}_{r}(v) \geq y\} . 
\end{equation} 
Note that $\mathbf{I} \leq \inf_{r>0}   \left\{   ( \frac{\theta^2}{2}+1) (s + r) - \theta M^{(u)}_{r} -\theta X_{s}(u) \right\}, \forall \, u \in \mathcal{N}_{s}$. 
Thus, provided that $\mathbf{I} \geq -  ( 1- \frac{\theta^2}{2} )s + z $, we have, for any $u \in \mathcal{N}_{s}$, 
\begin{equation}
\mathbf{I}^{(u)} := \inf_{r>0} \left(  \frac{\theta^2}{2} + 1 \right)r - \theta M^{(u)}(r)  \geq - \theta(bs-X_{s}(u))+z = - \theta\Delta_s(u)+z  . 
\end{equation} 
Then it follows that 
  \begin{equation}\label{eq-sum-01}
    \L_{t} (x t)  \ind{\mathbf{I}     \geq - \psi'(\kappa) s    -z} \leq \sum_{u \in \mathcal{N}_{s}} \L^{(u)}_{t-s}(xt-X_{s}(u))  \ind{\mathbf{I}^{(u)}    \geq - \theta \Delta_s(u) +z } . 
  \end{equation}  
Besides, when    $ \mathbf{I} \geq - (1-\frac{\theta^2}{2})s+z$, since $ \mathbf{I} \leq \mathbf{I}_{s} \leq  (\frac{\theta^2}{2}+1)  s- \theta X_{s}(u) $ for any $u \in \mathcal{N}_{s}$,   by using the fact that $\theta b = 2$, we get 
  \begin{equation}\label{eq-Dsl0}
   \Delta_s(u) \geq  {z}/{\theta}\geq 0 , \ \forall \, u \in \mathcal{N}_{s}.
  \end{equation} 
 
Now applying   Lemma  \ref{lem-one-big-jump} to the right hand side  of \eqref{eq-sum-01}, and from the branching property, it follows  that for $\lambda = \frac{2}{\theta^2}+ \delta $ with $\delta=1/2$, we have 
  \begin{align}
    & \mathbf{P}\left(\L_{t} (x t)   
    \geq \frac{1}{\sqrt{t}} e^{a   t} , \mathbf{I} \geq -  \psi'(\kappa) s + z   \mid \mathcal{F}_{s} \right) \\ 
    &\leq \sum_{u \in \mathcal{N}_{s}} 
    \underbrace{
   \mathbf{P} \left( \L^{(u)}_{t-s}(xt-X_{s}(u)) \geq   \frac{1}{\lambda \sqrt{t}} e^{a t} ,  \mathbf{I}^{(u)} \geq - \theta \Delta_s(u) + z   \mid \mathcal{F}_s  \right) }_{\mathrm{Pr}_{\ref{eq-one-jump}}(s,u,z)} \ind{\Delta_s(u) \geq z/\theta}
       \label{eq-one-jump} \\
    &\qquad +   \underbrace{\left(  \frac{  e\, \sum_{u \in \mathcal{N}_{s}} \mathbf{E} [ \L^{(u)}_{t-s}(xt-X_{s}(u) ) \mid \mathcal{F}_s ]}{ \frac{1}{\sqrt{t}} e^{a t} } \right)^{\lambda}  }_{\Sigma_{\ref{eq-many-jumps}}(s) }  \ind{\mathbf{I}_{s} \geq -\psi'(\kappa)s +z}  . \label{eq-many-jumps}
   \end{align}
   Taking the expectation on both sides of the inequality gives us
\begin{equation}
 \mathrm{P}(t,z) \leq    \mathbf{E} \left[  \sum_{u \in \mathcal{N}_s} \mathrm{Pr}_{\ref{eq-one-jump}}(s,u,z)  \ind{\Delta_s(u) \geq z/\theta} \right] + \mathbf{E} \left[  \Sigma_{\ref{eq-many-jumps}}(s) ;  \mathbf{I}_{s} \geq -\psi'(\kappa)s +z \right]. \label{eq-ubn-ptz}
\end{equation} 
We claim  that there exists   constant $C$ depending only on $\epsilon$ such that for large $t$ depending on $\epsilon$ and for all $(x,a) \in \mathcal{A}^{\eqref{eq-admissable-xa}}_{\epsilon}$ and  $z \in [0, \ln^{3/2}(t)]$,  
\begin{align}
  & \mathbf{E} \left[  \sum_{u \in \mathcal{N}_s} \mathrm{Pr}_{\ref{eq-one-jump}}(s,u,z) \ind{\Delta_s(u) \geq z/\theta} \right] \leq C e^{-I(x,a)t- \frac{1}{2}z} .\label{eq-cliam-1} \\
  & \mathbf{E} \left[  \Sigma_{\ref{eq-many-jumps}}(s) ;  \mathbf{I}_{s} \geq - \psi'(\kappa) s +z \right] \leq C  e^{-I(x,a)t- \frac{1}{2}z} . \label{eq-cliam-2} 
\end{align}
Then the desired result \eqref{eq-p-t-z-1} follows from \eqref{eq-ubn-ptz}, \eqref{eq-cliam-1} and \eqref{eq-cliam-2}. 
 \vspace{5pt}
  
\underline{\textit{Proof of \eqref{eq-cliam-1}}}.  We upper bound $\mathrm{Pr}_{\ref{eq-one-jump}}(s,u,z)$, according to the value of $X_{s}(u)$. Firstly notice that by the many-to-one formula and  Gaussian tail inequality, 
\begin{align}
  & \mathbf{E} \left[  \sum_{u \in \mathcal{N}_s} \mathrm{Pr}_{\ref{eq-one-jump}}(s,u,z) \ind{X_{s}(u) \leq - 2b s} \right] 
  \leq  \mathbf{E} \left[  \sum_{u \in \mathcal{N}_s}  \ind{X_{s}(u) \leq - 2b s} \right] \\
   & \leq  e^{s} \P( N(0,s) \geq 2 b s) \leq e^{-(\frac{(2b)^2}{2}-1)s}  \leq e^{-I(x,a)t-\ln^{3/2}(t)}. 
\end{align} 
So it suffices to bound $ \mathrm{Pr}_{\ref{eq-one-jump}}(s,u,z)$ for which $u$ satisfies that $z/\theta \leq \Delta_{s}(u)   \leq 3 bs$. 

Consider first that $\Delta_s(u) \in [\sqrt{s}\ln s, 3 b s]$.  
We directly drop the restriction on $\mathbf{I}^{(u)} $ in  the expression \eqref{eq-one-jump}  of $ \mathrm{Pr}_{\ref{eq-one-jump}}(s,u,z)$
and then apply Lemma \ref{lem-rough-bnd}. This yields that 
\begin{align}
&   \mathrm{Pr}_{\ref{eq-one-jump}}(s,u,z)
  \leq \mathbf{P} \left( \L^{(u)}_{t-s}(xt-X_{s}(u)) \geq   \frac{1}{\lambda \sqrt{t}} e^{a t}    \mid \mathcal{F}_s  \right) \\
 & \lesssim_{\epsilon}    \exp \left\{ - \left[ \frac{[xt-X_{s}(u)]^2}{2(t-s)^2(1-\frac{at}{t-s})}  - 1 \right] (t-s) + \frac{[xt-X_{s}(u)]^2}{4(t-s)^2(1-\frac{at}{t-s})^2}   \ln t   \right\} \\
 & =   \exp \left\{ (t-s)-   \frac{[xt-X_{s}(u)]^2}{2(t-s-at)}    + \frac{[xt-X_{s}(u)]^2}{4(t-s-at)^2}   \ln t   \right\}.
\end{align}
Here we check the requirements in Lemma \ref{lem-rough-bnd}.
In fact $(\frac{xt-X_{u}(s)}{t-s}, \frac{at}{t-s}) \in \widetilde{\mathcal{A}}^{\eqref{eq-admissable-xa2}}_{f(\epsilon)} $ for all $(x,a) \in \mathcal{A}^{\eqref{eq-admissable-xa}}_{\epsilon} $
 provided that $f(\epsilon)>0$  sufficiently small because $\frac{xt-X_{u}(s)}{t-s} =  \theta + \frac{\Delta_{u}(s)}{t-s} $ and  $\frac{at}{t-s}= 1- \frac{\theta^2}{2}$ and $0< \inf_{(x,a) \in   \mathcal{A}^{\eqref{eq-admissable-xa}}_{\epsilon}}   \theta< \sup_{(x,a) \in   \mathcal{A}^{\eqref{eq-admissable-xa}}_{\epsilon}}   \theta < \sqrt{2}$. By applying the   many-to-one formula we get that  
 \begin{align}
   & \mathbf{E} \left[  \sum_{u \in \mathcal{N}_s} \mathrm{Pr}_{\ref{eq-one-jump}}(s,u,z) \ind{\Delta_s(u) \in [\sqrt{s}\ln s , 3bs]}  \right]   \\
   & \lesssim_{\epsilon} e^{s}  \mathbf{E}\left[  e^{ (t-s) -   \frac{(xt-B_{s})^2}{2(t-s-at)}  } \, t^{ \frac{(xt-B_{s})^2}{4(t-s-at)^2 }     } \ind{bs-B_{s}\geq \sqrt{s}\ln s}  \right] \\ 
  & =   e^{t} \int_{-\infty}^{1- \frac{\ln s}{b\sqrt{s}}}  e^{  -   \frac{(xt- \alpha b s)^2}{2(t-s-at)}  } \, t^{ \frac{(xt- \alpha b s)^2}{4(t-s-at)^2 }     }    e^{-\frac{\alpha^2 b^2  s}{2}}  \frac{ bs \dif \alpha }{\sqrt{2 \pi s}}.
 \end{align}
Above, in the equality we rewrite the expectation according to the density of $B_{s}/(bs)$. For $ \alpha \in (-\infty,1 ]  $, define 
\begin{align}
  H(\alpha) & :=      \frac{b^2 \alpha^2 p}{2}    +  \frac{(x-\alpha b p)^2}{2(1-p-a)} -1  =    \frac{b^2}{2}  \alpha ^2 p   +   \frac{b^2}{2}\frac{  (1-a-\alpha  p)^2}{(1-a-p)}  -1
\end{align}
where in the second equality we used that $x= b(1-a)$.  
From the computations above we deduce that 
\begin{equation}
  \mathbf{E} \left[  \sum_{u \in \mathcal{N}_s} \mathrm{Pr}_{\ref{eq-one-jump}}(s,u,z) \ind{\Delta_s(u) \in (\sqrt{s}\ln s , 3bs]} \right]  \lesssim_{\epsilon}  \int_{-\infty}^{1- \frac{\ln s}{b \sqrt{s}}} t^{ \frac{b^2 (1-a- \alpha p)^2}{4(1-a-p)^2 }     }   e^{- H(\alpha) t }     \frac{ bs \dif \alpha }{\sqrt{2 \pi s}},
\end{equation}
Note that 
$H(1)= \frac{  b^2}{2}p  +  \frac{(x-  b p)^2}{2(1-p-a)} -1   =  (\frac{b^2}{2}-1)p =I(x,a)$ since $I(\frac{a}{1-p}, \frac{x-bp}{1-p})=0$. 
 We compute that 
\begin{equation}
  H'(\alpha) 
  = - (1-\alpha) \frac{(1-a)b^2 p }{1-a-p}< 0  \text{ for } \alpha <1 \text{, and } H' (1) = 0. 
\end{equation} 
One see directly from the definition of $H$ that   $H''(\alpha) \equiv  b^2(p+\frac{p^2}{1-a-p}) > 0$.
It follows from the monotonicity of $H$ and   Taylor's theorem that   for  large $t$  
\begin{equation}
  \inf_{-\infty < \alpha \leq 1- \frac{\ln s}{b\sqrt{s}} } H(\alpha ) = H(1-\frac{\ln s}{b\sqrt{s}}) =
I(x,a)+  \frac{H''(1)}{2} \frac{ \ln^2 s}{b^2 s} \geq I(x,a)+  \frac{\ln^2 (s)}{2t} . 
\end{equation}
 Thus it follows from the classical Laplace method that  for large $t$ depending only on $\epsilon$, 
\begin{equation}
  \mathbf{E} \left[  \sum_{u \in \mathcal{N}_s} \mathrm{Pr}_{\ref{eq-one-jump}}(s,u,z)  \ind{\Delta_s(u) \in [\sqrt{s}\ln s , 3bs]} \right]  \lesssim_{\epsilon} e^{- I(x,a) t - \frac{1}{4}\ln^2(t)}. 
\end{equation}

It remains to consider the case where $0 \leq \Delta_s(u) \leq \sqrt{s}\ln s$. By using the many-to-one formula, the Gaussian tail asymptotic, and the fact that  $xt-X_{s}(u)= \theta(t-s)+\Delta_s(u) $,   we compute 
\begin{align}
  \mathbf{E}  \left[  \L^{(u)}_{t-s}  (xt-X_{s}(u) ) \mid \mathcal{F}_s  \right]  &= \frac{[1+o(1)]\sqrt{t-s}}{\sqrt{2 \pi}[\theta(t-s)+\Delta_s(u)]} e^{ (t-s) - \frac{(\theta (t-s)+ \Delta_s(u))^2}{2(t-s)} } \\
  & =  \frac{[1+o(1)]\sqrt{t-s}}{\sqrt{2 \pi}[\theta(t-s)+\Delta_s(u)]}   e^{ at } e^{- \theta\Delta_s(u) } e^{ - \frac{\Delta_s(u)^2}{2(t-s)} }  ,   \label{eq-exp-Z-u-t-s}
\end{align}
where the term $o(1)$ is deterministic and uniformly in $\Delta_s(u) \in [0,\sqrt{s}\ln s]$. Above in the last equality we have used that $at= ( 1- \frac{\theta^2}{2} )(t-s)$. As a result of \eqref{eq-exp-Z-u-t-s}, we get  
 \begin{align}
 &  \mathrm{Pr}_{\ref{eq-one-jump}}(s,u,z)
  \leq \mathbf{P} \left( \L^{(u)}_{t-s}(xt-X_{s}(u)) \geq   \frac{1}{\lambda \sqrt{t}} e^{a t} ,  \mathbf{I}^{(u)} \geq - \theta\Delta_s(u)+z   \mid \mathcal{F}_s  \right)  \\
  & \leq \mathbf{P} \left( \frac{ \L^{(u)}_{t-s}( \theta(t-s)+\Delta_s(u) )}{ \mathbf{E} [ \L^{(u)}_{t-s} ( \theta(t-s)+\Delta_s(u) ) \mid \mathcal{F}_s ]} \geq c e^{  \theta\Delta_s(u) +   \frac{\Delta_s(u)^2}{2(t-s)} }  ,  \mathbf{I}^{(u)} \geq - \theta\Delta_s(u)+z   \mid \mathcal{F}_s  \right)
 \end{align}
 where $c := \frac{\sqrt{2 \pi}}{2\lambda} \theta\sqrt{1-p}> 0 $.  Now we apply part (ii) of Lemma \ref{Moderate-deviation} (where $\beta$ in  Lemma \ref{Moderate-deviation} is taken to be $\theta$, $t$ in  Lemma \ref{Moderate-deviation} is taken to be $t-s$,   $\delta$ is taken to be $\frac{\Delta_s(u)}{t-s}$ and $y$ is taken to be $\theta \Delta_s(u) + \frac{\Delta_s(u)^2}{2(t-s)} + \ln c$). This yields that  
\begin{equation}
  \mathrm{Pr}_{\ref{eq-one-jump}}(s,u,z) \lesssim_{\epsilon}  \exp \left\{  - \frac{2}{\theta^2} \theta \Delta_s(u) - \frac{1}{2} z \right\}   =    \exp \left\{    b(X_{s}(u)-bs)     - \frac{1}{2} z \right\}  .
\end{equation}
Above, in the equality we have used  that $\theta b = 2$. Finally we conclude that  
\begin{align}
  & \mathbf{E} \left[  \sum_{u \in \mathcal{N}_s} \mathrm{Pr}_{\ref{eq-one-jump}}(s,u,z) \ind{0 \leq \Delta_s(u) \leq \sqrt{s}\ln s}  \right]  \\
  & \lesssim_{\epsilon}   e^{-\frac{1}{2}z } \E \left[ \sum_{u \in \mathcal{N}_s}  e^{ b(X_{s}(u)-bs) }        \right]  =     e^{-\frac{1}{2}z } e^{-(\frac{b^2}{2}-1)s} =   e^{-\frac{1}{2}z } e^{-I(x,a)t} ,
\end{align}
where in the last equality we have used that  $I(x,a)t= (\frac{b^2}{2}-1)s$. 
This proves \eqref{eq-cliam-1}.
 \vspace{5pt}

\underline{\textit{Proof of \eqref{eq-cliam-2}}}. 
On the event that  $\mathbf{I}_{s} \geq -\psi'(\kappa) s+z$, since $\Delta_s(u)\geq 0$ for each $u \in \mathcal{N}_{s}$ (see \eqref{eq-Dsl0}), it follows from \eqref{eq-exp-Z-u-t-s} that for large $t$ depending only on $\epsilon$,  we have 
\begin{equation}\label{eq-ubndexp-01}
  \mathbf{E} [  \L^{(u)}_{t-s}  (xt-X_{s}(u) ) \mid \mathcal{F}_s ]  \leq  \frac{2 }{\sqrt{2 \pi(1-p)}\theta \sqrt{t}}   e^{ at } e^{- \theta\Delta_s(u) }.  
\end{equation}
Recall that $\lambda= \frac{2}{\theta^2}+ \delta$ with $\delta=1/2$.  
 Substituting \eqref{eq-ubndexp-01} in to the expression \eqref{eq-many-jumps} yields that on the event $\{ \mathbf{I}_{s} \geq -\psi'(\kappa)s +z  \}$ we have 
\begin{equation}
  \Sigma_{\ref{eq-many-jumps}}(s)  \lesssim_{\epsilon} \left( \sum_{u \in \mathcal{N}_s} e^{\theta X_{s}(u)-2s}      \right)^{\frac{2}{\theta^2}+\delta} = e^{- (\frac{2}{\theta^2}+\delta) (1-\frac{\theta^2}{2})  s} \left[ W_{s}(\theta)  \right]^{\frac{2}{\theta^2}+\delta} .
\end{equation}
As a consequence, we deduce that  
\begin{align}
  & \mathbf{E} \left[  \Sigma_{\ref{eq-many-jumps}}(s) ;  \mathbf{I}_{s} \geq -\psi'(\kappa) s +z \right]  \lesssim_{\epsilon}     e^{- (\frac{2}{\theta^2}+\delta) (1-\frac{\theta^2}{2})  s} \mathbf{E}\left[     W_{s}(\theta)^{\frac{2}{\theta^2}+\delta} ;  \mathbf{I}_{s} \geq - \psi'(\kappa) s +z \right] \\
  & \leq    C_{\ref{martmin}}(\theta)  e^{- (\frac{2}{\theta^2}+\delta) (1-\frac{\theta^2}{2})  s}  e^{\delta   (1-\frac{\theta^2}{2})s - \delta z   } =      C_{\ref{martmin}}(\theta) e^{- \left( \frac{2}{\theta^2}-1 \right)s- \delta z } \lesssim_{\epsilon} e^{- I(x,a)t- \delta z }.
\end{align} 
Above, in the second line we have used    assertion (ii) of Lemma \ref{martingale limit and minimum}, the facts that  $\psi'(\kappa)=1-\frac{\theta^2}{2}$,  $\theta b=2$, $(\frac{b^2}{2}-1)p=I(x,a)$,  $0 < \inf_{(x,a)\in  \mathcal{A}^{\eqref{eq-admissable-xa}}_{\epsilon} } \theta \leq  \sup_{(x,a)\in  \mathcal{A}^{\eqref{eq-admissable-xa}}_{\epsilon} } \theta  < \sqrt{2}$ 
and that  $ C_{\ref{martmin}}(\theta) $ is continuous on $(0,\sqrt{2})$. We now completes the proof of     \eqref{eq-cliam-2}.
\end{proof}

\subsection{Proof of Lemmas in Section \ref{sec-Up-to-cst}}

 \begin{proof}[Proof of Lemma \ref{lem-rough-bnd} ]
  
   Take arbitrary $t\geq 1$ and  $(x,a) \in \widetilde{\mathcal{A}}^{\eqref{eq-admissable-xa2}}_{\epsilon}$. For each $u \in \mathcal{N}_{t}$ let $W^{(u)}_{r}(\theta)= \sum_{v \in \mathcal{N}^{(u)}_{r}} e^{\theta X^{(u)}_{r}(v) -(1+\theta^2/2)r}$ denote the additive martingale associated to the sub-BBM $X^{(u)}$. Then by taking limit as $r \to \infty$,  we obtain the following decomposition:
   \begin{equation}
    W_{\infty}(\theta)=\sum_{u \in \mathcal{N}_t} e^{\theta X_{t}(u)-\left(1+\frac{\theta^2}{2}\right) t} W_{\infty}^{(u)}(\theta) . 
   \end{equation} 
   By the branching probability, conditioned on $\mathcal{F}_{t}$,  $\{W^{(u)}_{\infty}(\theta): u \in \mathcal{N}_{t} \}$ are i.i.d. with the same law as $W_{\infty}(\theta)$. Thus  conditioned 
   on the event $\{\L_t(x t) \geq e^{a t}/\sqrt{t}\}$ we have 
  \begin{equation}
    W_{\infty}(\theta) \geq_{ \mathrm{st}} \frac{1}{\sqrt{t}} e^{\theta x t-\left(1+\frac{\theta^2}{2}\right) t} \sum_{i=1}^{\exp(a t)}  W_{\infty}^{(i)}.
  \end{equation}
  where $ \geq_{\mathrm{st}}$ refers to ``stochastically dominance", and $W_{\infty}^{(i)}$ are i.i.d. copies of $W_{\infty}(\theta)$. As a result, for any $M \geq 1$,  
    \begin{align}
       &  \mathbf{P}\left(\L_t (x t) \geq \frac{1}{\sqrt{t}} e^{ a t }\right)  \\
     &    \leq \P \left(  W_{\infty}(\theta) \geq \frac{\sqrt{2}-\theta}{ M\sqrt{t}} e^{(a+ \theta x ) t-\left(1+\frac{\theta^2}{2}\right) t}   \right)+ \P \left(   \sum_{i=1}^{\exp(at)}  W_{\infty}^{(i)} \leq \frac{ \sqrt{2}-\theta }{M} e^{at}\right) .  
      \end{align}
      Next, observe that the random variable  $N_{\geq 1}:= \sum_{i=1}^{\exp(at)} \ind{ W_{\infty}^{(i)} \geq \sqrt{2}-\theta} $ obeys a binomial distribution  $ \mathrm{Bin}(\exp(a_{t} t), q_{\theta} ) $, where $q_{\theta}:= \P(W_{\infty}(\theta) \geq \sqrt{2}-\theta)$.  We claim that 
  \begin{equation}\label{eq-def-q}
  \underline{q} := \inf_{\theta \in [0,\sqrt{2})} q_{\theta}     >0 .
  \end{equation}
Indeed, by \cite[Theorem 1.1]{Madaule16}, we have  $\lim_{\theta \uparrow \sqrt{2}}\frac{W_{\infty}(\theta)}{\sqrt{2}-\theta} = 2 D_{\infty}$ in probability, where $D_{\infty}$  is the almost sure limit of derivative martingale and $\P(D_{\infty}>1/2)>0$. Furthermore, by \cite[Corollary 3]{Biggins92}, the mapping $\theta \mapsto W_{\infty}(\theta)$ is almost surely an analytic function on $[0,\sqrt{2})$. Combining these results yields the claim. Now taking $M = 2/\underline{q}$ and applying classical Chernoff bounds for the binomial distribution (Lemma \ref{Chernoff-bounds}), we have 
  \begin{align}
    \P \left(   \sum_{i=1}^{\exp(a t)}  W_{\infty}^{(i)} \leq  \frac{\sqrt{2}-\theta}{M} e^{a t}   \right) \leq  \P \left(  N_{\geq 1}  \leq    e^{a t} \underline{q}/2 \right) 
    \leq \exp\{- \underline{q} e^{a t} /8 \} . 
  \end{align}  
    Applying the inequality \eqref{martail} to $W_{\infty}(\theta)$, we get 
  \begin{equation}
    \P \left(  W_{\infty}(\theta) \geq \frac{\sqrt{2}-\theta}{ M\sqrt{t}} e^{(a+ \theta x ) t-\left(1+\frac{\theta^2}{2}\right) t}   \right) \leq  \frac{ 2^{2/\theta^2}  [1+C_{\ref{martmin}}(\theta)] }{ [(\sqrt{2}-\theta)\underline{q} ]  ^{2/\theta^2}} \, t^{\frac{1}{\theta^2}} e^{- t \left[  \frac{2x}{\theta} -\frac{2(1-a)}{\theta^2}-1 \right] }  \,.
  \end{equation}
Since 
  $ \frac{2x}{\theta} -\frac{2(1-a)}{\theta^2}-1   = \frac{x^2}{2(1-a)}-1= I(x,a) $, 
  for large  $t$ , we obtain that 
  \begin{equation}\label{eq-Z-mid-1}
    \mathbf{P}\left(\L_t (x t) \geq \frac{1}{\sqrt{t}} e^{ a t }\right) \leq   \frac{ 2^{2/\theta^2} [1+C_{\ref{martmin}}(\theta)] }{ [(\sqrt{2}-\theta)\underline{q} ]^{2/\theta^2}   } \,
    t^{ \frac{1}{\theta^2} } e^{-I(x,a) t}  +  e^{- \underline{q} e^{a t} /8} . 
  \end{equation}
  Finally, notice that for any $(x,a) \in  \widetilde{\mathcal{A}}^{\eqref{eq-admissable-xa2}}_{\epsilon}$, we have $
  0< \theta < \sqrt{2} $.\footnote{since $\theta= \sqrt{2}$ implies $a= 1-x/\sqrt{2}$ and  $1-x/\sqrt{2} \geq \epsilon$ leads to $1-x/\sqrt{2} < 1-x^2/2 \leq a$ which is absurd}  Thus there exists a compact interval $J_{\epsilon} \subset (0,\sqrt{2})$ such that $
  \theta \in J_{\epsilon}  $ for all $(x,a) \in \widetilde{\mathcal{A}}^{\eqref{eq-admissable-xa2}}_{\epsilon}$. By Lemma \ref{martmin},  $C_{\ref{martmin}}(\cdot)$ is a continuous function on $ (0,\sqrt{2})$. 
 This yields that  
  \begin{equation}\label{eq-rough-bound-cst}
  C_{\ref{eq-rough-bnd}}(\epsilon):= 2 \sup_{\theta \in J_{\epsilon}} \frac{ 2^{2/\theta^2} [1+C_{\ref{martmin}}(\theta)] }{ [(\sqrt{2}-\theta)\underline{q} ]  ^{2/\theta^2}}   < \infty. 
  \end{equation}
  Combining this with  \eqref{eq-Z-mid-1}, and noting that $e^{at} \geq e^{\epsilon t}$,  we  conclude that for sufficiently large  $t$  (depending only on  $\epsilon$),
  \begin{equation}
    \mathbf{P}\left(\L_t (x t) \geq \frac{1}{\sqrt{t}} e^{ a t }\right) \leq   \frac{1}{2} C_{\ref{eq-rough-bnd}}(\epsilon) \,
    t^{ \frac{1}{\theta^2} }   e^{   - I(a, x ) t  }+ e^{- \underline{q} e^{\epsilon t} /8}  \leq   C_{\ref{eq-rough-bnd}}(\epsilon) \,
    t^{ \frac{1}{\theta^2} } e^{   -I(a, x ) t  }. 
  \end{equation}
 This completes the proof.
\end{proof}

\begin{proof}[Proof of Lemma \ref{Moderate-deviation}] For simplicity, let us denote $ \hat{\beta} := \beta+ \delta$.  According to  \cite[Proposition 1.1]{CDM24},  there exists some constant still denoted by $c_{\beta}$  such that 
  \begin{equation}\label{eq-exp-rate-convergence}
    \E | W_{t}(\beta) -  W_{\infty}(\beta) | \leq e^{-c_{\beta} t} \, \text{ for large } t . 
  \end{equation}
   Set  $r=  \frac{4 }{\beta^2 c_{\beta}}\ell(t)$. 
   By applying inequality  \eqref{eq-exp-condexp-diff} from Lemma  \ref{lem-enhanced-convergence}, we obtain that  
    for large $t$ depending on $\beta$ and $\delta_t$, 
    \begin{align}
      \mathbf{P}\left(  \frac{\L_t (\hat{\beta} t)}{\mathbf{E}[\L_t (\hat{\beta} t)]}  \geq e^y\right)
       \leq    \mathbf{P}\left(     \frac{\mathbf{E} ( \L_t (\hat{\beta} t) | \mathcal{F}_{r} ) }{\mathbf{E}[\L_t (\hat{\beta} t)]}  \geq  e^y - 1  \right) +  4 e^{- 4 \ell(t)^2/\beta^2} .  
    \end{align}
    Let   $K:= \sqrt{2(1+c_{\beta})}$. On the event $\{  M_{r}\leq K r \} $, by the branching property,   
we see that  for large $t$ such that   $ \hat{\beta}  t- K r \geq 1$,   
    \begin{align}
    &\mathbf{E} ( \L_t (\hat{\beta} t) | \mathcal{F}_{r} ) \ind{M_{r} \leq K r} = \sum_{u \in \mathcal{N}_{r}}   \mathbf{E} \left( \L^{(u)}_{t-r}(\hat{\beta} t-X_{r}(u) ) \mid \mathcal{F}_{r} \right)  \ind{M_{r} \leq K r} \\
     & \leq  \sum_{u \in \mathcal{N}_{r}} e^{(t-r)}  \frac{\sqrt{t-r}}{\sqrt{2 \pi} [\hat{\beta} t-X_{r}(u)] } \exp\left\{ - \frac{[\hat{\beta} (t-r)+\hat{\beta} r-X_{r}(u)]^2}{2(t-r)} \right\} \\
     &= \frac{e^{(1-{\hat{\beta}^2 }/{2})t}}{\hat{\beta}  \sqrt{2 \pi t} }\sum_{u \in \mathcal{N}_{r}}   \frac{\sqrt{t(t-r)}}{t-X_{r}(u)/\hat{\beta}  } \, e^{\hat{\beta} X_{r}(u)- \left( \frac{\hat{\beta}^2}{2}+1 \right)r}   e^{  - \frac{(\hat{\beta} r-X_{r}(u))^2}{2(t-r)}} .  
    \end{align}
Notice that for large $t$, we have  $ \frac{e^{(1-\frac{\hat{\beta} ^2}{2})t}}{\hat{\beta}  \sqrt{2 \pi t} } \leq 2 \mathbf{E}[\L_{t}( \hat{\beta} t )]$ and   $ \frac{\sqrt{t(t-r)}}{t-X_{r}(u)/\hat{\beta} }  \leq 2$ when  $ X_{r}(u) \leq K  r$.  
We claim that  $ \hat{\beta} X_{r}(u)- \left( \frac{\hat{\beta}^2}{2}+1 \right)r    - \frac{(\hat{\beta} r-X_{r}(u))^2}{2(t-r)} \leq \beta X_{r}(u)- \left( \frac{\beta^2}{2}+1 \right)r +   \frac{\delta^2 t}{2} $.   Consequently, we get  for large $t$, 
    \begin{equation}\label{eq-dexp-W-1}
      \frac{\mathbf{E} ( \L_t (\hat{\beta} t) | \mathcal{F}_{r} ) }{\mathbf{E}[\L_t (\hat{\beta} t)]} \ind{ M_{r} \leq K r  } \leq 4 e^{\frac{\delta^2 t}{2}} \sum_{u \in \mathcal{N}_{r}}  e^{\beta X_{r}(u)- \left( \frac{\beta^2}{2}+1 \right)r} = 4 e^{\frac{\delta^2 t}{2}}  W_{r}(\beta) . 
    \end{equation}
    To justify the claim, we compute that 
\begin{equation}
 \hat{\beta} X_{r}(u)- \left( \frac{\hat{\beta}^2}{2}+1 \right)r   
  = \beta X_{r}(u)- \left( \frac{\beta^2}{2}+1 \right)r  + \delta(X_{r}(u) - \hat{\beta}  r) + \frac{\delta^2}{2}r  .
\end{equation} 
For each $u \in \mathcal{N}_s$, we have 
\begin{align}
  \delta(X_{r}(u) - \hat{\beta}  r) + \frac{\delta^2 r}{2}  - \frac{(\hat{\beta} r-X_{r}(u))^2}{2(t-r)} \leq \max_{z \in \mathbb{R}} \left\{ \delta z - \frac{z^2}{2(t-r)} \right\}+ \frac{\delta^2 r}{2}  = \frac{\delta^2 t}{2}.
\end{align}
By combining the last two displayed formulas, the claim follows.  Using the inequality \eqref{eq-dexp-W-1} we deduce 
  \begin{equation}  
    \mathbf{P}\left(     \frac{\mathbf{E} ( \L_t (\hat{\beta} t) | \mathcal{F}_{r} ) }{\mathbf{E}[\L_t (\hat{\beta} t)]}  \geq  e^{y}- 1  \right) \leq \P \left( W_r(\beta) \geq \frac{1}{5} e^{y-\frac{\delta^2 t}{2} } \right) + \P\left(  M_{r}> K r   \right).
  \end{equation}
Since $K^2= 2(1+c_{\beta})$, we have  $ \P\left( M_{r} > K r   \right) \leq \exp( (K^2/2-1)r ) \leq e^{-4 \ell(t)/\beta^2}$ for large $t$. Using inequality \eqref{eq-exp-rate-convergence} and    Lemma \ref{martingale limit and minimum}, we find 
  \begin{align}
    \P \left( W_r(\beta) \geq \frac{1}{5} e^{y-\frac{\delta^2 t}{2} }  \right) & \leq  \P \left( | W_r(\beta)  - W_{\infty}(\beta) | \geq 1 \right) +   \P \left(  W_{\infty}(\beta)  \geq   \frac{1}{6} e^{y-\frac{\delta^2 t}{2} } \right) \\
    & \leq \E \left[ | W_{r}(\beta) -  W_{\infty}(\beta) |   \right]+ [1+ C_{\eqref{martmin}}(\beta)] 6^{\frac{2}{\beta^2}} \exp \left\{ -\frac{2}{\beta^2} (y - \frac{\delta^2}{2} t) \right\} . 
  \end{align}
  Combining the above inequalities, we conclude that for large $t$ depending on $\beta$, $\delta_t$ and $\ell(t)$ 
  \begin{equation}
    \mathbf{P}\left(  \frac{\L_t (\hat{\beta} t)}{\mathbf{E}[\L_t (\hat{\beta} t)]}  \geq e^y\right)  \leq  2[1+ C_{\eqref{martmin}}(\beta)] \exp \left\{ -\frac{2}{\beta^2} y + \frac{\delta_{t}^2 t}{\beta^2} \right\}   
  \end{equation}
  for each $\delta^2 t/2 \leq y \leq \ell(t)$, which proves assertion (i).

  Next we prove assertion (ii). 
  Using the same arguments as above, we have 
  \begin{align}
    & \mathbf{P}\left( \frac{ \L_t (\hat{\beta} t)}{\mathbf{E}[\L_t (\hat{\beta} t)]} \geq e^{y} , \mathbf{I}^{\beta} \geq -y+ \frac{\delta^2 t}{2} +z  \right) \\
   & \leq \P \left( W_{\infty} (\beta) \geq \frac{1}{5} e^{y- \frac{\delta^2 t}{2}} , \mathbf{I}^{\beta} \geq - y + \frac{\delta^2 t}{2} +z  \right)  +  6 e^{- 4 \ell (t)/\beta^2} \\
   & \leq  e^{-\left(\frac{2}{\beta^2}+ \frac{1}{2} \right) (y- \frac{\delta^2 t}{2} )}\mathbf{E}\left[ W_{\infty}^{\frac{2}{\beta^2}+\frac{1}{2}}(\beta)  ;   \mathbf{I}^{\beta} \geq -y+ \frac{\delta^2 t}{2} +z  \right]  + 6 e^{- 4 \ell (t)/\beta^2} ,
  \end{align}
where  we have used the Markov inequality. By applying Lemma \ref{martingale limit and minimum}, we conclude that for all  $0 \leq z\leq  y- \frac{\delta^2 t}{2}$,  
\begin{equation}
 e^{-\left(\frac{2}{\beta^2}+ \frac{1}{2} \right)(y- \frac{\delta^2 t}{2}) }\mathbf{E}\left[ W_{\infty}^{\frac{2}{\beta^2}+ \frac{1}{2}}(\beta)  ;   \mathbf{I}^{\beta} \geq -y+ \frac{\delta^2 t}{2} +z   \right]   \leq   C_{\eqref{martmin}}(\beta)  e^{- \frac{2}{\beta^2}(y- \frac{\delta^2 t}{2}) -\frac{1}{2}z } . 
\end{equation}  
This completes the proof.
\end{proof}

%
%
%
%
%
%

\section{Sharp Estimate}
\label{sec-sharp-estimate}
 
\subsection{Proof of Theorem \ref{thm-LDP}: substituting level sets with martingale limits in large deviations}
\label{sharp-estimate-Sec1} 
We introduce some notation first. 
 Fix a large time $t>0$ and    $s=pt$. For each $z \in \mathbb{R}$, define the first passage time of level $-\psi'(\kappa)s+z$  for the process $(V_{r}(u): u \in \mathcal{N}_{r})_{r \ge 0}$ by:  
\begin{equation}
  \tau (z) =\tau(z; t) :=     \inf \left\{ r>0: \ \exists \, u \in \mathcal{N}_{r} \text{ s.t. } V_{r}(u) \leq -\psi'(\kappa)s+z\right\} \label{def-tauz} . 
\end{equation}
Clearly, $\tau(z)$ is a  $[0,\infty]$-valued  stopping time with respect to $(\mathcal{F}_{r})_{r \geq 0}$. 
Additionally, we set  \begin{equation}
  \tau_{t}(z) := \min\{  \tau(z), t\}     
\end{equation}
In Lemma \ref{lem-good-event} we will see that in fact $\P(\tau(z)>t)$ is negligible w.r.t. $\P(\L_{t}(xt) \geq e^{at}/\sqrt{t})$.

\begin{figure}[t]
  \centering
  \includegraphics[height=7cm]{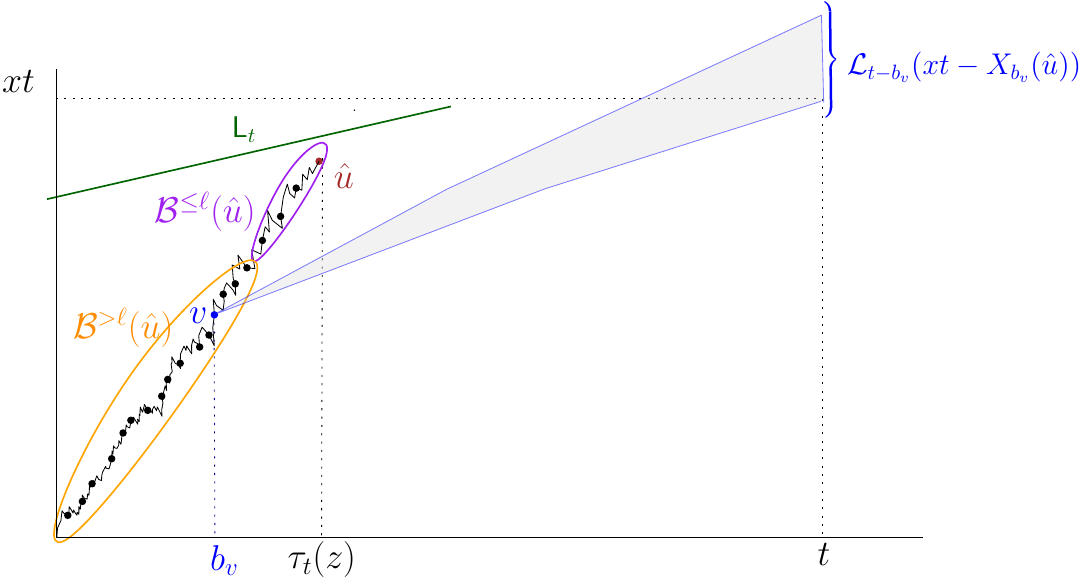} 
  \caption{The sets $\mathcal{B}^{>\ell}(\hat{u})$ and  $\mathcal{B}^{\leq \ell}(\hat{u})$ }\label{fig-Bhatu}
\end{figure}

For any two particles $u$, $v$ in the BBM. We denote  $v \prec u$ if $v$ is an ancestor of $u$; and   $v \preceq u$ if either $v \prec u$ or $v=u$. Define 
\begin{equation}
  \mathcal{B}(u) : = \{ v: \exists \, u' \preceq u \text{ s.t. }  v \text{ is the brother of } u' \} \cup \{u\} .
\end{equation} 
Let $ \hat{u}$  denote the particle alive at time $ \tau_{t}(z)$ that has the lowest position, i.e.,
\begin{equation}
 \hat{u} =  \hat{u}(t,z):=  \argmin_{u \in \mathcal{N}_{\tau_{t}(z) }} V _{\tau_{t}(z) }(u). 
\end{equation} 
Throughout this section, for each $v \in \mathcal{B}(\hat{u})$, we define 
\begin{equation}
  b_{v} = b_{v}( \hat{u}, \mathcal{B}(\hat{u})) := \begin{cases}
    \text{the birth time of } v ,& \text{ if } v \in  \mathcal{B}(\hat{u})\backslash \{\hat{u}\} \\
\tau_{t}(z) ,& \text{ if } v= \hat{u}. 
  \end{cases} 
\end{equation} 
 We divide $\mathcal{B}(\hat{u})$ into two subsets according to their genealogical distance to $\hat{u}$ (see Figure \ref{fig-Bhatu}). For $\ell < \tau_{t}(z)$, define:
   \begin{align}
     \mathcal{B}^{>\ell} (\hat{u}) &:= \{v \in \mathcal{B}(\hat{u}):   b_{v}  < \tau_{t}(z)-\ell \} \ ,  \text{ and }  \\
      \mathcal{B}^{\leq \ell} (\hat{u}) &:=   \{    v \in \mathcal{B}(\hat{u}):   \tau_{t}(z)-\ell    \leq b_{v} \leq \tau_{t}(z)  \}. 
   \end{align}  
In the following, for simplicity, we set  $
\Gamma[t,z] := -\psi'(\kappa) s  + (z-1,z]  \ \text{ for } z \in \mathbb{R}$.  
 
\begin{lemma}\label{lem-good-event}
  For  $\ell>0$, $L \geq 1$, define  
  \begin{align}
   G_{L}^{1} &:= \{   |\tau(z) -s|\leq  L\sqrt{s} \} , \label{eq-good-event-1} \\ 
  G_{\ell, L}^{2}  &:=  \left\{  | \mathcal{B}^{\leq \ell}(\hat{u}) |   \leq L  \right\}  \cap  \left\{ \forall \, r \in [\tau(z) -\ell, \tau(z)] ,  | V_{r}(\hat{u}) - V_{\tau(z)}(\hat{u})| \leq L \right\} . \label{eq-good-event-2}
    \end{align}
    Let $G_{\ell,L}:=  G_{L}^{1} \cap  G_{\ell, L}^{2}  $. 
 Then for fixed $\ell>0$,  $R>0$ we have 
 \begin{equation}\label{lem-tauz>t}
  \lim_{L \to \infty} \limsup_{t \to \infty}  \sup_{ |z| \leq R} e^{I(x,a)t} \, \P \left(   \mathbf{I}   \in \Gamma[t,z] , G_{\ell, L}^{c} \right)  = 0 .  
 \end{equation}
\end{lemma}

 Next, we introduce truncated versions of $\L_{t}[xt,\infty)$ and $W_{\infty}(\theta)$:  For   $ \ell< \tau_{t}(z)   $, define 
\begin{equation}
 \L_{t}^{> \ell} (x t) :=    \sum_{v  \in \mathcal{B}^{>\ell}(\hat{u})} \L^{(v)}_{t-b_{v}}(xt- X_{b_{v}}(\hat{u})) \ \text{ and }  \ W^{> \ell}_{\infty}(\theta) := \sum_{v \in \mathcal{B}^{>\ell}(\hat{u})}  e^{-V_{b_{v}}(\hat{u})} W^{(v)}_{\infty}(\theta). 
\end{equation}

\begin{lemma}\label{lem-truncated-level-set}
For any fixed $R>0$, the following assertions hold:
 \begin{enumerate}[(i)]
  \item  For  any $\epsilon>0$,   uniformly in $|z| \leq R  $  we have 
  \begin{align}
  & \lim_{\ell \to \infty} \limsup_{t \to \infty}  e^{I(x,a)t} \, \P \left(   \L_{t}^{ > \ell} (x t) \geq \frac{\epsilon}{\sqrt{t}} e^{at}  , \mathbf{I} \in \Gamma[t,z] \right) = 0 , \label{eq-distant-relative-level}\\
 & \lim_{\ell \to \infty} \limsup_{t \to \infty}   e^{I(x,a)t} \, \P \left(   W_{\infty}^{ > \ell } (\theta) \geq  \epsilon \, e^{\psi'(\kappa) s}   , \mathbf{I} \in \Gamma[t,z] \right) = 0 .\label{eq-distant-relative-mart}
  \end{align}
  \item  For any fixed $K>0$, uniformly in $|z| \leq R  $, we have 
\begin{equation}\label{eq-distant-relative-min} 
 \lim_{\ell \to \infty}  \limsup_{t \to \infty}  
   e^{I(x,a)t} \,   \P \left( \min_{v \in \mathcal{B}^{>\ell}(\hat{u})} \{V_{b_{v}}(\hat{u}) + \mathbf{I}^{(v)}\} \leq  \mathbf{I}+K ,     \mathbf{I} \in \Gamma[t,z] \right) = 0 . 
\end{equation}
\item  Recall that $ \s$ is defined in \eqref{eq-def-ss}.
  Uniformly in $|z| \leq R  $, we have 
\begin{equation}\label{eq-sstar-tauz-1}
\lim_{K \to \infty}  \limsup_{t \to \infty} e^{I(x,a)t} \P \left(  | \s - \tau_{t} (z)| > K ,   \mathbf{I}  \in  \Gamma[t,z]   \right) = 0 
\end{equation} 
\item Uniformly in $|z| \leq R  $, we have  
\begin{align}
 \lim_{\ell \to \infty}  \limsup_{t \to \infty}  
    e^{I(x,a)t} \,  \P \bigg( \max_{v \in \mathcal{B}^{>\ell}(\hat{u})}   \{ X_{b_{v}}(\hat{u})+ M^{(v)}_{t-b_{v}}\}   
     \geq  X_{\tau_{t}(z)}(\hat{u})  + \sqrt{2}[t-\tau_{t}(z)] ,
    \mathbf{I} \in \Gamma[t,z] \bigg)  =0 .\,\, \, \label{eq-distant-relative-max}
   \end{align}
 \end{enumerate}  
\end{lemma}

Naturally we define 
\begin{align}
   \L_{t}^{ \leq \ell} (x t)  &:=  \L_{t}  (x t)- \L_{t}^{ > \ell} (x t)=  \sum_{v \in \mathcal{B}^{\leq \ell}(\hat{u})} \L^{(v)}_{t-b_{v}}(xt- X_{b_{v}}(\hat{u})), \\
  W_{\infty}^{ \leq \ell}(\theta) &  :=  W_{\infty}(\theta) -W^{> \ell}_{\infty}(\theta) =    \sum_{v \in \mathcal{B}^{\leq \ell}(\hat{u})}   e^{-V(v)} W^{(v)}_{\infty}(\theta) . 
\end{align}
In the following, we use the standard notation  $A \sdif B:= (A \backslash B) \cup (B \backslash A)$.

\begin{lemma}\label{lem-replace-level-by-mart}
Let $c_{\theta}:= \theta \sqrt{2\pi(1-p)}$. Define  
\begin{equation}
  C_{t}(z)=  \exp \left(\frac{p}{1-p}   \frac{(1-{\theta^2}/{2})^2    }{2\theta^2  } \frac{(\tau_t(z)-s)^2}{s}   \right) .
\end{equation} 
  Fix $y>0$, $\ell>0$ and $R>0$. 
  For any $\epsilon>0$,  
  uniformly in $|z| \leq R$ and $|\lambda_{i}-1| < \epsilon$, $i=1,2$ we have   
  \begin{equation}
   \lim_{\epsilon \downarrow 0} \limsup_{t \to\infty} e^{I(x,a)t}      \P \left(  \left\{ \L_{t}^{\leq \ell}  (x t) \geq \frac{\lambda_{1} y \, e^{at}}{ \sqrt{t}}  \right\} \sdif  \left\{   \frac{W^{\leq \ell}_{\infty}(\theta)}{C_{t}(z)}\geq   \lambda_{2} c_{\theta}  
   y e^{\psi'(\kappa) s}  \right\},  \mathbf{I}  \in  \Gamma[t,z] \right) = 0. 
  \end{equation}
\end{lemma} 
Based on Lemma \ref{lem-truncated-level-set} and Lemma  \ref{lem-replace-level-by-mart}, now we are ready to prove Theorem \ref{thm-LDP}.  

\begin{proof}[Proof of Theorem  \ref{thm-LDP}]  
  For $t>0$ and $z \in \mathbb{R}$ we define 
  \begin{equation}\label{eq-normalized-prob-L-I}
    \bar{\mathrm{Pr}}_{\eqref{eq-normalized-prob-L-I}}(t,z) :=  e^{I(x,a)t} \P \left(  \L_{t}  (x t) \geq \frac{y}{\sqrt{t}} e^{at} ,  
    \mathbf{I}  \in  \Gamma[t,z] \right) .
  \end{equation}
  Fix an arbitrary small $\delta>0$.  According to Proposition \ref{prop-up-to-constant},    there exists large constant $R_{\delta}>0$ satisfying that for all $R \geq R_{\delta}$, 
\begin{equation} 
   \limsup_{t \to \infty}  \left| e^{I(x,a)t}\P \left(  \L_{t}  (x t) \geq \frac{y}{\sqrt{t}} e^{at}  \right) - \sum_{|z| \leq R} \bar{\mathrm{Pr}}_{\eqref{eq-normalized-prob-L-I}}(t,z)  \right|  < \delta . \label{eq-apro-0}
\end{equation} 
Similar to $C_{t}(z)$, we define $\mathcal{C}_t:=  \exp \left(\frac{p}{1-p}   \frac{(1-{\theta^2}/{2})^2    }{2\theta^2  } \frac{(\mathfrak{s}-s)^2}{s}   \right)$ and claim that 
for each fixed $R>0$ and $|z|\leq R$,  as $t \to \infty$
\begin{equation}
  \begin{aligned} \label{eq-level-set-min}
  \mathrm{Er}_{\eqref{eq-level-set-min}}(t,z)
  :=e^{I(x,a)t}    \mathbf{P} \left(  \left\{ \L_{t}  (x t) \geq \frac{y\, e^{at}}{\sqrt{t}}  \right\} \sdif  \left\{   \frac{W_{\infty}(\theta) }{\mathcal{C}_{t}}\geq c_{\theta}  y e^{\psi'(\kappa) s}  \right\} ,  \mathbf{I} \in \Gamma[t,z]   \right)  \to 0. 
  \end{aligned}
\end{equation} 
Then it follows from the assertion \eqref{eq-level-set-min} and  the fact that $|\P(A)-\P(B)| \leq \P(A \sdif B)$ that  uniformly in $|z|\leq R$ we have 
\begin{align}
\lim_{t \to \infty}  \bar{\mathrm{Pr}}_{\eqref{eq-normalized-prob-L-I}}(t,z)  
  = \lim_{t \to \infty} e^{I(x,a)t}    \mathbf{P} \left( W_{\infty}(\theta) \geq c_{\theta}y \, \mathcal{C}_{t} e^{\psi'(\kappa)s}, \mathbf{I} \in - \psi'(\kappa)s+ (z-1,z] \right). 
\end{align}     
Writing $\Gamma_{+}[t,z]:= z-\psi'(\kappa)s$ the right endpoint of $\Gamma[t,z]$, by use of Lemma \ref{BRWCondMin}, we get   
\begin{equation}
  \lim_{t \to \infty} e^{I(x,a)t}\P\left( \mathbf{I} \leq \Gamma_{+}[t,z]  \right) = \lim_{t \to \infty} c_{\mathbf{I}} e^{I(x,a)t-\kappa \psi'(\kappa) s   }   e^{\kappa z}  =  C_{\mathbf{I}} \, e^{\kappa z} . 
\end{equation} 
Thus to compute $\lim_{t \to \infty}  \bar{\mathrm{Pr}}_{\eqref{eq-normalized-prob-L-I}}(t,z)  $, it is sufficient to study  
\begin{align}
 & \lim_{t \to \infty} \P \left(   W_{\infty}(\theta) /\mathcal{C}_{t} \geq  c_{\theta}   y \, e^{\psi'(\kappa) s} , \Gamma_{+}[t,z] - \mathbf{I}\in [0,1)\mid \mathbf{I} \leq \Gamma_{+}[t,z] \right) \\
  &=   \lim_{t \to \infty} \P \left(  e^{\mathbf{I}} W_{\infty}(\theta) /\mathcal{C}_{t} \geq  c_{\theta}  y e^{z} e^{\mathbf{I}-\Gamma_+[t,z]} , \Gamma_{+}[t,z] - \mathbf{I}\in [0,1)\mid \mathbf{I} \leq \Gamma_{+}[t,z] \right) \\
  &=    \int_{0}^{1}  \kappa \P \left(  \mathcal{Z} /\mathcal{C}\geq  c_{\theta}   y \, e^{z-h}  \right) e^{-\kappa h}  \dif h . \label{eq-apro-W-limit-C}
\end{align}
Above, in the last equality we have used Lemma \ref{BRWCondMin} again: Conditionally on $ \mathbf{I} \leq \Gamma_{+}[t,z]$, $( e^{\mathbf{I}} W_{\infty}(\theta) , \mathcal{C}_{t}, \Gamma_+[t,z]- \mathbf{I}) $ converges weakly to $(\mathcal{Z}, \mathcal{C},\mathcal{E})$ as $t \to \infty$, where  
 $\mathcal{C}= \exp (  \frac{\mathcal{N}^2}{2} \frac{p}{1-p})$ with $\mathcal{N}$ following a standard Gaussian distribution, $\mathcal{E}$ has exponential distribution with mean $\kappa^{-1}$ and $(\mathcal{Z},\mathcal{N},\mathcal{E})$ are independent. Now we can conclude that 
\begin{align}
  \lim_{t \to \infty}  \bar{\mathrm{Pr}}_{\eqref{eq-normalized-prob-L-I}}(t,z)  
  & =  \kappa c_{\mathbf{I}}   \int_{0}^{1}   \P \left(\mathcal{Z}/{\mathcal{C}} \geq  c_{\theta}    y \, e^{z-h}  \right) e^{\kappa (z-h) }  \dif h \\
  & =  \kappa c_{\mathbf{I}} \int_{e^{z-1}}^{e^{z}}   \P \left(\mathcal{Z}/{\mathcal{C}} \geq  c_{\theta}    y \, \gamma  \right) \gamma^{\kappa-1} \dif \gamma .\label{eq-apro-W-limit}
\end{align}
Notice that  by Lemma \ref{BRWCondMin} we have $C_{\mathbf{I}} \mathbf{E}[\mathcal{Z}^{\kappa}]= C_{W_{\infty}(\theta)}$. We thus obtain that 
\begin{equation}
  \kappa c_{\mathbf{I}}  \sum_{|z| \leq R}  \int_{e^{z-1}}^{e^{z}}   \P \left( {\mathcal{Z}}/{\mathcal{C}} \geq  c_{\theta}   y  \gamma  \right)  \gamma^{\kappa-1} \dif \gamma   
     \overset{R \to \infty}{\longrightarrow}     \frac{ c_{\mathbf{I}}}{(c_{\theta} y )^{\kappa}} \E \left[ \mathcal{C}^{-\kappa} \mathcal{Z}^{\kappa} \right] = \frac{   C_{W_{\infty}(\theta)} }{(c_{\theta} y )^{\kappa}} \E\left[ \mathcal{C}^{-\kappa}\right] . 
\end{equation} 
We compute that $ \E\left[ \mathcal{C}^{-\kappa}\right]=\E[ \exp(-  \frac{\kappa p}{2(1-p)}  \mathcal{N}^2) ]= \sqrt{\frac{1-p}{1+(\kappa-1)p}} $.
So letting $R \to \infty$ in \eqref{eq-apro-0}, by combining all computations, we arrive at:
\begin{align}
  \E \left[ \mathcal{C}^{-\kappa} \right]\frac{   C_{W_{\infty}(\theta)} }{(c_{\theta} y )^{\kappa}}  -   \delta 
  & \leq    \liminf_{t \to \infty}   e^{I(x,a)t}\P \left(  \L_{t}  (x t) \geq \frac{y}{\sqrt{t}} e^{at}  \right) \\
    & \leq 
  \limsup_{t \to \infty}      e^{I(x,a)t}\P \left(  \L_{t}  (x t) \geq \frac{y}{\sqrt{t}} e^{at}  \right)  \leq  \frac{   C_{W_{\infty}(\theta)} }{(c_{\theta} y )^{\kappa}}   \E \left[ \mathcal{C}^{-\kappa} \right]   +   \delta.
\end{align} 
By sending $\delta \downarrow 0$, we get  Theorem \ref{thm-LDP}.  Now it remains to prove the claim \eqref{eq-level-set-min}. 

 Let us fix an arbitrary small $\epsilon \in (0,1)$ and an arbitrary  small $\delta>0$. Let 
 \begin{equation}
  E^{(0)}_{\ell,\epsilon} := \left\{  \L^{>\ell}_{t}  (x t) \leq {\epsilon \, ye^{at}}/{\sqrt{t}}  ,  {W^{>\ell}_{\infty}(\theta)}/{C_{t}(z)} \leq \epsilon c_{\theta}    y \, e^{\psi'(\kappa) s}  \right\} .
 \end{equation}
Since $C_{t}(z) \geq 1$, it follows from Lemma \ref{lem-good-event} and Lemma  \ref{lem-truncated-level-set} (i) that  for sufficiently large $L$ and  $\ell$ (depending on $\epsilon ,\delta, R$) and for sufficiently large $t$, the following holds
\begin{equation}
     e^{I(x,a)t}  \,  \P \left(  \left( G^{1}_{L} \cap E^{(0)}_{\ell,\epsilon}\right)^{c} , 
   \mathbf{I}  \in  \Gamma[t,z] \right) < \delta , \,\forall \,|z| \leq R.
\end{equation} 
Let $E_{\ell, \epsilon}:=  \{|C_{t}(z)/\mathcal{C}_{t}-1|\leq \epsilon\}\cap E^{(0)}_{\ell,\epsilon}$. 
Comparing the definition of $\mathcal{C}_{t}$ with that of $C_{t}(z)$ and applying Lemma \ref{lem-truncated-level-set} (iii) we get that for sufficiently large $t$, 
\begin{equation}
  e^{I(x,a)t}  \,  \P \left(   \left( G^{1}_{L} \cap E_{\ell,\epsilon}\right)^{c}, 
\mathbf{I}  \in  \Gamma[t,z]  \right) < 2 \delta , \,\forall \,|z| \leq R.
\end{equation} 
When $ E_{\ell,\epsilon}$ occurs, if  $\L_{t}   (x t) \geq \frac{y\, e^{at}}{\sqrt{t}} $ but $  W_{\infty}(\theta) \leq c_{\theta} y \mathcal{C}_{t} e^{\psi'(\kappa) s} $, we must have  $ \L^{>\ell}_{t}  (x t) \geq \frac{(1-\epsilon) y}{\sqrt{t}} e^{at} $  and  $  W_{\infty}^{>\ell}(\theta) \leq c_{\theta} y (1+\epsilon) C_{t}(z) e^{\psi'(\kappa) s} $. 
Another case is similar. Therefore we conclude that 
\begin{align}
   & \mathrm{Er}_{\eqref{eq-level-set-min}}(t,z) \leq e^{I(x,a)t}  \,  \P \left(   \left( G^{1}_{L} \cap E_{\ell,\epsilon}\right)^{c}, 
  \mathbf{I}  \in  \Gamma[t,z]  \right) \\
 &   \quad   + e^{I(x,a)t}  \mathbf{P} \left(  \left\{ \L^{\leq \ell}_{t}  (x t) \geq \frac{(1-\epsilon) y}{\sqrt{t}} e^{at} \right\} \cap  \left\{   \frac{W_{\infty}^{\leq \ell} (\theta)}{C_{t}(z)} \leq c_{\theta} (1+\epsilon) y e^{\psi'(\kappa) s}  \right\} ,  \mathbf{I} \in  \Gamma[t,z]   \right)  \\
 & \quad +   e^{I(x,a)t}  \mathbf{P} \left(  \left\{ \L^{\leq \ell}_{t}  (x t) \leq \frac{y}{\sqrt{t}} e^{at} \right\} \cap \left\{       \frac{W_{\infty}^{\leq \ell} (\theta)}{C_{t}(z)} \geq c_{\theta} (1-2\epsilon) y e^{\psi'(\kappa) s}  \right\} ,  \mathbf{I} \in   \Gamma[t,z]   \right)   .  
\end{align}
Now by use of  Lemma \ref{lem-replace-level-by-mart},  taking the limit as $t \to \infty$ followed by  $\epsilon \downarrow 0$, we obtain that 
  $ \limsup_{t \to \infty} \mathrm{Er}_{\eqref{eq-level-set-min}}(t,z) \leq \delta $ uniformly  in $z \in [-R,R]$.  Since $\delta$ is arbitrary, the desired result \eqref{eq-level-set-min} follows. This completes the proof. 
\end{proof}

\subsection{Change of measure} In this section, we apply the change of measure technique  outlined in \S \ref{sec-spine} with the additive martingale with parameter $b=2/\theta$
\begin{equation}
  W_{r}(b) = \sum_{u \in \mathcal{N}_{r }} e^{ b X_{r}(u)- \left( \frac{b^2}{2}+1 \right)r  } = \sum_{u \in\mathcal{N}_{r }} e^{-\kappa V_{r}(u)} , \, r > 0 . 
\end{equation}
We begin by introducing a stopping time for  the BBM with a spine ($\{X_{r}(u): u\in \mathcal{N}_{r}\} , w_{r})_{r \geq 0}$ defined by  
 \begin{equation}
  \tau^{w}(z) := \inf\{ r > 0: V_{r}(w) = -\psi'(\kappa) s + z  \} . 
 \end{equation}
 Denote by $\hat{w}$ the spine particle at the stopping time $\tau^{w}(z)$ when it is finite, i.e., $\hat{w}:=w_{\tau^{w}(z)}$.

\begin{lemma}\label{lem-cop33}
For any $t>1$, $ r \geq 1$, $z \in \mathbb{R}$, and $A \in \mathcal{F}_{\infty}$,   
  \begin{equation}
    e^{ I(x,a)t} \P \left(  A ,  \tau(z)< r
    \right) =  e^{\kappa z}   (\mathbf{Q}^{b,*}_{\tau^{w}(z)}  \otimes \P) \left(    A ,  \tau^{w}(z)=\tau (z) < r  \right). 
  \end{equation}
\end{lemma}

\begin{proof} 
  Consider the measure  $ \mathbf{Q}^{b,*}_{\tau_{r}(z)} \otimes \P $ defined via the stopping time $\tau_{r}(z)=\tau(z) \wedge r$.  From the definition of the density \eqref{eq-density-Qtau}, we have
\begin{equation}
  (\mathbf{Q}^{b,*}_{\tau_{r}(z)} \otimes \P) ( \cdot \cap\{ \tau_{r}(z)=\tau^w(z)  \}) =  (\mathbf{Q}^{b,*}_{\tau^{w}(z)}\otimes \P) ( \cdot \cap \{ \tau_{r}(z)=\tau^w(z) \}) .
\end{equation}
Note that, almost surely, if $\tau(z)<\infty$, exactly one particle of the BBM will reach the line $-\psi'(\kappa)s + z$ at time $\tau(z)$. In other words, we have
 \begin{equation}
  \ind{\tau(z)< r} = \sum_{u \in \mathcal{N}_{\tau(z)}} \ind{\tau^u(z)=\tau(z)< r} \ \text{ a.s.}
 \end{equation}
 where $\tau^u(z):= \inf\{0 < t < d_{u}: V_{t}(u) =  -\psi'(s)+z\}$ (recall that $d_u$ denotes the death time of $u$). 
By using \eqref{eq-cop-11} from Lemma \ref{changeofp}, we deduce that 
\begin{equation}
  \P \left(  A ,  \tau(z)< r
  \right) =   \E_{ \mathbf{Q}^{b,*}_{\tau_{r}(z)} \otimes \P} \left[   \sum_{u \in \tau_{r}(z) }  \frac{e^{-\kappa V_{\tau_{r}(z) }(u)}}{W_{\tau_{r}(z)}(b)} e^{\kappa V_{\tau_{r}(z) }(u)} \ind{\tau^u(z)=\tau(z)< r } \ind{A}   \right]. 
\end{equation} 
Applying \eqref{eq-cop-22} of Lemma \ref{changeofp} to the inequality above,  we obtain that 
\begin{align}
\P \left(  A ,  \tau(z)< r
  \right) 
  & = \E_{ \mathbf{Q}^{b,*}_{\tau_{r}(z)} \otimes \P } \left[  e^{\kappa V_{\tau^{w}(z)}(w)}    \ind{ \tau^{w}(z)=\tau (z) < r}    \ind{A} \right] \\
  & = e^{-I(x,a)t + \kappa z}   (\mathbf{Q}^{b,*}_{\tau_{r}(z)} \otimes \P)   \left(    \tau^{w}(z)=\tau(z) < r, A  \right).
\end{align} 
Above, we have used the fact that $ V_{\tau^{w}(z)}(w)= - \psi'(\kappa)s + z$  and  $\kappa \psi'(\kappa) s=I(x,a)t$. This completes the proof.
\end{proof}

In the following, we will omit the parameter  $b$  and write \( \mathbf{Q}^{b,*}_{\tau^{w}(z)} \otimes \P \) as \( \mathbf{Q}^{*}_{\tau^{w}(z)} \otimes \P \), since only this parameter is used.
Note that on the event $\{\tau^{w}(z)=\tau(z)<t\}$ we have $\hat{u}= \hat{w}$. This observation leads us to  an immediate and useful corollary of Lemma \ref{lem-cop33}.

\begin{corollary}For $L \geq 1$,  define 
  \begin{equation} 
   G_{L}^{1}:= \{   |\tau(z) -s|\leq  L\sqrt{s} \} ,\ L \geq 1.
  \end{equation}
Then for any $t>0$, $z \in \mathbb{R}$, and $A \in \mathcal{F}_{\infty}$, we have   
\begin{equation}\label{eq-change-of-measure}
  e^{ I(x,a)t} \P \left( A, G^{1}_{L} , \mathbf{I} \in \Gamma[t,z]  
  \right) =  e^{\kappa z}   (\mathbf{Q}^{*}_{\tau^{w}(z)} \otimes \P)    \left(   A ,    G^{1}_{L}(w) , I_0^{w}   \right). 
\end{equation} 
where  $G^{1}_{L}(w):= \{ |\tau^w(z)-s| \leq L \sqrt{s} \}$ and 
\begin{align}
I_0^{w} &:= \{   \mathbf{I} + \psi'(\kappa) s \in  [z-1,z] \}\cap \{ \tau^{w}(z)=\tau(z) \} \\
&=  \{ \tau^{w}(z)=\tau(z) \} \cap \bigcap_{v \in \mathcal{B}(\hat{w})}\{ \mathbf{I}^{(v)} > V_{\tau^{w}(z)}(w)-1-V_{b_{v}}(w) \}  .
\end{align} 
\end{corollary}

\subsection{Proofs of Lemmas in Section \ref{sharp-estimate-Sec1}}

To prove Lemma \ref{lem-good-event}, we need to use a well-known fact about the first passage times of Brownian motion:

\begin{lemma}\label{lem-BMhitting} 
Let   $(\beta_t= \nu t - \sigma B_{t}: t \geq 0)$ where $B_{t}$ is a standard Brownian motion and $\nu, \sigma>0$. Then the first passage time $ T(\alpha) :=\inf \left\{t>0 : \beta_t=\alpha\right\}$ for a fixed level $\alpha>0$ by $\beta_t$ is distributed according to an inverse Gaussian $(\frac{\alpha}{\nu}, \frac{\alpha^2}{\sigma^2})$. That is,  $T(\alpha)$ has  p.d.f. 
\begin{equation}
  f_{\frac{\alpha}{\nu},\frac{\alpha^2}{\sigma^2}}(T):=   \frac{\alpha}{  \sigma\sqrt{2 \pi T^3}} \exp \left(-\frac{(\alpha-\nu T)^2}{2 \sigma^2  T}\right) , \  T>0.
\end{equation}   
  \end{lemma}

\begin{proof}[Proof of Lemma \ref{lem-good-event}]
  We begin by asserting that
  \begin{equation}\label{eq-G-estimate}
  \mathrm{Pr}_{\eqref{eq-G-estimate}}(L) := \limsup_{t \to \infty}    \sup_{ |z|\leq R}    e^{I(x,a)t} \, \P \left(  \mathbf{I} \in \Gamma[t,z],  ( G^{1}_{L} )^{c} \right) \lesssim_{R,x,a}  e^{- \frac{\psi'(\kappa)^2 }{2 \psi''(\kappa) } L^2}. 
  \end{equation}  
  In fact, since  the event $   \mathbf{I} \in \Gamma[t,z]$ implies that $\tau(z)<\infty$, it follows from  Lemma \ref{lem-cop33}  that 
\begin{equation}
  e^{I(x,a)t} \, \P \left(  \mathbf{I} \in \Gamma[t,z],   |\tau(z) - s| > L\sqrt{s} \right)     \lesssim_{R}  (\mathbf{Q}^{*}_{\tau^{w}(z)} \otimes \P)  \left(   |\tau^w(z) -  s|> L \sqrt{s} , \tau^w(z) <\infty  \right) .
\end{equation}
Under  $\mathbf{Q}^{*}_{\tau^{w}(z)} \otimes \P$, $(V_{r}(w): r \leq \tau^w(z))$ is a Brownian motion with drift $-\psi'(\kappa)$ and diffusion coefficient $\psi''(\kappa)$  stopped upon hitting the line $-\psi'(\kappa) s +z$. By applying Lemma \ref{lem-BMhitting} we obtain   
\begin{align}
  & e^{I(x,a)t} \, \P \left(  \mathbf{I} \in \Gamma[t,z],   |\tau(z) - s| > L\sqrt{s} \right) \\
  & \lesssim_{R}   \int_{\mathbb{R}_+ \backslash [s-L\sqrt{s},s+L\sqrt{s}] }   \frac{\psi'(\kappa) s-z}{  \sqrt{2 \pi \psi''(\kappa) T^3}} \exp \left(-\frac{(\psi'(\kappa)s-z-\psi'(\kappa) T)^2}{2 \psi''(\kappa) T}\right) \dif T \\
   &\lesssim_{x,a}  \int_{(-\sqrt{s}, \infty) \backslash [-L,L] } \sqrt{s}\, \frac{\psi'(\kappa) s-z } {  \sqrt{ (s+\lambda \sqrt{s}) ^3}} \exp \left(- \frac{1}{2 \psi''(\kappa) } \frac{[\psi'(\kappa)\lambda   +z/\sqrt{s} ]^2}{1+ \lambda/ \sqrt{s}} \right)  \dif \lambda .
\end{align}  
Above, we have made a change of variable $T=s+\lambda \sqrt{s}$. 
By employing the dominated convergence theorem, we conclude that
\begin{align}
  \mathrm{Pr}_{\eqref{eq-G-estimate}}(L) \lesssim_{R, x,a}   \int_{L}^{\infty}  e^{ - \frac{\psi'(\kappa)^2 }{2 \psi''(\kappa) }\lambda^2  } \dif \lambda 
  \lesssim_{x,a} e^{- \frac{\psi'(\kappa)^2 }{2 \psi''(\kappa) }L^2}. 
\end{align} 
Next, by applying \eqref{eq-change-of-measure} and  observing that $\hat{u}=\hat{w}$ on $G^{1}_{L}(w) \cap I_{0}^{w}$, we have 
\begin{equation}
  e^{I(x,a)t} \, \P \left(    \mathbf{I} \in \Gamma[t,z],   G^{1}_{L} \backslash G^{2}_{\ell,L} \right) \leq  e^{\kappa R} ( \mathbf{Q}^{*}_{\tau^{w}(z)} \otimes \P ) \left(    G^{1}_{L}(w) \backslash G^{2}_{\ell, L}(w)   \right), \label{eq-G-estimate-2}
\end{equation}
where $G_{\ell, L}^{2}(w)  :=  \left\{  | \mathcal{B}^{\leq \ell}(\hat{w}) |   \leq L  \right\}  \cap  \left\{ \forall \, r \in [\tau^{w}(z) -\ell, \tau^{w}(z)] ,  | V_{r}(w) - V_{\tau^{w}(z)}(w)| \leq L \right\} $. First, observe that  for a fixed $\ell$, the law of the random variable $\max\{ | V_{r}(w) - V_{\tau^{w}(z)}(w)| : r \in [\tau^{w}(z) -\ell, \tau^{w}(z)] \} $ does not depend on $t,z$ and it is tight. Second, conditioned on $\tau^{w}(z)$ and $\{V_{r}(w): r \leq \tau^{w}(z)\}$, times of branching events along the spine   
$(b_{v}: v \in \mathcal{B}(\hat{w})\backslash\{\hat{w}\})$ form a Poisson process with rate Poisson process with rate $2$  on $[0,\tau^{w}(z)]$. Thus the   number of branching points in the time interval  $[\tau^{w}(z)-\ell,\tau^{w}(z)]$ does not depend on $t$ or $z$ and is also tight. Hence, for every $\ell \geq 1$, uniformly in $z$, we have
\begin{equation}
   \lim_{L \to \infty}\limsup_{t \to \infty}     \sup_{ |z| \leq R}   (\mathbf{Q}^{*}_{\tau^{w}(z)} \otimes \P)  \left(    G^{1}_{L}(w) \backslash G^{2}_{\ell, L}(w)   \right) =0. 
\end{equation}   
This, together with \eqref{eq-G-estimate} and \eqref{eq-G-estimate-2}, implies the desired result.  
\end{proof}

\begin{proof}[Proof of Lemma \ref{lem-truncated-level-set}] 
  We  provide the proof of \eqref{eq-distant-relative-level},   \eqref{eq-sstar-tauz-1} and \eqref{eq-distant-relative-max} here. A similar argument can be applied to prove \eqref{eq-distant-relative-mart} and \eqref{eq-distant-relative-min} (alternatively, see \cite[Lemmas 3.7 and 3.3]{CDM24} for a proof in the context of branching random walks). 

  \underline{Proof of \eqref{eq-distant-relative-level}.} Thanks to Lemma \ref{lem-good-event}, it is sufficient  to show that for each fixed $L \geq 1 $, uniformly in $|z| \leq R$,  we have 
 \begin{equation}
 \lim_{\ell \to \infty} \limsup_{t \to \infty} e^{I(x,a)t} \, \P \left(   \L_{t}^{ > \ell} (x t) \geq \frac{\epsilon}{\sqrt{t}} e^{at}  , \mathbf{I}  \in  \Gamma[t,z] , G_{L}^{1} \right)   =0. 
 \end{equation} 
 Applying \eqref{eq-change-of-measure} with $A=\{\L_{t}^{ > \ell} (x t) \geq \frac{\epsilon}{\sqrt{t}} e^{at} \}$, we only need to prove that 
 \begin{align}
   \lim_{\ell \to \infty} \limsup_{t \to \infty}     ( \mathbf{Q}^{*}_{\tau^w(z)} \otimes\P)  \left(     \L_{t}^{> \ell}  (x t) \geq \frac{\epsilon}{\sqrt{t}} e^{at} , I_{0}^{w},  G^{1}_{L}(w) \right)   = 0 . 
 \end{align}
On the event $ I_{0}^{w} \cap G^{1}_{L}(w) $, since $\hat{u}=\hat{w}$,  we can rewrite  $\L_{t}^{> \ell}  (x t) $ as $ \sum_{v  \in \mathcal{B}^{>\ell}(\hat{w})} \L^{(v)}_{t-b_{v}}(xt- X_{b_{v}}(\hat{w}))  $. So it suffices to show that for any fixed   $L\geq 1$,   uniformly in $|z| \leq R$,
 \begin{equation}\label{eq-distant-relative-2}
   \lim_{\ell \to \infty}\limsup_{t \to \infty}  \underbrace{  ( \mathbf{Q}^{*}_{\tau^w(z)} \otimes\P) 
    \left( \sum_{v  \in \mathcal{B}^{>\ell}(\hat{w})} \L^{(v)}_{t-b_{v}}(xt- X_{b_{v}}(w))   \geq \frac{\epsilon}{\sqrt{t}} e^{at} , I_{0}^{w},  G^{1}_{L}(w) \right) }_{ Q_{\eqref{eq-distant-relative-2}}(t,\ell) }  = 0 .
 \end{equation}

 We claim that on $ I_{0}^{w} \cap G^{1}_{L}(w) $,   $ xt-X_{r}(w) > 0$   for all $r \leq \tau^w(z)$.  In fact, by definition  for every $r \leq \tau^w(z)$, we have $V_{r}(w) \geq - \psi'(\kappa)s+z$, which implies    $\theta X_{r}(w) \leq (\frac{\theta^2}{2}+1)r+\psi'(\kappa)s-z  \leq 2s+O(\sqrt{s})+R$. 
   Thus, we get  $X_{r}(w) \leq bs+O(\sqrt{s})+O(1)< xt$.
 Next, let  $\mathcal{G}^{*}_{\tau^w(z)}:= \sigma( V_{r}(w): r \leq  \tau^w(z),  b_{v} ,  v \in \mathcal{B}(\hat{w}))$ denote the $\sigma$-field that contains the information about the movement and branching times of the spine particles up to time  $\tau^w(z)$.  
 Conditionally on  $\mathcal{G}^{*}_{\tau^w(z)}$, by using the branching property, the many-to-one formula, and the Gaussian tail inequality that $\P(N(0,1) >y) \leq \frac{1}{y\sqrt{2\pi} }e^{-y^2/2}, \forall\,y >0$,  we obtain,
  on the event  $G^{1}_{L}(w) \in \mathcal{G}^{*}_{\tau^w(z)}$, 
 \begin{align}
  & \E_{\Q^{*}_{\tau^w(z)} \otimes \P } \left(      \sum_{v  \in \mathcal{B}^{>\ell}(\hat{w})} \L^{(v)}_{t-b_{v}}(xt- X_{b_{v}}(w))   \mid  \mathcal{G}^{*}_{\tau^w(z)}  \right) \\
  &\lesssim   \sum_{ b_{v} \leq \tau^w(z) - \ell} \,   \frac{\sqrt{t-b_{v}}}{xt-X_{b_{v}}(w) }  \exp \left\{ (t-b_{v})-\frac{[xt-\mathsf{L}_{t}(b_{v})+\mathsf{L}_{t}(b_{v})- X_{b_{v}}(w)]^2}{2(t-b_{v})} \right\} \\
  & \lesssim_{x,a} \frac{1}{\sqrt{t}} e^{at}  \sum_{ b_{v} \leq \tau^w(z) - \ell} \, \exp \left\{ -\frac{xt-\mathsf{L}_{t}(b_{v})}{t-b_{v}}[\mathsf{L}_{t}(b_{v})-X_{b_{v}}(w)] \right\} \label{eq-distant-relative-232} .
 \end{align}
 Above in the last inequality, we have   used the fact that   (see  \eqref{eq-AHScurve-2} for the property of $\mathsf{F}_{t}$) 
 \begin{equation}
  t-b_{v} - \frac{[xt-\mathsf{L}_{t}(b_{v})]^2}{2(t-b_{v})} \leq  (t-b_{v}) - \frac{[xt-\mathsf{F}_{t}(b_{v})]^2}{2(t-b_{v})} = at, 
 \end{equation}
as $\mathsf{L}_{t}(r) \leq \mathsf{F}_{t}(r) \ll x t $ for all $r \leq \tau^w(z) \leq s + L \sqrt{s}$. 

We denote the probability in \eqref{eq-distant-relative-2} by $Q_{\eqref{eq-distant-relative-2}}(t,\ell)$.
 By \eqref{eq-distant-relative-232} and  applying Markov's inequality,  we obtain, 
 \begin{align}
   Q_{\eqref{eq-distant-relative-2}}(t,\ell)    \lesssim_{x,a} \frac{1}{\epsilon} \, \E_{\mathbf{Q}^{*}_{\tau^w(z)}\otimes\P } \left[  \min\left\{ \epsilon , \sum_{ b_{v} \leq \tau^w(z) - \ell} \, e^{ -\frac{xt-\mathsf{L}_{t}(b_{v})}{t-b_{v}}[\mathsf{L}_{t}(b_{v})-X_{b_{v}}(w)] } \right\}  \ind{  G^{1}_{L}(w) } \right].
 \end{align}

 We choose an arbitrary constant $c_0=c_0(x,a)>0$ satisfying  $0<c_0< \psi'(\kappa)$. Define  
 \begin{equation}
   G^{4}_{\ell}(w):= \{ \forall r \in (\ell, \tau^w(z)], V_{\tau^w(z)-r}(w) -V_{\tau^w(z)}(w) \geq c_0 r   \} . 
 \end{equation}
 See also Figure \ref{fig-BMhit}.   
 We claim that there is   deterministic sequence  $\{\delta(\ell): \ell \geq 1\}$ with $\delta(\ell) \to 0$ as $\ell \to \infty$ such that  
 \begin{equation}\label{eq-path-of-spine-G-4}
  (\mathbf{Q}^{*}_{\tau^w(z)} \otimes \P) \left(  G_{L}^{1}(w) \backslash  G_{\ell}^{4} (w) \right)  \leq \delta(\ell)  . 
 \end{equation}  
Then   on $G_{L}^{4}(w)$ for any $r \in [0,\tau^w(z)-\ell)$ we have $V_{r}(w) +\psi'(\kappa)s   \geq  c_0 [\tau^w(z) - r]+z$. Noting that $ \theta[\mathsf{L}_{t}(r) - X_{r}(w)]= V_{r}(w)+\psi'(\kappa)s$, this yields $  \mathsf{L}_{t}(r) - X_{r}(w) \geq [c_0 (\tau^w(z) - r)+z]/\theta$. Besides, for large $t$, we have  
 $\frac{xt-\mathsf{L}_{t}(r)}{t-r} \geq \frac{\theta}{2}$ for every $r \leq s+L \sqrt{s}$. See Figure  \ref{fig-FL}. 
 Thus  we conclude  
 \begin{align} 
  & Q_{\eqref{eq-distant-relative-2}} (t,\ell) 
      \lesssim_{x,a} (\mathbf{Q}^{*}_{\tau^w(z)} \otimes \P) (G_{L}^{1}(w) \backslash  G_{\ell}^4 (w)  )  \\
    & \qquad \qquad  \qquad   + \frac{1}{\epsilon}\E_{  \mathbf{Q}^{*}_{\tau^w(z)} \otimes \P } \left[   \sum_{ b_{v} \leq \tau^w(z) - \ell} \, e^{ - \frac{xt-\mathsf{L}_{t}(b_{v})}{\theta (t-b_{v})} [c_0(\tau^w(z)-b_{v})+z] }    \ind{ G_{L}^{1}(w) \cap  G^{4}_{\ell}(w) } \right] \\
 &\leq   \delta(\ell) + \frac{e^{-z/2}}{\epsilon} \E_{ \mathbf{Q}^{*}_{\tau^w(z)} \otimes \P } \left[   \sum_{ b_{v} \leq \tau^w(z) - \ell} \, e^{ -c_0\, (\tau^w(z)-b_{v})/2 } \right] \lesssim_{\epsilon}  \delta(\ell) + e^{R/2} \int_{\ell}^{\infty} e^{-c_0   y/2} \dif y .
 \end{align}
 Above,   we have used  that  conditioned on $\tau^w(z)$,  $(b_{v}: v \in \mathcal{B}(\hat{w})\backslash \{\hat{w}\})$ forms a Poisson process with rate $2$ on the interval $[0,\tau^w(z)]$. 
 Now we can deduce \eqref{eq-distant-relative-2}  and hence   \eqref{eq-distant-relative-level}.

 \begin{figure}[t]
  \includegraphics*[width=\textwidth]{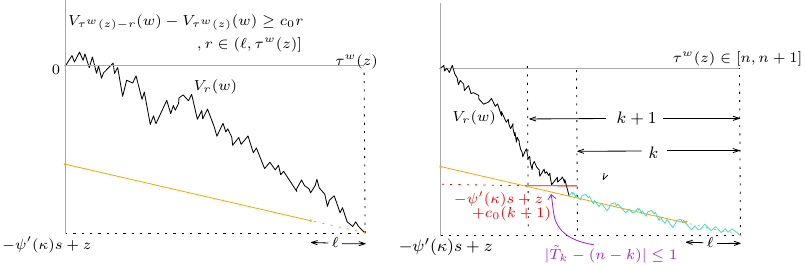}
  \caption{
    Left: Event $ G^{4}_{\ell}(w)$. Right: Event $ G^{1}_{L}(w) \backslash G^{4}_{\ell}(w)$. }\label{fig-BMhit}
\end{figure} 

It remains to show the claim \eqref{eq-path-of-spine-G-4}.  In fact, if $G_{L}^{1}(w) \backslash  G_{\ell}^{4} (w)$ happens, then there must exist  some  $k\geq \ell$ such that the process $\{V_{r}(w)\}_{r \geq 0}$ hits the step function $r \mapsto V_{\tau^w(z)}(w) + c_0 \lceil \tau^w(z)- r \rceil$ during the time interval $ [\tau^w(z)-k-1,\tau^w(z)-k]$. See Figure \ref{fig-BMhit}. Thus the first passaging time at level $V_{\tau^w(z)}(w) + c_0(k+1)$ by the process $\{V_{r}(w)\}$, say $T^{*}_{k} =  T^{V}(-\psi'(\kappa)s+z+c_0(k+1))$,
  belongs  to $[\tau^w(z)-k-1,\tau^w(z)-k]$.   
 So if addition $\tau^w(z) \in [n,n+1]$, we have  $|T^{*}_{k} - (n-k)|\leq 1$ and $V_{n}- V_{T^{*}_{k}} \geq - c_0(k+1)$. 
In conclusion,  we can  bound $ (\mathbf{Q}^{*}_{\tau^w(z)} \otimes \P) (G_{L}^{1}(w) \backslash  G_{\ell}^4 (w)  )$ above by 
\begin{equation}
  \sum_{\substack{|n-s|\leq L \sqrt{s} \\  \ell \leq k <n}} (\mathbf{Q}^{*}_{\tau^w(z)} \otimes \P) \left(  |T^{*}_{k} - (n-k)|\leq 1, V_{n}(w)- V_{T^{*}_{k}}(w) \geq - c_0(k+1) ,\tau^w(z) \in [n,n+1] \right) .
\end{equation}
Let  $(W,\P)$ denote a Brownian motion with drift $\psi'(\kappa)$ and  diffusion coefficient $\psi''(\kappa)$, and 
let $T^{W}_{k}$ be the first passaging time at level $ -\psi'(\kappa)s + z + c_0(k+1)$.  
Since under $\mathbf{Q}^{*}_{\tau^w(z)} \otimes \P$,  $(V_{r}(w): r \leq \tau^w(z))$ has the same law as $W $  stopped upon hitting the line  $-\psi'(\kappa)s+z$, by using  the strong Markov property, we get  
  \begin{align}
   &  (\mathbf{Q}^{*}_{\tau^w(z)} \otimes \P) (G_{L}^{1}(w) \backslash  G_{\ell}^4 (w)  ) \\ 
  & \leq   \sum_{\ell \leq k  }  \sum_{\substack{k+1 \leq n  \\|n-s|\leq L \sqrt{s}}}  \mathbf{P}\left( |T^{W}_{k} -(n-k)| \leq 1 \right) \max_{|r| \leq 1} \P\left(  W_{k+r} \geq - c_0(k+1)  \right) .
  \end{align} 
By applying the Gaussian tail inequality, we deduce that there exists a constant $c_{1}=c_{1}(x,a)>0$ such that $\max_{|r| \leq 1} \P\left(  W_{k+r} \geq - c_0(k+1)  \right) \lesssim e^{- c_{1} k} $.  On the other hand, notice that 
\begin{equation} 
  \sum_{ k+1 \leq n   }  \mathbf{P}\left( |T^{W}_{k} -(n-k)| \leq 1 \right) \leq   2\sum_{ j \geq 1  }  \mathbf{P}\left(  T^{W}_{k} \in [j-1,j] \right) \leq 2.
\end{equation}
Finally we obtain 
\begin{equation}
 (\mathbf{Q}^{*}_{\tau^w(z)} \otimes \P)(G_{L}^{1}(w) \backslash  G_{\ell}^4 (w)  )    \lesssim_{x,a} \sum_{\ell \leq k } e^{-c_{1} k } . 
\end{equation}
as desired. We now completes the proof of \eqref{eq-distant-relative-level}.
  


\underline{Proof of \eqref{eq-sstar-tauz-1}.} By the definition of $\s$ in \eqref{eq-def-ss}, on the event $\{\mathbf{I}  \in  \Gamma[t,z]\}$ we have  $\tau(z) \leq \s$.  So it suffices to show that  
\begin{equation}
  \limsup_{t \to \infty} e^{I(x,a)t} \P \left(   \s - \tau(z) > K ,   \mathbf{I}  \in  \Gamma[t,z]   \right)   \overset{K \to \infty}{\longrightarrow} 0. 
 \end{equation}  
 For  $\ell \geq 1$, define  $E^{1}_{\ell}:= \{ \exists v \in \mathcal{B}^{>\ell}(\hat{u}), V_{b_{v}}(v) + \mathbf{I}^{(v)} \geq -\psi'(s)+z  \}$. 
According to Lemma \ref{lem-good-event} and \eqref{eq-distant-relative-min}, for each $\delta>0$ there exists $\ell=\ell(\delta,R), L = L(\delta,R)$
such that  
\begin{equation}\label{eq-sstar-tauz-2}
  \limsup_{t \to \infty} e^{I(x,a)t} \P \left(  ( E^{1}_{\ell} \cap G_{\ell,L} )   ^{c} , \mathbf{I}  \in  \Gamma[t,z]  \right) <\delta.  
\end{equation}
Define   $E^{1}_{\ell}(w):= \{ \exists v \in \mathcal{B}^{>\ell}(\hat{w}), V_{b_{v}}(w) + \mathbf{I}^{(v)} \geq -\psi'(s)+z  \}$. By making change of measure and applying \eqref{eq-change-of-measure},  we have 
\begin{align}
  & e^{I(x,a)t} \P \left(   \s - \tau(z) > K , E^{1}_{\ell},  G_{\ell,L} , \mathbf{I}  \in  \Gamma[t,z] \right) \\
  & =   e^{\kappa z} ( \mathbf{Q}^{*}_{\tau^w(z)} \otimes \P )  \left(  \s - \tau^{w}(z) > K ,  E^{1}_{\ell}(w) , I_0^{w}, G^{w}_{\ell,L}  \right). 
\end{align}
 On the event $ E^{1}_{\ell}(w)$, whenever $\s- \tau^{w}(z)>K$, there must exist  $v \in \mathcal{B}^{\leq \ell}(\hat{w})$ such that $V_{b_{v}}(w)+\min\{V^{(v)}_{r}(u): u \in \mathcal{N}^{(v)}_{r}, r>K\}  <  V_{\tau^w(z)}(w)$. If in addition $G^{w}_{\ell,L}$ occurs, we have  $|\mathcal{B}^{\leq \ell}(\hat{w})|\leq L$. Noting that $V_{b_{v}}(w)- V_{\tau^w(z)}(w) \geq 0 $ and  using the branching property,   we conclude that 
\begin{align}
& ( \mathbf{Q}^{*}_{\tau^w(z)} \otimes \P ) \left(  \s - \tau^{w}(z) > K ,  E^{1}_{\ell}(w), G^{w}_{\ell,L}  \mid \mathcal{G}^{*}_{\tau^w(z)}   \right) \\
 & \quad \leq L\,  \P(  \min\{V_{r}(u): u \in \mathcal{N}_{r}, r>K\} \leq 0) \overset{K \to \infty}{\longrightarrow} 0 .
\end{align}
Above we have used \eqref{eq-HS-maximum} and the definition $V_{r}(u):= (\frac{\theta^2}{2}+1)r - X_{r}(u)$. Then the  bounded convergence theorem yields that uniformly in $|z| \leq R$
\begin{equation}
 \limsup_{K \to \infty} \limsup_{t \to \infty} e^{I(x,a)t} \P \left(   \s - \tau(z) > K , E^{1}_{\ell}, G_{\ell,L}  , \mathbf{I}  \in  \Gamma[t,z] \right) =0 . 
\end{equation} 
Combining this with \eqref{eq-sstar-tauz-2} we get  $ \limsup\limits_{K \to \infty} \limsup\limits_{t \to \infty} e^{I(x,a)t} \P \left(   \s - \tau(z) > K ,     \mathbf{I}  \in  \Gamma[t,z] \right) < \delta$. Since $\delta$ is arbitrary, the desired result \eqref{eq-sstar-tauz-1} follows.


 \underline{Proof of \eqref{eq-distant-relative-max}.}  Let us define \( Q_{\eqref{eq-relative-max-3}}(t, \ell) \) as the following probability:
 \begin{equation}\label{eq-relative-max-3}
    (\mathbf{Q}^{*}_{\tau^w(z)} \otimes \P)  \left(  \max_{v \in \mathcal{B}^{>\ell}(\hat{w})}   \{ X_{b_{v}}(w)+ M^{(v)}_{t-b_{v}}\} \geq X_{\tau^w(z)}(w)+ \sqrt{2}[t-\tau^{w}(z)], I_{0}^{w},  G^{1}_{L}(w) \right).
 \end{equation} 
 As in the proof of \eqref{eq-distant-relative-level},
 it suffices to show that for any fixed $R>0$, $L\geq 1$,  $ \lim\limits_{\ell \to \infty}\limsup\limits_{t \to \infty} Q_{\eqref{eq-relative-max-3}}(t,\ell) = 0  $  uniformly in $|z| \leq R$. We recall a well-known inequality bout the maximal position of the BBM (see \cite{Bramson83} or \cite{Aidekon13}):
 $\mathbf{P}(M_{t} \geq \sqrt{2} t + y ) \lesssim  e^{-\sqrt{2} y}$ for $y \geq 1$.   
 By applying the branching property, conditioned on $ \mathcal{G}^{*}_{\tau^w(z)} $,
  for each $v \in \mathcal{B}^{>\ell}(\hat{w})$,  we have 
  \begin{align}
   &  (\mathbf{Q}^{*}_{\tau^w(z)} \otimes \P) \left( X_{b_{v}}(w)+ M^{(v)}_{t-b_{v}} \geq X_{\tau^w(z)}(w)+ \sqrt{2}[t-\tau^{w}(z)]   \mid  \mathcal{G}^{*}_{\tau^w(z)}   \right)  \ind{G^{1}_{L}(w)} \\
    & \qquad \lesssim_{R} \exp \left\{ - \sqrt{2} [\mathsf{L}_{t}(b_{v}) -X_{b_{v}}(w)] \right\} \ind{G^{1}_{L}(w)}, 
   \end{align}  
   where we have used the facts  that  $ X_{\tau^w(z)}(w)= \mathsf{L}_{t}(\tau^w(z))- z/\theta$, and hence 
   \begin{align}
    & X_{\tau^w(z)}(w)+ \sqrt{2}[t-\tau^{w}(z)] - \mathsf{L}_{t}(b_{v}) + \mathsf{L}_{t}(b_{v}) -X_{b_{v}}(w) \\
     & = \sqrt{2}(t-b_{v})+  \frac{(\theta-\sqrt{2})^2}{2 \theta} \left( \tau^w(z)-b_{v} \right)  + \mathsf{L}_{t}(b_{v}) -X_{b_{v}}(w)- \frac{z}{\theta} . 
   \end{align}
 By using the the Markov inequality and the fact that   $\theta[\mathsf{L}_{t}(b_{v}) -X_{b_{v}}(w)] \geq c_0(\tau^w(z)-b_{v}) +z$ on $G^{4}_{L}(w)$ (see the argument below \eqref{eq-path-of-spine-G-4}), we finally get  
   \begin{align} 
     Q_{\eqref{eq-relative-max-3}}(t,\ell) &\lesssim_{x,a,R}  (\mathbf{Q}^{*}_{\tau^w(z)} \otimes \P)(G_{L}^{1}(w) \backslash  G_{\ell}^4 (w)  )     \\
    & \qquad  + \E_{\mathbf{Q}^{*}_{\tau^w(z)}\otimes \P } \bigg[   \sum_{ b_{v} \leq \tau^w(z) - \ell} \, e^{ - \sqrt{2}c_0(\tau^w(z)-b_{v}) /\theta }    \ind{ G_{L}^{1}(w) \cap  G^{4}_{\ell}(w) } \bigg] \\
    &    \leq  \delta(\ell) + 2 e^{R/2} \int_{\ell}^{\infty} e^{- \sqrt{2} c_0 y/{\theta} }\dif y
    \end{align}
   Again  we have used  that  conditioned on  $\tau^w(z)$,  $(b_{v}: v \in \mathcal{B}(\hat{w})\backslash \{\hat{w}\})$ forms a Poisson process with rate $2$ on the interval $[0,\tau^w(z)]$.  Then we get that $   \lim\limits_{\ell \to \infty}\limsup\limits_{t \to \infty} Q_{\eqref{eq-relative-max-3}}(t,\ell) = 0  $ and this
    completes the proof of    \eqref{eq-distant-relative-max}.
     \end{proof}

 \begin{proof}[Proof of Lemma \ref{lem-replace-level-by-mart}]
  Thanks to Lemma \ref{lem-good-event},    it suffices to show that for each fixed $L\geq 1$, uniformly in $|z| \leq R$ and $|\lambda_{i}-1| \leq \epsilon$, $i=1,2$,  
  \begin{align}\label{eq-reduce-replacing-0}
    \quad \lim_{\epsilon \downarrow 0}  \limsup_{t \to\infty}  e^{I(x,a)t}    \P \left(  \left\{  \L_{t}^{\leq \ell}  (x t) \geq \frac{\lambda_{1} e^{at} }{c_{\theta}\sqrt{t}}  \right\}  \sdif  \left\{  \frac{W^{\leq \ell}_{\infty}(\theta)}{C_{t}(z)} \geq  \lambda_{2} y e^{\psi'(\kappa) s} \right\},  
     \mathbf{I} \in  \Gamma[t,z],  G_{\ell, L}   \right)  = 0. \quad 
    \end{align} 
  To establish \eqref{eq-reduce-replacing-0}, by applying \eqref{eq-change-of-measure}  with  $A=\{\L_{t}^{\leq \ell}  (x t) \geq \frac{\lambda_{1} y}{c_{\theta} \sqrt{t}} e^{at}\} \sdif \{W^{\leq \ell}_{\infty}(\theta) /C_{t}(z)\geq  \lambda_{2} y e^{\psi'(\kappa) s}\} \cap G_{\ell, L}^{2}$,  
 it suffices to show that uniformly in $|z| \leq R$ and $|\lambda_{i}-1| \leq \epsilon$, $i=1,2$,  
\begin{equation}\label{eq-reduce-replacing}
  \lim_{\epsilon \downarrow 0} \limsup_{t \to\infty}      (\mathbf{Q}^{*}_{\tau^{w}(z)} \otimes \P) \left(   \left\{  \L_{t}^{\leq \ell}  (x t) \geq \frac{\lambda_{1} e^{at} }{c_{\theta}\sqrt{t}}  \right\} \sdif \left\{   \frac{W^{\leq \ell}_{\infty}(\theta)}{C^{w}_{t}(z)} \geq  \lambda_{2} y e^{\psi'(\kappa) s} \right\},   G^{w}_{\ell,L} , I_0^{w} \right)   = 0
\end{equation} 
Above we have used the fact  that  on the event $\{\tau(z)=\tau^w(z) <t\} $, we have $\hat{u}=\hat{w}$ and hence $C_{t}(z)=C_t^w(z)  $ and  $I_{0}^{w} \cap G_{\ell,L} = I_{0}^{w} \cap G_{\ell,L}^{w}$ where 
\begin{align}
  & C_t^w(z)  :=    \exp \left(\frac{p}{1-p}   \frac{(1-{\theta^2}/{2})^2    }{2\theta^2  } \frac{(\tau^{w}(z)-s)^2}{s}   \right) \ , \text{ and } \\ 
  & G^{w}_{\ell,L} := G^{1}_{L}(w) \cap  \left\{ | \mathcal{B}^{\leq \ell}(\hat{w}) |   \leq L  \right\}  \cap \left\{  \forall \, r \in [\tau^{w}(z) -\ell, \tau^{w}(z)] ,  | V_{r}(w) - V_{\tau^{w}(z)}(w)| \leq L \right\} . 
\end{align} 
Again, on the event $I_{0}^{w} \cap G_{\ell,L}^{w}$, since $\hat{u}=\hat{w}$, we have 
\begin{equation}
  \L_{t}^{\leq \ell}  (x t) =   
  \sum_{v \in \mathcal{B}^{\leq \ell}(\hat{w})}  \L_{t-b_{v}}^{(v)} (x t-X_{b_{v}}(\hat{w})   ) . 
\end{equation}

Let  $\mathcal{G}^{*}_{\tau^w(z)}:= \sigma( V_{r}(w): r \leq  \tau^w(z),  b_{v} ,  v \in \mathcal{B}(w))$ denote the $\sigma$ field that contains the information about the movement and branching times of the spine particles up to time  $\tau^w(z)$. 
We assert that on the event $G^{w}_{\ell, L} \in \mathcal{G}^{*}_{\tau^w(z)}$,  
\begin{equation}
  \E_{ \mathbf{Q}^{*}_{\tau^w(z)} \otimes\P } \left[  \L_{t-b_{v}}^{(v)} (x t-X_{b_{v}}(\hat{w})  ) \mid \mathcal{G}^{*}_{\tau^w(z)}  \right]  =[ 1+ h_{b_{v}}(w)  ] \frac{1}{c_{\theta}C^{w}_{t}(z)  \sqrt{t}} e^{at} e^{V_{\tau^w(z)}(w) - V_{b_{v}}(w)}  e^{-z}  
\end{equation} 
where $\sup_{v \in \mathcal{B}^{\leq \ell}(\hat{w})}|h_{b_{v}}(w)|$ is   bounded above by a   deterministic $o(1)$ term. To show this, write  $Y_{b_{v}}(w)=x t-X_{b_{v}}(w) -\theta(t-b_{v})  $. 
We  compute that 
for each $v \in \mathcal{B}^{\leq \ell}(\hat{w})$,
\begin{align}
 & \E_{ \mathbf{Q}^{*}_{\tau^w(z)} \otimes\P} \left[  \L_{t-b_{v}}^{(v)} (x t-X_{b_{v}}(w)   ) \mid \mathcal{G}^{*}_{\tau^w(z)}  \right] \\
 &= \frac{[1+o(1)]}{\sqrt{2 \pi}} \frac{\sqrt{t-b_{v}}}{x t-X_{b_{v}}(w)} \exp \left\{  (t-b_{v}) - \frac{[xt- X_{b_{v}}(w)]^2}{2(t-b_{v})} \right\}  \\
 &= \frac{[1+o(1)]}{\sqrt{2 \pi}} \frac{\sqrt{t-b_{v}}}{x t-X_{b_{v}}(w)} \exp \left\{  (t-b_{v}) - \frac{\theta^2}{2}(t-b_{v}) - \theta Y_{b_{v}}(w) - \frac{Y^{2}_{b_{v}}(w)}{2(t-b_{v})} \right\}  .
\end{align}
From the relations  $at= (1- \frac{\theta^2}{2})(t-s)$, $xt-bs=\theta(t-s)$, $\theta b = 2$ and $V_{b_{v}}(w)= ( \frac{\theta^2}{2} + 1 ) b_{v}-\theta X_{b_{v}}(w) $, we can infer that  
\begin{align}
 (t-b_{v}) - \frac{\theta^2}{2}(t-b_{v}) - \theta Y_{b_{v}}(w)
& = at- (1-\frac{\theta^2}{2}) s- V_{b_{v}}(w)   \\
 & = at + V_{\tau^w(z)}(w) - V_{b_{v}}(w) - z . \label{eq-Y-bv-w-22}
\end{align}
Moreover, on the event $G^{w}_{\ell, L} $, from \eqref{eq-Y-bv-w-22}  we deduce that  
\begin{equation}\label{eq-Y-b-w} 
  \left|\theta Y_{b_{v}}(w) - (1-\frac{\theta^2}{2})(s-b_{v})\right| \leq  |V_{\tau^w(z)}(w) - V_{b_{v}}(w) |+|z|  \leq  \ell + R +L.
\end{equation}  
Since for each time point $b_{v} \in [\tau^w (z) -\ell,  \tau^w(z) ]$, we have 
$\left| |s-b_{v}| -  |\tau^{w}(z)-s| \right| \leq  \ell$. This implies that  $\frac{x t-X_{b_{v}}(w)}{(t-b_{v})} = \theta + \frac{Y_{b_{v}}(w)}{t-b_{v}}=\theta+o(1)$, and  
\begin{equation}
  \frac{Y_{b_{v}}^{2}(w)}{2(t-b_{v})} = \frac{(1-\frac{\theta^2}{2})^2 (\tau^w(z)-s)^2}{2 \theta^2(t-b_{v} )} + o(1) =  \frac{(1-\frac{\theta^2}{2})^2 p (\tau^w(z)-s)^2}{2 \theta^2 (1-p) s} + o(1)
\end{equation}
The assertion then follows from the computations above.

For simplicity, we define  $ \bar{\L}_{t-b_{v}}^{(v)} ( \cdot ) :=  { \L_{t-b_{v}}^{(v)} (\cdot )}/{  \mathbf{E}_{ \mathbf{Q}^{*}_{\tau^{w}(z)} \otimes \P} [\L_{t-b_{v}}^{(v)} (\cdot ) \mid \mathcal{G}^{*}_{\tau^w(z)} ]  }$. Then on the event $I_{0}^{w} \cap G_{\ell,L}^{w}$,  we can rewrite 
  \begin{equation}
  \left\{  \L_{t}^{\leq \ell}  (x t) \geq \frac{\lambda_{1} y}{c_{\theta}\sqrt{t}} e^{at}   \right\} = \bigg\{  \Sigma_{\eqref{eq-rewite-L}}(\ell) \geq  \lambda_{1} C^{w}_{t}(z) y e^{z}\bigg\} . 
\end{equation}
where 
\begin{align}
  \Sigma_{\eqref{eq-rewite-L}}(\ell,t,z):=   \sum_{v \in \mathcal{B}^{\leq \ell}(w)}  e^{V_{\tau^w(z)}(w) - V_{b_{v}}(w)}  [1+ h_{b_{v}}(w)] \bar{\L}_{t-b_{v}}^{(v)} (x t-X_{b_{v}}(w)   ) \label{eq-rewite-L}  
\end{align} 
By using the branching property and Lemma \ref{lem-enhanced-convergence}, for any fixed $\epsilon>0$, it holds that 
for every $v \in \mathcal{B}^{\leq \ell}(w)$, 
\begin{align}
&       (\mathbf{Q}^{*}_{\tau^{w}(z)} \otimes \P) \left(  [1+ h_{b_{v}}(w)] \bar{\L}_{t-b_{v}}^{(v)} (x t-X_{b_{v}}(w)) \notin [1-\epsilon,1+\epsilon]   W^{(v)}_{\infty}(\theta)    \mid  \mathcal{G}^{*}_{\tau^{w}(z)} \right)  \ind{G^{w}_{\ell,L}}   \\
  &\leq  \mathbf{P}\left(  \L_{t'} (\theta t'+ y')  / W_{\infty}(\theta) \notin [1-2\epsilon,1+2\epsilon]  \right) \bigg| _{ t'= t-b_{v}, y'=Y_{b_{v}}(w) } \ind{G^{w}_{\ell,L}} \overset{t \to \infty}{\longrightarrow} 0 .  \text{ a.s. }
\end{align}
Thus by employing the bounded convergence theorem, we obtain that  for any fixed  $\epsilon>0$, the probability  
\begin{equation}
 Q_{t}(\epsilon):=       (\mathbf{Q}^{*}_{\tau^{w}(z)} \otimes \P)  \left(  \exists \, v \in \mathcal{B}^{\leq \ell}(w), \text{ s.t. }  \bigg| [1+ h_{t_{i}}(w)] \frac{\bar{\L}_{t-t_i}^{(v)} (x t-X(v))}{W^{(v)}_{\infty}(\theta)}  -1 \bigg|>\epsilon  ; G^{w}_{\ell,L}\right) 
\end{equation}
converges to $0$ as $t \to \infty$. Again, on the event $I_{0}^{w} \cap G_{\ell,L}^{w}$,   since $\hat{u}=\hat{w}$ we can rewrite
  \begin{equation} 
   \left\{   \frac{ W^{\leq \ell}_{\infty}(\theta)}{C^{w}_{t}(z)} \geq  \lambda_{2} y  e^{\psi'(\kappa) s}   \right\}    =  \bigg\{    \Sigma_{\eqref{eq-rewite-W}}(\ell)  \geq  \lambda_{2} C^{w}_{t}(z)  y   e^{z}    \bigg\}  ,
\end{equation} 
where
\begin{equation}
  \Sigma_{\eqref{eq-rewite-W}}(\ell) :=   \sum_{v \in \mathcal{B}^{\leq \ell}(w)}   e^{V_{\tau^w(z)}(w) - V_{b_{v}}(w)}   W^{(v)}_{\infty}(\theta  )    \label{eq-rewite-W} 
\end{equation}
Note that when $  [1+ h_{t_{i}}(w)] \frac{\bar{\L}_{t-t_i}^{(v)} (x t-X(v))}{W^{(v)}_{\infty}(\theta)}  \in [1-\epsilon, 1+ \epsilon]$ holds for every $ v \in \mathcal{B}^{\leq \ell}(w)$,  we must have $
 \frac{   \Sigma_{\eqref{eq-rewite-L}}(\ell,t,z) }{   \Sigma_{\eqref{eq-rewite-W}}(\ell) } \in [1-\epsilon, 1+ \epsilon ]$.
Thus, if $ \Sigma_{\eqref{eq-rewite-W}}(\ell)/C_{t}^{w}(z) \geq \lambda_{2} ye^{z} $ but $    \Sigma_{\eqref{eq-rewite-L}}(\ell,t,z)/C_{t}^{w}(z)  \leq \lambda_{1} y e^{z}$, we have $ \lambda_{2} \leq   \frac{\lambda_{1}}{1-\epsilon}  $. If $ \Sigma_{\eqref{eq-rewite-W}}(\ell) /C_{t}^{w}(z)\leq \lambda_{2} ye^{z} $ but  $ \Sigma_{\eqref{eq-rewite-L}}(\ell,t,z) /C_{t}^{w}(z) \geq \lambda_{1} y e^{z}$ we have $   \frac{\lambda_{1}}{1+\epsilon}\leq \lambda_{2}   $. Finally  we conclude that for any $|\lambda_{i}-1|< \epsilon$, $i=1,2$,
\begin{align}
  &       (\mathbf{Q}^{*}_{\tau^{w}(z)} \otimes \P)   \left(    \left\{  \L_{t}^{\leq \ell}  (x t) \geq \frac{\lambda_{1} e^{at} }{c_{\theta}\sqrt{t}}  \right\} \sdif \left\{  \frac{W^{\leq \ell}_{\infty}(\theta)}{C^{w}_{t}(z)} \geq  \lambda_{2} y e^{\psi'(\kappa) s} \right\} ,   I_{0}^{w} , G^{w}_{\ell,L}   \right)  \\
  & \leq         (\mathbf{Q}^{*}_{\tau^{w}(z)} \otimes \P)   \left(  \frac{\Sigma_{\eqref{eq-rewite-W}}(\ell)}{ C_{t}^{w}(z)} \in    y e^{z}  \left[ \frac{1-\epsilon}{1+\epsilon}, \frac{1+\epsilon}{1-\epsilon} \right]  \right) + Q_{t}(\epsilon ). 
\end{align}  
Observe that, the distribution of $  \Sigma_{\eqref{eq-rewite-W}}(\ell):= \sum_{v \in \mathcal{B}^{\leq \ell}(w)}   e^{V_{\tau^w(z)}(w) - V_{b_{v}}(w)}  W^{(v)}_{\infty}$  does not depend on $\tau^w(z)$, $t$ or $z$ at all  
 (so it is independent of $C^{w}_t(z)$)
  and it  has a probabilistic density function since the martingale limit $W_{\infty}(\theta)$ does. (Indeed  \cite{Stigum66} demonstrated that for  the  martingale  of a nonextinct GW process
limit has a density; \cite[Proposition 3]{Neveu88} and \cite[Theorem 8]{Kyprianou04} further  showed that in fact $W_{\infty}(\theta)$ is the martingale limit for certain GW processes). 
Additionally, notice that $C^{w}_{t}(z)$ converges in distribution to the random variable $\mathcal{C}$ introduced in \eqref{eq-apro-W-limit-C} as $t \to \infty$,
because $\{ (\tau^w(z)-s)/\sqrt{s} ,\mathbf{Q}^{*}_{\tau^{w}(z)} \otimes \P \}$ converges in distribution to a centered Gaussian random variable with variance $\frac{\psi''(\kappa)}{\psi'(\kappa)^2}$  by Lemma \ref{lem-BMhitting}. (Indeed one can write down the density function of $(\tau^w(z)-s)/\sqrt{s}$ and see it converges to the density of that Gaussian distribution). Therefore as $t \to \infty$, we have 
 \begin{equation}
  (\mathbf{Q}^{*}_{\tau^{w}(z)} \otimes \P)   \left(  \frac{\Sigma_{\eqref{eq-rewite-W}}(\ell)}{ C_{t}^{w}(z)} \in    y e^{z}  \left[ \frac{1-\epsilon}{1+\epsilon}, \frac{1+\epsilon}{1-\epsilon} \right]  \right) \to \mathsf{P}  \left(  \frac { \Sigma(\ell)}{\mathcal{C}} \in    y e^{z}  \left[ \frac{1-\epsilon}{1+\epsilon}, \frac{1+\epsilon}{1-\epsilon} \right]  \right) .
 \end{equation}
where under $\mathsf{P}$, $ \Sigma(\ell)$ as the same distribution as $  \{   \Sigma_{\eqref{eq-rewite-W}}(\ell), \mathbf{Q}^{*}_{\tau^{w}(z)} \otimes \P  \} $ and $\mathcal{C}$
is a positive random variable independent to $\Sigma(\ell)$.
Now letting $\epsilon \downarrow 0$, since $\Sigma(\ell)/\mathcal{C}$ has a probability density function,
we obtain the desired convergence in \eqref{eq-reduce-replacing}. This completes the proof. 
\end{proof}

\subsection{Proofs of Theorem \ref{thm-CondOverlap} and \ref{thm-CondMaxs}}
 
\begin{proof}[Proof of Theorem \ref{thm-CondOverlap}] 
  
  \underline{Step 1.} 
  Fix $R>0$. We aim to show that uniformly in $|z| \leq R$, the following holds:
  \begin{equation}\label{eq-genealogy-0}
    \limsup_{t \to \infty}\P \left( \left| \mathscr{R}(u^{1}_{t},u^{2}_{t}) - \tau(z)  \right|> \ell
    \mid \L_{t}(xt) \geq \frac{1}{\sqrt{t}} e^{at}, \mathbf{I}  \in  \Gamma[t,z]   \right)    \overset{\ell \to \infty}{\longrightarrow} 0. 
  \end{equation} 
Note that  if $\mathscr{R}(u^{1}_{t},u^{2}_{t}) \leq  \tau(z)- \ell$, then there must exist $v \in \mathcal{B}^{>\ell}(\hat{u}) $ satisfying  $ v \prec u^{i}_{t}$ for at least one $i \in \{1,2\}$. So we first demonstrate that, uniformly in $|z| \leq R$, 
 \begin{align}
  & \limsup_{t \to \infty} \P\left( \exists \,  v \in \mathcal{B}^{>\ell}(\hat{u}) \text{ s.t. } v \prec u^{1}_{t}
   \mid \L_{t}(xt) \geq \frac{1}{\sqrt{t}} e^{at}, \mathbf{I}  \in  \Gamma[t,z]   \right)  \\ 
   & =  \limsup_{t \to \infty} \E\left(  \frac{ \L_{t}^{> \ell}(xt) }{\L_{t}(xt) }
   \mid \L_{t}(xt) \geq \frac{1}{\sqrt{t}} e^{at}, \mathbf{I}  \in  \Gamma[t,z]   \right)   \overset{\ell \to \infty}{\longrightarrow} 0 . \label{eq-genealogy-1}
 \end{align}
 By applying  \eqref{eq-apro-W-limit} and Lemma \ref{lem-truncated-level-set} (i),  we have, for any fixed $\epsilon > 0$ 
\begin{align}
  & \limsup_{t \to \infty} \P\left(  \frac{ \L_{t}^{> \ell}(xt) }{\L_{t}(xt) } > \epsilon 
  \mid \L_{t}(xt) \geq \frac{1}{\sqrt{t}} e^{at}, \mathbf{I}  \in  \Gamma[t,z]   \right)  \\
 & \leq \limsup_{t \to \infty} \frac{ \P\left(   \L_{t}^{> \ell}(xt)  > \frac{\epsilon}{\sqrt{t}} e^{at} ,   \mathbf{I}  \in  \Gamma[t,z]   \right)}{\P\left(   \L_{t}(xt) \geq \frac{1}{\sqrt{t}} e^{at}, \mathbf{I}  \in  \Gamma[t,z]   \right)} \\
  & \lesssim_{x,a}  \limsup_{t \to \infty} e^{I(x,a)t}\P\left(   \L_{t}^{> \ell}(xt)  > \frac{\epsilon}{\sqrt{t}} e^{at} ,   \mathbf{I}  \in  \Gamma[t,z]   \right)  \overset{\ell \to \infty}{\longrightarrow} 0 . 
\end{align}  
Then  the claim \eqref{eq-genealogy-1} follows from the inequality $\E( Y ) \leq \P(Y > \epsilon) + \epsilon$ for any random variable $Y$ taking values in $[0,1]$.  

Next, we show that fixed $\ell>0$  and $R>0$, uniformly in   $|z| \leq R$, the following holds:
\begin{equation}\label{eq-genealogy-2}
  \limsup_{t \to \infty} \P\left( \mathscr{R}(u^{1}_{t},u^{2}_{t}) \geq \tau(z) -\ell + 2 r  
  \mid \L_{t}(xt) \geq \frac{1}{\sqrt{t}} e^{at}, \mathbf{I}  \in  \Gamma[t,z]   \right)  \overset{ r \to \infty}{\longrightarrow} 0 . 
\end{equation}
Notice that  on the event $\{ \exists v^{i} \in \mathcal{B}^{\leq \ell}(\hat{u}) \text{ s.t. } v^{i} \prec u^{i}_{t} ,i=1,2  \}$, the event $ \{\mathscr{R}(u^{1}_{t},u^{2}_{t}) \geq \tau(z) -\ell + 2 r \}$ occurs if and only if there is $v \in \mathcal{B}^{\leq \ell}(\hat{u})$ and $v' \in \mathcal{N}^{(v)}_{r}$ such that both $u^{1}_{t} $ and $ u^{2}_{t} $ are descendants of $v'$, provided that $r>\ell$. 
 By use of  \eqref{eq-apro-W-limit},  we obtain  
\begin{align}
  & \P\left( 
    \begin{array}{l} 
      \exists v^{i} \in \mathcal{B}^{\leq \ell}(\hat{u}) \text{ s.t. } v^{i} \prec u^{i}_{t}, \  \forall i \\
       \mathscr{R}(u^{1}_{t},u^{2}_{t})  \geq \tau(z) -\ell +2 r  
    \end{array} 
  \mid \L_{t}(xt) \geq \frac{1}{\sqrt{t}} e^{at}, \mathbf{I}  \in  \Gamma[t,z]   \right) \\
  & \lesssim_{x,a}  e^{I(x,a)t} \E\left( \sum_{v \in \mathcal{B}^{\leq \ell}(\hat{u})}  \sum_{v' \in \mathcal{N}^{(v)}_{r} }   \frac{\L^{(v')}_{t-b_{v}-r}(xt- X_{b_{v}+r}(v'))^2 }{\L_{t} (xt)^2 }   \ind{\L_{t}(xt) \geq \frac{1}{\sqrt{t}} e^{at}, \mathbf{I}  \in  \Gamma[t,z]  } \right) \\
  & \leq  e^{\kappa z} \E_{\mathbf{Q}^{*}_{\tau^{w}(z)}\otimes \P} \left( \sum_{v \in \mathcal{B}^{\leq \ell}(w)}  \sum_{v' \in \mathcal{N}^{(v)}_{r} }   \frac{\L^{(v')}_{t-b_{v}-r}(xt- X_{b_{v}+r}(v'))^2 }{\L_{t} (xt)^2 }   \ind{G_{\ell,L}^{w}}    \right) +    e^{I(x,a)t} \mathbf{P} (G_{\ell,L}^{c} ). 
\end{align}  where in the last step we have applied  \eqref{eq-change-of-measure}. By use of the branching property, we have  
\begin{equation}
   \E_{\mathbf{Q}^{*}_{\tau^{w}(z)} \otimes \P} \left(   \sum_{v' \in \mathcal{N}^{(v)}_{r} }   \frac{\L^{(v')}_{t-b_{v}-r}(xt- X_{b_{v}+r}(v'))^2 }{\L_{t} (xt)^2 } \mid   \mathcal{G}^{*}_{\tau^w(z)} \right)  \leq   \mathtt{OL}_{t-b_{v}} \left( r, \frac{xt- X_{b_{v}}(w)}{t-b_{v}} \right). 
\end{equation}
 We have showed in \eqref{eq-Y-b-w}  that, on the event $G^{w}_{\ell,L}$, there holds $ \frac{xt- X_{b_{v}}(w)}{t-b_{v}}=  \theta + \frac{Y_{b_{v}}(w)}{t-b_{v}}$ with $|Y_{b_{v}}(w) | \lesssim L\sqrt{s}$. Thus we conclude that 
\begin{align}
& \P\left( \begin{array}{l} 
  \exists v^{i} \in \mathcal{B}^{\leq \ell}(\hat{u}) \text{ s.t. } v^{i} \prec u^{i}_{t}, \  \forall i \\
  \mathscr{R}(u^{1}_{t},u^{2}_{t})  \geq \tau(z) -\ell + 2 r
    \end{array} 
\mid \L_{t}(xt) \geq \frac{1}{\sqrt{t}} e^{at}, \mathbf{I}  \in  \Gamma[t,z]  \right)\\
& \lesssim_{x,a}   L e^{R}  \times \sup_{|s'-s|\leq L \sqrt{s}} \sup_{|\delta|\leq t^{-1/3}}
\mathtt{OL}_{t-s'} \left( r, \theta+ \delta \right)  +    e^{I(x,a)t} \mathbf{P} (G_{\ell,L}^{c} ). 
\end{align}
Letting $t \to \infty$ first, then $r \to \infty$ and finally $L \to \infty$, and applying Lemmas \ref{lem-overlap} and   \ref{lem-good-event} we obtain the desired result in \eqref{eq-genealogy-2}.  
Combining  \eqref{eq-genealogy-2} and \eqref{eq-genealogy-1}, we  conclude \eqref{eq-genealogy-0}.

\underline{Step 2.} We now establish \eqref{eq-overleap-tight}. Fix $R>0$, $|z| \leq R$. Combining \eqref{eq-apro-W-limit} with Lemma \ref{lem-truncated-level-set} (iii), we deduce that uniformly in $|z| \leq R$, 
\begin{equation}\label{eq-sstar-tauz-0}
 \limsup_{t \to \infty} \P \left(  |\s - \tau(z)| > K \mid\L_{t}(xt) \geq \frac{1}{\sqrt{t}} e^{at}, \mathbf{I}  \in  \Gamma[t,z]  \right)   \overset{K \to \infty}{\longrightarrow} 0. 
\end{equation}
Using \eqref{eq-genealogy-0} and the inequality
$  \left|  \mathscr{R}(u^{1}_{t},u^{2}_{t}) - \s \right| \leq   \left|  \mathscr{R}(u^{1}_{t},u^{2}_{t}) - \tau(z) \right|+  \left| \tau(z)- \s \right|$,  we obtain  that   uniformly in $|z| \leq R $, 
\begin{equation}\label{eq-overlap-argmin-1}
 \lim_{K \to \infty}\limsup_{t \to \infty} \P \left(  \left|  \mathscr{R}(u^{1}_{t},u^{2}_{t}) - \s \right| >  K \mid\L_{t}(xt) \geq \frac{1}{\sqrt{t}} e^{at}, \mathbf{I}  \in  \Gamma[t,z]  \right) = 0.
\end{equation}
Additionally,  
thanks to Proposition \ref{prop-up-to-constant}, we have 
\begin{equation}\label{eq-level-set-vs-inf}
  \lim_{R \to \infty}\limsup_{t \to \infty} \P \left(  \left| \mathbf{I} + \psi'(\kappa)s \right| > R   \mid \L_{t}(xt) \geq \frac{1}{\sqrt{t}} e^{at}  \right) = 0.
\end{equation}
Combining \eqref{eq-level-set-vs-inf} with \eqref{eq-overlap-argmin-1}, we  deduce that the conditioned law  $ ( \mathscr{R}(u^{1}_{t},u^{2}_{t}) - \s \mid   \L_{t}(xt) \geq \frac{1}{\sqrt{t}} e^{at} )_{t>0}$ is tight.  
Furthermore,  by using  \eqref{eq-level-set-vs-inf},  \eqref{eq-apro-W-limit} and similar reasoning, 
it remains to  show that   uniformly in $z \in [-R,R]$,
\begin{equation} 
  \limsup_{t \to \infty} e^{I(x,a)t} \P \left(    \left| X_{ \mathscr{R}(u^{1}_{t}, u^{2}_{t})}(u^{1}_{t}) -  \mathsf{F}_{t}( \mathscr{R}(u^{1}_{t}, u^{2}_{t})) \right|   > K ,   \mathbf{I}  \in  \Gamma[t,z]  \right)    \overset{K \to \infty}{\longrightarrow} 0. 
\end{equation}
Define $E^{2}_{\ell,r}  = \{ \exists v^{i} \in \mathcal{B}^{\leq \ell}(\hat{u}) \text{ s.t. }   v^{i} \prec u^{i}_{t} , \mathscr{R}(u^{1}_{t},u^{2}_{t}) \leq \tau(z) -\ell + r\}$. By  \eqref{eq-genealogy-1},   \eqref{eq-genealogy-2} and  Lemma \ref{lem-good-event}, for any $\delta>0$, there are large constants $\ell,r,L$ depending on $\delta$ and $R$ such that 
\begin{equation}
  \limsup_{t \to \infty} e^{I(x,a)t} \P \left(    (E^{2}_{\ell,r} \cap     G_{\ell,L})^{c}, \mathbf{I}  \in  \Gamma[t,z] \right) < \delta. 
\end{equation} 
By \eqref{eq-change-of-measure} and the fact that  $\hat{u}=\hat{w}$ on $I_0^{w} \cap G_{\ell,L}$, we have
\begin{equation}
  \begin{aligned} 
   &  \limsup_{t \to \infty} e^{I(x,a)t} \P \left(    \left| X_{ \mathscr{R}(u^{1}_{t}, u^{2}_{t})}(u^{1}_{t}) -  \mathsf{F}_{t}( \mathscr{R}(u^{1}_{t}, u^{2}_{t})) \right|   > K ,   \mathbf{I}  \in  \Gamma[t,z]   \right)  \\
    &\leq  \limsup_{t \to \infty}  e^{\kappa z} (\mathbf{Q}^{*}_{\tau^w(z)} \otimes \P)  \left(  \left| X_{ \mathscr{R}(u^{1}_{t}, u^{2}_{t})}(u^{1}_{t}) -  \mathsf{F}_{t}( \mathscr{R}(u^{1}_{t}, u^{2}_{t})) \right|   > K , E^{2}_{\ell,r}(w), G^{w}_{\ell,L}, I_0^{w} \right)   + \delta 
    \label{eq-desired-ineq-00}
   \end{aligned}
\end{equation}  
where $E^{2}_{\ell,r}(w) = \{ \exists v^{i} \in \mathcal{B}^{\leq \ell}(\hat{w}) \text{ s.t. }   v^{i} \prec u^{i}_{t} , \mathscr{R}(u^{1}_{t},u^{2}_{t}) \leq \tau^w(z) -\ell + r\}$. Notice that when $E^{2}_{\ell,r}(w) $ occurs,  we have 
\begin{equation}
 \left| X_{\mathscr{R}(u^{1}_{t} , u^{2}_{t} )}(u^{1}_{t})- X_{\tau^w(z)}(w) \right| \leq \sum_{v \in \mathcal{B}^{\leq \ell}(\hat{w})} M^{(v)}_{r} + | X_{b_{v}}(w)- X_{\tau^w(z)}(w)|.  \label{eq-X-R-F-122}
\end{equation}
From  definition  $\theta X_{\tau^w(z)}(w) = \left( 1+ \frac{\theta^2}{2} \right) \tau^w(z) - \left( 1- \frac{\theta^2}{2} \right)s - z$, it follows that  $X_{\tau^w(z)}(w) = \mathsf{L}_{t}(\tau^w(z)) - {z}/{\theta}$.  Hence on the event $E^{2}_{\ell,r}(w)\cap G^{w}_{\ell,L} $, we have 
\begin{align}
 &\left| X_{\tau^w(z)}(w) -     \mathsf{F}_{t}(\mathscr{R}(u^{1}_{t},u^{2}_{t}) ) \right| \\
 &  \leq  \frac{R}{\theta}+ \left| \mathsf{L}_{t}(\mathscr{R}(u^{1}_{t},u^{2}_{t}))-  \mathsf{F}_{t}(\mathscr{R}(u^{1}_{t},u^{2}_{t})) \right| +  \left| \mathsf{L}_{t}(\tau^w(z))  -  \mathsf{L}_{t}(\mathscr{R}(u^{1}_{t},u^{2}_{t}) ) \right|\\
 & \leq  \frac{R}{\theta} + \sup_{|y-s| \leq L \sqrt{t}} |\mathsf{L}_{t}(y)- \mathsf{F}_{t}(y)|  +   \frac{ | \tau^w(z)   -  \mathscr{R}(u^{1}_{t},u^{2}_{t})  |  }{\theta} \lesssim_{x,a} R+L^{2}+\ell+r . \label{eq-X-R-F-1}
\end{align}
Above, we have used the facts that $|\mathsf{L}_{t}(y)- \mathsf{F}_{t}(y)| \leq \mathsf{F}''(s) L^{2} t $ by Taylor's formula and $\mathsf{F}''(s)= \mathsf{f}''(p)/t$ by definition \eqref{AHScurve}. Finally, 
using \eqref{eq-X-R-F-1} and \eqref{eq-X-R-F-122}, we find that for  large  $K$, 
\begin{equation}
  (\mathbf{Q}^{*}_{\tau^w(z)} \otimes \P) \left(  \left| X_{ \mathscr{R}(u^{1}_{t}, u^{2}_{t})}(u^{1}_{t}) -  \mathsf{F}_{t}( \mathscr{R}(u^{1}_{t}, u^{2}_{t})) \right|   > K ,   E^{2}_{\ell,r}(w), G^{w}_{\ell,L}, I_0^{w} \right)   \leq     L \P  (  M_{r}  >  K/2  ). \label{eq-X-R-F-2}
\end{equation}
Letting $t \to \infty$ first, then $K \to \infty$ and finally $\delta \to 0$ in  \eqref{eq-desired-ineq-00}, we deduce the desired result.
The conclusion holds because for any fixed $r$, we have \( \lim_{K \to \infty} \P \left( M_{r} > K/2 \right) = 0 \). This completes the proof of \eqref{eq-overleap-tight}.

 \underline{Step 3.} 
 We now proceed to show \eqref{eq-overleap-conv}. We claim that  the conditional law  $  (  X_{\mathscr{R}(u^{1}_{t} , u^{2}_{t} )}(u^{1}_{t})- \mathsf{L}_{t}(\mathscr{R}(u^{1}_{t},u^{2}_{t})) \mid \L_{t}(xt) \geq \frac{1}{\sqrt{t}} e^{at}  )  $ is also tight. To verify this, simply replace the function  $\mathsf{F}_{t}$ in \eqref{eq-X-R-F-1} and \eqref{eq-X-R-F-2} with $\mathsf{L}_{t}$. Notice    the identity 
\begin{equation}
  \mathsf{L}_{t}(\mathscr{R}(u^{1}_{t},u^{2}_{t}))- b s  = \left( \frac{\theta}{2}+ \frac{1}{\theta} \right)  \left( \mathscr{R}(u^{1}_{t},u^{2}_{t}) - s  \right). 
\end{equation}
Thus we have only to study the conditional law $  (  \frac{\mathscr{R}(u^{1}_{t},u^{2}_{t})- s}{\sqrt{s}} \mid \L_{t}(xt) \geq \frac{1}{\sqrt{t}} e^{at} ) $. Given that \eqref{eq-overleap-tight} ensures the tightness of    $({\mathscr{R}(u^{1}_{t},u^{2}_{t})- \mathfrak{s}}  \mid \L_{t}(xt) \geq \frac{1}{\sqrt{t}} e^{at} ) $, the desired result \eqref{eq-overleap-conv} follows from the validity of the following assertion:
\begin{equation}\label{eq-cond-min-pt-clt}
   \left(  \sqrt{\frac{\psi'(\kappa)^2}{\psi''(\kappa)}}  \frac{\mathfrak{s}- s}{\sqrt{s}} \mid \L_{t}(xt) \geq \frac{1}{\sqrt{t}} e^{at} \right)  \Rightarrow \mathcal{G}:=   \sqrt{\frac{1-p}{1-p+\kappa p}} \mathcal{N} 
\end{equation} 
Here $\mathcal{N}$ is a standard Gaussian random variable. 
By using the same arguments in the proof of Theorem \ref{thm-LDP},     we  obtain that   
\begin{align} 
  & \lim_{t \to \infty}   \mathbf{P}  \left(  \sqrt{\frac{\psi'(\kappa)^2}{\psi''(\kappa)}}   \frac{\mathfrak{s}- s}{\sqrt{s}}   \in [A,B], \L_{t}(xt) \geq \frac{1}{\sqrt{t}} e^{at }  \right)\\
 & = \sum_{z\in \mathbb{Z}} \lim_{t \to \infty}  e^{I(x,a)t}\P\left(  \sqrt{\frac{\psi'(\kappa)^2}{\psi''(\kappa)}} \frac{\s-s}{  \sqrt{s}} \in [A,B],    W_{\infty}(\theta)/\mathcal{C}_{t} \geq c_{\theta} e^{\psi'(\kappa) s} ,  \mathbf{I}   \in  \Gamma[t,z]   \right)  
\end{align} 
Following the  computations that lead to \eqref{eq-apro-W-limit}, we find that the limit above equals 
\begin{align}
  \kappa c_{\mathbf{I}}  \sum_{z \in \mathbb{Z}}  \int_{e^{z}}^{e^{z+1}}   \P \left( {\mathcal{Z}}/{\mathcal{C}} \geq  c_{\theta}     \gamma , \mathcal{N} \in [A,B] \right)  \gamma^{\kappa-1} \dif \gamma =  \frac{   C_{W_{\infty}(\theta)} }{(c_{\theta} y )^{\kappa}} \E\left[ \mathcal{C}^{-\kappa}\ind{\mathcal{N} \in [A,B]}\right] ,   \label{eq-cond-sdis-clt}  
\end{align}
where recall that $\mathcal{N}$ is a standard Gaussian distribution and $\mathcal{C}=\exp(\frac{p}{2(1-p)}\mathcal{N}^2)$. Hence we conclude that $   \mathbf{P}  (  \sqrt{\frac{\psi'(\kappa)^2}{\psi''(\kappa)}}   \frac{\mathfrak{s}- s}{\sqrt{s}}   \in [A,B] \mid \L_{t}(xt) \geq \frac{1}{\sqrt{t}} e^{at }  ) = \E(\frac{\mathcal{C}^{-\kappa}}{\E [ \mathcal{C}^{-\kappa}  ]} \ind{\mathcal{N} \in [A,B]})$. This concludes the assertion \eqref{eq-cond-min-pt-clt} 
 and hence completes the proof of Theorem \ref{thm-CondOverlap}. 
\end{proof}

\begin{proof}[Proof of Theorem \ref{thm-CondMaxs}] Note that $ \mathsf{L}_{t} \left( \tau(z)  \right) + \sqrt{2}[t- \tau(z)] = vt +  \frac{(\theta- \sqrt{2})^2}{2 \theta} (\tau(z)-s) $.
  As a direct consequence of   \eqref{eq-sstar-tauz-0} and \eqref{eq-cond-min-pt-clt}   we have 
  \begin{equation}
   \mathbf{P}\left(   \frac{\mathsf{L}_{t} \left( \tau(z)  \right) + \sqrt{2}[t- \tau(z)] - vt}{\sqrt{s}} \in \cdot   \mid \L_{t}(xt) \geq \frac{1}{\sqrt{t}} e^{at}  , \mathbf{I}  \in  \Gamma[t,z]  \right) \Rightarrow   \frac{\sqrt{2}-\theta}{\sqrt{2}+ \theta}  \, \mathcal{G}
  \end{equation}
  as $t \to \infty$ for any fixed $z \in \mathbb{R}$. So it suffices to show that for each fixed $\epsilon>0$, 
\begin{equation}
  \lim_{t \to \infty}   \P \left(  \left| \frac{ M_{t}- \left[ \mathsf{L}_{t} \left( \tau(z)  \right)+ \sqrt{2}(t- \tau(z)) \right]}{\sqrt{s}}  \right| > \epsilon \mid  \L_{t}(xt) \geq \frac{1}{\sqrt{t}} e^{at} ,  \mathbf{I}  \in  \Gamma[t,z]   \right) =0. 
\end{equation}

Since $X_{\tau(z)}(\hat{u})= L_{t}(\tau(z))-  {z}/{\theta}$ when $\tau(z)<t$, 
by use of \eqref{eq-apro-W-limit}, 
Lemma \ref{lem-good-event} and \ref{lem-truncated-level-set} (iv),  and \eqref{eq-change-of-measure}, 
we have only to show that for each fixed $\ell \geq 1$ and $L\geq 1$
\begin{equation}
  \begin{aligned}
 &  \lim_{t \to \infty} e^{I(x,a)t} \P \left(  \left| \frac{ M_{t}- \left[ X_{\tau(z)}(\hat{u})+ \sqrt{2}(t- \tau(z)) \right]}{\sqrt{s}}  \right| > \epsilon , E^{3}_{\ell},  G_{\ell,L} ,  \mathbf{I}  \in  \Gamma[t,z]   \right)  \\
 & = \lim_{t \to \infty} e^{\kappa z} (\mathbf{Q}^{*}_{\tau^w(z)} \otimes \P) \left(  \left| \frac{ M_{t}- \left[ X_{\tau^{w}(z)}(w)+ \sqrt{2}(t- \tau^{w}(z)) \right]}{\sqrt{s}}  \right| > \epsilon ,E^{3}_{\ell}(w), G_{\ell, L}^{w}, I_0^{w} \right)  =0 \label{eq-max-clt-00}
  \end{aligned}   
\end{equation} 
  where   $E^{3}_{\ell}:= \{ \forall v \in \mathcal{B}^{>\ell}(\hat{u}) ,  X_{b_{v}}(\hat{u})+ M^{(v)}_{t-b_{v}}  \leq  X_{\tau_{t}(z)}(\hat{u}) + \sqrt{2}[t-\tau_{t}(z)]  \}$, and  $E^{3}_{\ell}(w):=  \{ \forall v \in \mathcal{B}^{>\ell}(w) ,  X_{b_{v}}(w)+ M^{(v)}_{t-b_{v}}  \leq  X_{\tau^{w}(z)}(w)+ \sqrt{2}[t-\tau^w(z)]  \}$.

Observing that $M_{t}  \geq X_{\tau^{w}(z)}(w) + M^{(\hat{w})}_{t-\tau^{w}(z) }$, by use of the branching property and  \eqref{eq-HS-maximum} we get that 
  \begin{align}
    & \lim_{t \to \infty}  (\mathbf{Q}^{*}_{\tau^w(z)} \otimes \P)  \left( M_{t} \leq    X_{\tau^{w}(z)}(w)+ \sqrt{2}(t- \tau^{w}(z))   - \epsilon \sqrt{s}     \mid \mathcal{G}^{*}_{\tau^w(z)} \right) \ind{ G_{\ell, L}^{w}} \\
    &\leq \lim_{t \to \infty}  \max_{|r -s| \leq L \sqrt{s}} \P( M_{t-r} \leq  \sqrt{2}(t-r)- \epsilon \sqrt{s} ) = 0.  \label{eq-max-clt-01}
  \end{align}  
Moreover, on the event $E^{3}_{\ell}(w)\cap G_{\ell,L}^{w}$, $ M_{t} \geq    X_{\tau^{w}(z)}(w)+ \sqrt{2}(t- \tau^{w}(z)) +  \epsilon \sqrt{s}  $ holds  means that  there exists $v \in \mathcal{B}^{\leq \ell}(w)$ such that $M^{(v)}_{t-b_{v}} \geq   \sqrt{2}(t- \tau^{w}(z)) +  \epsilon \sqrt{s}- L$. Thus 
  \begin{align}
    & \lim_{t \to \infty}  \mathbf{Q}^{*}_{\tau^w(z)} \left( M_{t} \geq    X_{\tau^{w}(z)}(w)+ \sqrt{2}(t- \tau^{w}(z))  + \epsilon \sqrt{s}     \mid \mathcal{G}^{*}_{\tau^w(z)} \right) \ind{ G_{\ell, L}^{w}} \\
    &\leq \lim_{t \to \infty} L  \max_{|r -s| \leq L \sqrt{s}} \P( M_{t-r-\ell} \geq  \sqrt{2}(t-r) + \epsilon \sqrt{s} - L ) = 0.  \label{eq-max-clt-02}
  \end{align} 
Applying the bounded convergence theorem,
  the desired result \eqref{eq-max-clt-00} then follows from \eqref{eq-max-clt-01} and  \eqref{eq-max-clt-02}. We now complete the proof.
\end{proof}


\appendix

\section{Proof of Lemmas in Section   \ref{ref-preliminary}}
  \label{app-A}

  \subsection{Proof of Lemma \ref{changeofp}} 
Assertion (i) follows directly from the description of the process under  $\mathbf{Q}^{\beta,*}_{\tau} \otimes \P$. To prove assertion (ii), note that  $\P= \P^{*}|_{\mathcal{F}_{\infty}}$ and 
\begin{align}
  & \frac{\dif (\mathbf{Q}^{\beta,*}_{\tau} \otimes \P)|_{\mathcal{F}_{\infty}} }{\dif \P^{*}|_{\mathcal{F}_{\infty} }}  = \mathbf{E}^{*} \left[  \frac{\dif (\mathbf{Q}^{\beta,*}_{\tau} \otimes \P)  }{\dif \P^{*}} \mid \mathcal{F}_{\infty}   \right] \\
  &= \sum_{u \in \mathcal{N}_{\tau}} e^{\beta X_{\tau}(u)-  (\frac{\beta^2}{2}+1)\tau } 2^{|u|} \mathbf{P}^{*} \left( w_{\tau}=u \mid \mathcal{F}_{\infty}   \right) =  \sum_{u \in \mathcal{N}_{\tau}} e^{\beta X_{\tau}(u)-  (\frac{\beta^2}{2}+1)\tau }   = W_{\tau}(\beta).
\end{align} 
Above we have used that $
  \P^{*}( w_{\tau}= u \mid \mathcal{F}_{\infty} ) = \ind{u \in \mathcal{N}_{\tau}} \frac{1}{2^{|u|}}$. 
Then  we get that  $     \dif (\Q^{\beta,{*}}_{\tau} \otimes \P ) |_{\mathcal{F}_{\infty}} =W_{\tau}(\beta) \dif \P $.   
For any $F \in \mathcal{F}_{\infty}$ and a label $u$, 
\begin{align}
&(\mathbf{Q}^{\beta,*}_{\tau} \otimes \P) \left(\left\{w_{\tau}=u\right\} \cap F \right)  = \mathbf{E}^{*} \left( \mathbf{1}_{\left\{w_{\tau}=u\right\}} \mathbf{1}_{\{F\}}   e^{\beta \Xi_{\tau}-  (\frac{\beta^2}{2}+1)\tau } 2^{|w_{\tau}| }  \right) \\
&= \mathbf{E}^{*} \left( \mathbf{1}_{\left\{w_{\tau}=u\right\}} \mathbf{1}_{\{F\}}   e^{\beta X_{\tau}(u)-  (\frac{\beta^2}{2}+1)\tau } 2^{|u| }  \right)  =\mathbf{E}\left[\mathbf{1}_{\{u \in \mathcal{N}_{\tau}\}} \mathbf{1}_{\{F\}}   e^{\beta X_{\tau}(u)-  (\frac{\beta^2}{2}+1)\tau }  \right] \\
& =\E_{\mathbb{Q}^{\beta,*}_{\tau} \otimes \P} \left[\mathbf{1}_{\{u \in \mathcal{N}_{\tau} \}} \mathbf{1}_{\{F\}} \frac{1}{W_{\tau}(\beta) }     e^{\beta X_{\tau}(u)-  (\frac{\beta^2}{2}+1)\tau }     \right].
\end{align}
This completes the proof.  

\subsection{Proof of Lemma \ref{martingale limit and minimum}} 
We only need to prove the continuity of $ C_{\eqref{martmin}}(\beta)$ for $ \beta \in (0, \sqrt{2})$  in assertion (ii). The remaining assertions  are addressed in \cite[Theorem 1.3, Lemmas 3.1, and 3.4]{CDM24}.

  As shown in  \cite[Lemma 3.1, (B.8), (B.14)]{CDM24}, inequality $\eqref{martmin} $ holds   provided that   $C_{\eqref{martmin}}(\beta) $ is greater than  $ C_o(\beta)$ defined  as follows.  
Define $\psi_{\beta}(\lambda):=\left( \frac{\beta^2}{2} \lambda -1\right)(\lambda-1)$.
 Let  $\lambda_o := \argmin \psi_{\beta}(\lambda) =\frac{1}{2}+\frac{1}{\beta^2}$  and note that $\psi_{\beta}(\lambda_o)=  -\frac{(2/\beta^2 -1)^2}{8/\beta^2}$. Then, with some absolute constant $K$, we set 
   \begin{equation}
   C_o(\beta):= K     \left( \frac{2}{\beta^2} \sum_{n ,k,j \geq 0} (n+1)^{\frac{2}{\beta^2}}  e^{t+n\psi_{\beta}(\lambda_o)} e^{ \psi_{\beta}(\lambda_o)  k }e^{ - \frac{  \beta^2 j }{4}   } \right)^{\frac{2}{\beta^2}}    \E [  W_{1}(\beta)^{\frac{2}{\beta^2}+\frac{1}{2}}      ]  \E  [ W_{\infty} (\beta)^{\frac{2}{\beta^2}-\frac{1}{2}}   ]   .
   \end{equation}  
 The dominated convergence theorem yields that 
   the series in the parentheses is a continuous function of  $\beta \in (0,\sqrt{2})$. 
   When $p \leq 1$,  $ \E  [ W_{\infty} (\beta)^{p}   ] \leq   \E  [ W_{\infty} (\beta)]=1$. When $\frac{2}{\beta^2}-\frac{1}{2} \in (1,2)$, it follows from \cite[Lemmas 3.2 and 4.2]{Liu00}  that,   
   \begin{equation}
     \E \left[ W_{\infty}(\beta)^{p}  \right] \leq \frac{ \E[W_{1}(\beta)^{p}] } {(1- e^{\psi_{\beta}(p)/(p-1)})^{p-1}}  \text{ for } p \in (1, \frac{2}{\beta^2})
   \end{equation}
   and if $2/\beta^2>2$,  we have 
   \begin{equation}
      \E \left[ W_{\infty}(\beta)^{p}  \right]    \leq \frac{ \E[W_{1}(\beta)^{p}] \E[ W_{\infty}(\beta)^{p-1}]}{(1- e^{\psi_{\beta}(p)})^{p-1}} \text{ for } p \in (2, \frac{2}{\beta^2}). 
   \end{equation}
   Using the many-to-one lemma, we can see that $\beta \mapsto \E[W_{1}(\beta)^{f(\beta)}]$ is bounded by a continuous function on $(0, \sqrt{2})$, provided $f:(0, \sqrt{2}) \to \mathbb{R}{+}$ is continuous. Hence, there exists a continuous function $C_{\eqref{martmin}}(\beta)$ on $(0, \sqrt{2})$ such that $C_{\eqref{martmin}}(\beta) \geq C_0(\beta)$. This completes the proof.

\subsection{Proof of Lemma \ref{lem-overlap}} 
 For simplicity we write $\hat{\beta}=\beta+\delta$.   Fix $r>0$. Notice that  $ \mathscr{R}(u^{1}_{t} , u^{2}_{t})  \geq r$ holds if and only if  $u^{1}_{t}$ and $u^{2}_{t}$ share a common   ancestor at time $r$. Thus   conditioned on $\mathcal{F}_{r}$, we have 
  \begin{equation}
    \P( \mathscr{R}(u^{1}_{t} , u^{2}_{t})  \geq r \mid \mathcal{F}_{r}) = \E \left(   \frac{ \sum_{v \in \mathcal{N}_{r} } \left[ \L^{(v)}_{t-r} (\hat{\beta} t-X_{r}(v)) \right]^2  }{ \left[ \sum_{v \in \mathcal{N}_{r} } \L^{(v)}_{t-r} (\hat{\beta} t-X_{r}(v))  \right]^2 }  \mid \mathcal{F}_{r}  \right).
  \end{equation} 
 Since $ \E  [ \L_{t-r}(\hat{\beta} t- X_{r}(v)) \mid \mathcal{F}_{r} ] / \E  [ \L_{t-r}(\hat{\beta} t)] = (1+o(1))  e^{\beta X_{r}(v)} $ as $t \to \infty$, by use of Lemma \ref{lem-enhanced-convergence} and the branching property,  we obtain that   
  \begin{equation}
   \lim_{t \to \infty }   \sup_{\delta \leq \delta_{t}} \left| \frac{ \L^{(v)}_{t-r} (\hat{\beta} t-X_{r}(v)) }{ \E \left[ \L_{t-r}(\hat{\beta} t) \right] }  -  e^{\beta X_{r}(v)} W^{(v)}_{\infty}(\beta) \right|=0, \ \forall\,  v \in \mathcal{N}_{r} , \ \P(\cdot | \mathcal{F}_{r}) \text{-a.s. } 
  \end{equation}
Recall that   $ W_{\infty}(\beta)= \sum_{v \in \mathcal{N}_{r}} e^{\beta X_{r}(v) - (\frac{\beta^2}{2}+1) r} W_{\infty}^{(v)}(\beta)  $. Now,  employing the bounded convergence theorem, we deduce that 
  \begin{equation}
   \lim_{t \to \infty}   \sup_{\delta \leq \delta_{t}}   \left|  \P( \mathscr{R}(u^{1}_{t} , u^{2}_{t})  \geq r | \mathcal{F}_{r}) - \E\left[ \sum_{v \in \mathcal{N}_{r} } \left( e^{\beta X_{r}(v) - (\frac{\beta^2}{2}+1) r} \frac{ W_{\infty}^{(v)}(\beta) }{W_{\infty}(\beta)} \right)^2 \big| \, \mathcal{F}_{r}  \right]  \right| = 0 \  \text{ a.s. } 
  \end{equation}
 The desired result \eqref{eq-ol-conv-1} then follows by using the bounded convergence theorem again. 

It remains to show that $   \mathtt{OL}(r; \beta) \to 0$ as $r \to \infty$.  
Fix an $\epsilon >0$. We say an individual $u \in \mathcal{N}_{t}$ with position $X_{t}(u) \geq \beta t$ is good,  if for any $t' \in [r,t]$,   $X_{t'}(u) \leq (\beta +\epsilon)t'$.  Let $\bar{\L}_{t,r}(\beta t)$ denote the number of good individuals at time $t$. And for each $ v \in \mathcal{N}_{r}$, let $\bar{\L}_{t,r}^{(v)}(\beta t) $ denote the number of descendants of $v$ at time $t$ that is good.  Then 
\begin{align}
& \P(  u^{i}_{t}  \text{ is not  good}\mid \mathcal{F}_{r}) =  \frac{\L_{t}(\beta t) - \bar{\L}^{(v)}_{r,t}(\beta t)  }{\L_{t}(\beta t) } , i=1,2 \quad \text{and } \\
& \P( \mathscr{R}(u^{1}_{t} , u^{2}_{t})  \geq r , u^{i}_{t}  \text{ is good}, i=1,2 \mid \mathcal{F}_{r})  \leq \sum_{v \in \mathcal{N}_{r}}  \frac{ \bar{\L}^{(v)}_{r,t}(\beta t)^2 }{\L_{t}(\beta t)^2 } . 
\end{align}
Thus we have, for any $\lambda \in (0,1)$
\begin{align}
 \P( \mathscr{R}(u^{1}_{t} , u^{2}_{t})  \geq r  )   
 & \leq  \P \left(   \L_{t}(\beta t) \leq \lambda  \E[  \L_{t}(\beta t)  ] \right)  +  2  \frac{\E [\L_{t}(\beta t) - \bar{\L}^{(v)}_{r,t}(\beta t) ] }{ \lambda \E[\L_{t}(\beta t) ]} \\
  & \quad + 
\E \left(  \sum_{v \in \mathcal{N}_{r}}  \frac{ \bar{\L}^{(v)}_{r,t}(\beta t)^2 }{ \lambda^2 [\E\L_{t}(\beta t)]^2 } \right) . \label{eq-overplap-scaling-1}
\end{align}
It then follows from \cite[Lemmas 2.3 and 2.4]{GKS18} respectively that for large $r$  and large $t$ depending on $r$,
\begin{equation}
\frac{\E [\L_{t}(\beta t) - \bar{\L}^{(v)}_{r,t}(\beta t) ] }{   \E[\L_{t}(\beta t) ]}  \leq \exp \left( - \frac{\epsilon^2}{4} r \right) \text{ and }    \E \left(  \sum_{v \in \mathcal{N}_{r}}  \frac{ \bar{\L}^{(v)}_{r,t}(\beta t)^2 }{  [\E\L_{t}(\beta t)]^2 } \right)  \leq \exp \left(-\frac{1-2 \beta^2}{4} r \right) .
\end{equation}
Applying these two inequalities along with \eqref{eq-typical-GKS}   to the inequality  \eqref{eq-overplap-scaling-1} yields that 
\begin{equation}
\limsup_{t \to \infty }  \P( \mathscr{R}(u^{1}_{t} , u^{2}_{t})  \geq r  )  \leq \P \left(  W_{\infty}(\beta) \leq \lambda \right) + \frac{1}{\lambda} e^{ - \frac{\epsilon^2}{4} r } + \frac{1}{\lambda^2} e^{-\frac{1-2 \beta^2}{4} r }.
\end{equation}
Letting $r \to \infty$ first and then $\lambda \downarrow 0$, since $\P\left( W_{\infty}(\beta) =0\right)=0$,
the desired result follows.


\section*{Acknowledgement}
The authors would like to thank Prof. Louigi Addario-Berry for an inspiring discussion. The research is supported by the National Key R\&D program of China No. 2022YFA1006500.


 \bibliographystyle{alpha}
 \bibliography{biblio}
 \end{document}